\documentclass[11pt]{article}

\pdfoutput=1 
\def\comments{1}

\usepackage{preamble}

\renewcommand{\citet}{\cite}
\renewcommand{\citep}{\cite}

\begin{document}


\title{A Bias-Accuracy-Privacy Trilemma for Statistical Estimation\footnote{Authors are ordered alphabetically.}}
\author{
    Gautam Kamath\thanks{\texttt{g@csail.mit.edu}. University of Waterloo, Waterloo ON, Canada.}\hspace{0.25in}
    \and
    Argyris Mouzakis\thanks{\texttt{amouzaki@uwaterloo.ca}. University of Waterloo, Waterloo ON, Canada.}\hspace{0.25in}
    \and
    Matthew Regehr\thanks{\texttt{matt.regehr@uwaterloo.ca}.
    University of Waterloo, Waterloo ON, Canada.}
    \and
    Vikrant Singhal\thanks{\texttt{vikrant.singhal@uwaterloo.ca}. OpenDP, Cambridge MA, USA. Work done while a postdoc at the University of Waterloo.}\hspace{0.25in}
    \and
    Thomas Steinke\thanks{\texttt{bias@thomas-steinke.net}. Google Research, Mountain View CA, USA.}\hspace{0.25in}
    \and
    Jonathan Ullman\thanks{\texttt{jullman@ccs.neu.edu}. Northeastern University, Boston MA, USA.}
}
\maketitle

\begin{abstract}
    Differential privacy (DP) is a rigorous notion of data privacy, used for private statistics. The canonical algorithm for differentially private mean estimation is to first clip the samples to a bounded range and then add noise to their empirical mean.
    Clipping controls the sensitivity and, hence, the variance of the noise that we add for privacy. But clipping also introduces statistical bias.
    This tradeoff is inherent: we prove that no algorithm can simultaneously have low bias, low error, and low privacy loss for arbitrary distributions.

    Additionally, we show that under strong notions of DP (i.e., pure or concentrated DP), unbiased mean estimation is impossible, even if we assume that the data is sampled from a Gaussian. On the positive side, we show that unbiased mean estimation is possible under a more permissive notion of differential privacy (approximate DP) if we assume that the distribution is symmetric.
\end{abstract}


\clearpage

\tableofcontents

\clearpage

\section{Introduction}

While the goal of statistical inference and machine learning is to learn about a population, most statistical and learning algorithms reveal information specific to their sample, raising concerns about the \emph{privacy} of the individuals contributing their data.  \emph{Differential privacy} (DP)~\citep{DworkMNS06} has emerged as the standard framework for addressing privacy concerns. Informally, a differentially private algorithm guarantees that no attacker can infer much more about an individual in the sample than they could have inferred if that person's data was never collected.  There is a rich literature of differentially private algorithms for various statistical inference and machine learning tasks, many of which are now deployed.

Adding the constraint of differential privacy to a statistical inference or machine learning task can, and often does, incur an inherent cost \citep{BunUV14,DworkSSUV15,KarwaV18,KamathLSU19}, and there has been a large body of work pinning down these costs for a variety of tasks. The costs are typically studied via the two-way tradeoff between privacy and error, as measured by some loss function.  However, in many applications, we have multiple desiderata for the estimator, not all of which can be captured by a single loss function.

In this work, we study the \emph{statistical bias} of differentially private mean estimators, which adds an extra dimension to the tradeoff betweeen privacy and error. More precisely, given $n$ independent samples $X_1, \dots, X_n \in \mathbb{R}$ from an unknown univariate distribution $P$, we estimate the mean $\mu(P) \coloneqq \mathbb{E}_{X \gets P}[X]$, subject to the constraint that the estimator $\hat\mu$ be differentially private. Without a privacy constraint, the empirical mean $\bar{X} = \frac{1}{n} \sum_{i=1}^{n} X_i$ is both unbiased and provides optimal error bounds for all the settings we consider. Research on private mean estimation has also pinned down the optimal mean squared error (MSE) $\mathbb{E}[(\hat\mu(X) - \mu(P))^2]$ for a variety of families of distributions, such as sub-Gaussian distributions~\citep{KarwaV18,BunS19} and distributions satisfying bounded moment conditions~\citep{BarberD14,KamathSU20}. Unfortunately, these estimators can be very biased. Estimators with little or no bias are desirable because error due to variance can be averaged out by combining multiple estimates, whereas error due to bias can be difficult to eliminate.


\subsection{Differential Privacy}
A \emph{dataset} $x = (x_1, \dots, x_n) \in \cX^n$ consists of elements from a \emph{data universe} $\cX$. Two datasets $x, x' \in \cX^n$ are \emph{neighboring} (denoted $x \sim x'$) if the Hamming distance between them is $1$.

\begin{defn}[Differential Privacy (DP)~\citep{DworkMNS06}]\label{defn:dp}
    For $\eps, \delta \geq 0$, a randomized algorithm $A: \cX^n \rightarrow \cY$ satisfies $(\eps,\delta)$-\emph{differential privacy} ($(\eps,\delta)$-DP) if, for every neighboring pair of datasets $x \sim x' \in \cX^n$ and every measurable $Y \subseteq \cY$,
    $$\pr{}{A(x) \in Y} \leq \exp(\eps) \pr{}{A(x') \in Y} + \delta.$$
    This property is called \emph{pure DP} (or $\eps$-DP) when $\delta = 0$, and \emph{approximate DP} when $\delta > 0$.
\end{defn}

\subsection{Contributions}

Our main contribution is to show that privacy inherently leads to statistical bias, by establishing a \emph{trilemma} between bias, accuracy, and privacy for the fundamental task of \emph{mean estimation}. We present the result informally below as Theorem~\ref{thm:main_trilemma_informal}. See Section~\ref{sec:trilemma} for proofs and a more detailed discussion as well as Theorem~\ref{thm:main_trilemma_formal} for a more precise statement of the result. In Subsection~\ref{subsec:amplification}, we present an alternative proof technique based on amplification by shuffling, which may be of independent interest, although it yields a slightly weaker trilemma.

Furthermore, Theorem~\ref{thm:eps_delta_ub_mix} below asserts that this lower bound is qualitatively tight. More details and proofs are given in Section~\ref{sec:low_bias}.

We also identify \emph{asymmetry} as the primary cause of bias in private mean estimation by constructing unbiased private estimators for symmetric distributions. An informal statement is given in Theorem~\ref{thm:positive_unbiased}. Section~\ref{sec:symmetric} has the details. We also present an empirical comparison of our bias-corrected estimator to a well-known biased alternative in Subsection~\ref{subsec:kv_vs_unbiased}.

Our final main result is that unbiased mean estimation is impossible under pure differential privacy. A simplified version of the result is given below in Theorem~\ref{thm:main_packing}. Section~\ref{sec:analyticity} contains more details and discussion of the more general result for exponential families, given in Theorem~\ref{thm:unbiased_ef_lb}. Mathematical background is given in Appendix~\ref{app:complex_analysis} and Appendix~\ref{app:measure_theory}. In addition, we present an extension of this result to concentrated DP in Appendix~\ref{app:cdp}.

Implicitly, all of our results apply to datasets of fixed, publicly known size. An extension of our results to setting where the dataset size is itself private can be found in Appendix~\ref{app:n}.

\subsection{Overview of Results}

There are a variety of methods for private mean estimation, all of which introduce bias. To understand the source of bias, it is useful to review one common approach -- the \emph{noisy clipped mean} $M(X)$, which we define as follows. First, it clips the samples to some bounded range $[a, b]$, defined by
\begin{equation*}
    \clip_{[a,b]}(x) \coloneqq
        \min\{\max\{x, a\}, b\}.
\end{equation*}
Next, it computes the empirical mean of the clipped samples $\hat\mu_{[a,b]}(X) \coloneqq \frac{1}{n} \sum_{i=1}^{n} \clip_{[a,b]}(X_i)$.  Finally, it perturbs the clipped mean with random noise whose variance is calibrated to the \emph{sensitivity} of the clipped mean -- i.e., the width of the clipping interval.  Specifically, to ensure $\eps$-differential privacy ($\eps$-DP), we have
\begin{equation*}
    M(X) := \frac{1}{n}\sum_{i=1}^{n} \clip_{[a,b]}(X_i) + \Lap\left(\frac{b-a}{\eps n}\right),
\end{equation*}
where $\Lap$ denotes the Laplace distribution, which has mean $0$ and variance $\frac{2(b-a)^2}{\eps^2 n^2}$.

Since the Laplace distribution has mean $0$, we have $\mathbb{E}[M(X)] = \mathbb{E}[\clip_{[a,b]}(X_i)]$, so the only step that can introduce bias is the clipping.  If we choose a large enough interval so that the support of the distribution $P$ is contained in $[a,b]$, then clipping has no effect, and the estimator is unbiased.  However, in this case, $[a,b]$ might have to be very wide, resulting in $M(X)$ having large variance.  On the other hand, if we reduce the variance by choosing a small interval $[a,b]$, then we will have $\mathbb{E}[\clip_{[a,b]}(X_i)] \neq \mathbb{E}[X_i]$ and the estimator will be biased.  Thus, we are faced with a non-trivial bias-accuracy-privacy tradeoff.  The exact form of the bias and the accuracy depends on assumptions we make about $P$.  In particular, if we consider the class of distributions $P$ with bounded variance $\exin{X \gets P}{(X - \mu(P))^2} \leq 1$, then for any $\bias > 0$, one can instantiate \citep{KamathSU20} the noisy-clipped-mean estimator $M$ with an appropriate interval so that it satisfies $\eps$-DP, has bias at most $\bias$, and has MSE
\begin{equation} \label{eq:intro_mse_ub}
\ex{X \gets P^n, M}{(M(X)-\mu(P))^2} \le O\left( \frac{1}{n} + \bias^2 + \frac{1}{n^2 \cdot \eps^2 \cdot \bias^2} \right).
\end{equation}
We show that no private estimator with bias bounded by $\beta$ can achieve a smaller MSE.

We do so by proving a lower bound on the MSE of any differentially private estimator for the mean of an arbitrary bounded-variance distribution. 

\begin{thm}[Bias-Accuracy-Privacy Trilemma]\label{thm:main_trilemma_informal}
    Let $M : \mathbb{R}^n \to \mathbb{R}$ be an $(\eps,\delta)$-DP algorithm, for some $\eps,\delta$ satisfying\footnote{This assumption is natural because $\epsilon = \Theta(1)$ and $\delta = o(1/n)$ is the standard DP parameter regime.} $0<\delta \leq \eps^2 / 200 \leq 1$.  Suppose $M$ satisfies the following bounds on its bias $\bias$ and MSE $\rmse^2$: for every distribution\footnote{We write $X \gets P^n$ to mean that $X$ is a sequence of $n$ i.i.d. samples from a distribution $P$.} $P$ with $\ex{X \gets P}{X} = \mu \in [0,1]$ and\footnote{Since this theorem proves a lower bound, restricting the mean to $[0, 1]$ only \emph{strengthens} the result. In particular, even if our estimator is provided a coarse estimate of the mean, we still face the same bias-accuracy-privacy tradeoff. It is common to consider coarse and fine private mean estimation separately.} $\exin{X \gets P}{(X - \mu)^2} \leq 1$,
    \[
        \left|\ex{X \gets P^n, M}{M(X) - \mu} \right| \leq \beta \leq \frac{1}{100}
        ~~~\text{ and }~~~
        \ex{X \gets P^n, M}{(M(X) - \mu)^2} \leq \rmse^2.
    \]
    Then 
    \begin{equation} \label{eq:intro_mse_lb}
        \rmse^2 \ge \Omega\left(\min\left\{\frac{1}{n^2 \cdot \eps^2 \cdot \bias^2}~,~\frac{1}{ n^2 \cdot \eps \cdot \delta^{1/2}}\right\}\right).
    \end{equation}
\end{thm}

To interpret the lower bound in Theorem~\ref{thm:main_trilemma_informal} and compare it to the upper bound \eqref{eq:intro_mse_ub}, we consider two parameter regimes. First assume $\delta$ is small -- specifically, $\delta \ll \bias^4 \eps^2$ -- so that the first term in the minimum dominates and the bound simplifies to $\rmse^2 \ge \Omega(1/n^2 \eps^2 \bias^2)$. This case corresponds to applications where we require strong privacy, but are less strict about bias. In any case, note that $(\eps,\delta)$-DP is only a meaningful privacy constraint when $\delta$ is quite small (see, e.g.,~\cite{KasiviswanathanS14}).
Observe that the upper bound \eqref{eq:intro_mse_ub} has two other terms, which are not reflected in Theorem~\ref{thm:main_trilemma_informal}'s lower bound \eqref{eq:intro_mse_lb}. These terms are also inherent, but for reasons unrelated to the privacy constraint.
First, we also know that $\rmse^2 \ge \Omega(1/n)$, which is a lower bound on the MSE of any mean estimator, even those that are not private (such as the unperturbed empirical mean).  Second,
by the standard bias-variance decomposition of MSE,
we have that $\rmse(P)^2 \geq \bias(P)^2$ for each distribution $P$, where $\bias(P)$ and $\rmse(P)^2$ denote, respectively, the bias and MSE of the estimator on that distribution $P$.  Thus, if we set\footnote{Theorem \ref{thm:main_trilemma_informal} permits us to set $\bias \gg \sup_P \bias(P)$. Thus, we cannot conclude $\rmse^2 \ge \bias^2$ in the theorem.} $\bias = \sup_{P} \bias(P)$ and combine the three lower bounds, we conclude that
\begin{equation}
    \rmse^2 \ge \Omega\left( \underbrace{~~~~~\frac{1}{n}~~~~~}_{\text{non-private error}} + \underbrace{~~~~~\bias^2~~~~~}_{\text{MSE $\ge$ bias$^2$}} + \underbrace{~~~\frac{1}{n^2 \cdot \eps^2 \cdot \bias^2}~~~}_{\text{error due to privacy}\atop\text{(our result)}} \right),
\end{equation}
which matches the upper bound \eqref{eq:intro_mse_ub} up to constant factors.

However, there is also the case where $\delta \gg \bias^4 \eps^2$, which corresponds to having  strict requirements for low bias. In this case, the second term in the minimum of Theorem~\ref{thm:main_trilemma_informal} dominates and we conclude $\rmse^2 \ge \Omega(1/n^2\varepsilon^2\delta^{1/2})$.
Our trilemma is close to tight in this setting, although there are some caveats, which we now state.

\begin{thm}[Tightness of Bias-Accuracy-Privacy Trilemma]\label{thm:eps_delta_ub_mix}
    For all $\eps,\delta,\beta>0$ and $n \in \mathbb{N}$, there exists an $(\eps,\delta)$-DP algorithm $M : \mathbb{R}^n \to \mathbb{R}$ satisfying the following bias and accuracy properties.
    For any distribution $P$ on $\mathbb{R}$ with $\ex{X \gets P}{X} = \mu \in [0,1]$, $\ex{X \gets P}{(X-\mu)^2}\le1$, and $\ex{X \gets P}{(X-\mu)^4}\le\psi^4$, we have
    $
        \left|\ex{X\gets P^n, M}{M(X) - \mu} \right| \le \bias
    $
    and
    \begin{align*}
        \ex{X\gets P^n, M}{(M(X)-\mu)^2} = O\left(\frac{1}{n} + \min\left\{\frac{1}{n^2\cdot \eps^2\cdot \bias^2} + \bias^2 , \frac{\psi^2}{n^{3/2} \cdot \eps \cdot \delta^{1/2}} + \frac{1}{n^2 \cdot \eps^2} , \frac{1}{n \cdot \delta} \right\}\right).
    \end{align*}
\end{thm}

Note the mild fourth moment bound assumption in Theorem \ref{thm:eps_delta_ub_mix}, which does not appear in Theorem \ref{thm:main_trilemma_informal}. This assumption can be relaxed to $\ex{X \gets P}{|X-\mu|^\lambda}\le\psi^\lambda$ for any $\lambda>2$ at the expense of weakening the result; see Proposition \ref{prop:eps_delta_ub}. 
One can also set $\psi = \infty$ and rely on the other two terms in the minimum. 
In any case, Theorem~\ref{thm:eps_delta_ub_mix} shows that our main result, Theorem \ref{thm:main_trilemma_informal}, is tight in the setting where $\delta$ is small relative to the bias, namely, $\delta \ll \beta^4 \epsilon^2$. In the complementary case where $\delta$ is large relative to the bias, our bounds are tight in all factors except $n$, which differs by $\sqrt{n}$. In particular, the dependence on $\delta^{-1/2}$ is necessary in the case where the bias $\beta$ is small.

The lower bound in Theorem~\ref{thm:main_trilemma_informal} applies to private mean estimators that are accurate for the entire class of distributions with bounded variance.  \emph{Can we obtain unbiased private estimators by making stronger assumptions on the distribution?}  Our subsequent results show that the answer depends on what assumptions we are willing to make.  Namely, we show that a generalization of the lower bound in Theorem~\ref{thm:main_trilemma_informal} holds even for the case of distributions with bounded higher moments, but we also show that unbiased private mean estimation is possible for symmetric distributions.

\mypar{Generalization to Higher Moment Bounds.} If the distribution $P$ has bounded interval support, then unbiased private mean estimation is possible as clipping to this support has no effect. 
More generally, if $P$ is more tightly concentrated, then we can clip more aggressively and obtain better bias-accuracy-privacy tradeoffs for the noisy clipped mean.

We consider the class of distributions that satisfy the stronger assumption $\exin{X \gets P}{|X - \mu|^\lambda} \leq 1$ for some $\lambda > 2$. 
For bias $\beta$ we can achieve MSE
\[
\ex{}{(M(X)-\mu(P))^2} \le O\left(\frac{1}{n} + \bias^2 + \frac{1}{n^2 \cdot \eps^2 \cdot \bias^{2/(\lambda-1)}}\right).
\]
Note that, although we can achieve a lower MSE for the same bias, this tradeoff is still qualitatively similar to the case of Theorem~\ref{thm:main_trilemma_informal} in that estimators with optimal MSE must have large bias.
We prove an analogue of Theorem~\ref{thm:main_trilemma_informal} showing that this tradeoff is tight for this class of distributions, for every $\lambda > 2$, and conclude that bias remains an essential feature of private estimation even under stronger concentration assumptions.

\mypar{Symmetric Distributions.}
Noisy clipped mean leads to bias because clipping to the interval $[a,b]$ might affect the distribution asymmetrically. Thus, it is natural to consider whether we can achieve unbiased private estimation when the distribution is \emph{symmetric} around its mean, which holds for many families of distributions like Gaussians.  If the distribution is symmetric and we could clip to an interval $[a,b]=[\mu - c, \mu + c]$, then the clipped mean would be unbiased, but this would require us to already know $\mu$. Nonetheless, we construct a private, unbiased mean estimator for any symmetric distribution.

\begin{thm}[Unbiased Private Mean Estimation for Symmetric Distributions]\label{thm:positive_unbiased}
For all $\varepsilon,\delta>0$, $\lambda > 2$, and $n \ge O(\log(1/\delta)/\eps)$ with $\delta \le 1/n$, there exists an $(\varepsilon,\delta)$-DP algorithm $M : \mathbb{R}^n \to \mathbb{R}$ satisfying the following: Let $P$ be a symmetric distribution on $\mathbb{R}$ -- i.e., there exists $\mu \in \mathbb{R}$ so that $X-\mu$ and $\mu-X$ are identically distributed -- satisfying $\ex{X \gets P}{|X-\mu|^\lambda} \le 1$.  If $X \gets P^n$, then $\exin{}{M(X)} = \mu$, and \[\ex{}{(M(X)-\mu)^2} \le O\left( \frac{1}{n} + \frac{1}{(n \cdot \eps)^{2-2/\lambda}} + \frac{\delta\cdot\mu^2}{n}\right).\]
\end{thm}

Note that the MSE in the theorem has a dependence on $\mu$, which can be unbounded.  However, if we assume that we know some $r$ such that $|\mu|\le r$, then we can remove this term by setting $\delta\le 1/O(nr^2)$.  
Furthermore, if the distribution $P$ is unit-variance Gaussian (or sub-Gaussian), then the central moments satisfy $\exin{X \gets P}{|X-\mu|^\lambda} \le O(\sqrt{\lambda}^\lambda)$ for all $\lambda$.
In particular, for the special case of unit-variance Gaussians with $O(1)$ mean, we can set $\lambda=\Theta(\log n)$ and the guarantee of our algorithm simplifies to
\[
    \mu^2 \le 1/\delta ~~~ \implies ~~~
    \ex{X \gets \cN(\mu,1)^n,M}{(M(X)-\mu)^2} \le O \left( \frac1n + \frac{\log n}{n^2 \cdot \eps^2} \right).
\]
This matches what is possible under the biased algorithm\footnote{See Theorem 1.1 of \url{https://arxiv.org/pdf/1711.03908}} of \citet{KarwaV18}. Both our and their algorithm are optimal up to polylogarithmic factors (see e.g. Theorem 5 of \citet{KamathLSU19}).

We note that, unlike the noisy clipped mean method, and many other methods for private mean estimation, the estimator of Theorem~\ref{thm:positive_unbiased} only satisfies $(\eps,\delta)$-DP for $\delta > 0$. This is fundamental to the techniques we use; our estimator cannot be made to satisfy $(\eps,0)$-DP.  We show that this is inherent by proving that every unbiased mean estimator even for restricted classes of distributions like Gaussians cannot satisfy $(\eps,0)$-DP.

\begin{thm}[Impossibilty of Unbiased Estimators under Pure DP]\label{thm:main_packing}
    Let $M : \R^n \to \R$ be a randomized algorithm.  Assume that $M$ satisfies the following guarantee for Gaussian data: there is a nonempty open interval $(a,b)$ such that for every $\mu \in (a,b)$, 
    \[
        \ex{X \gets \mathcal{N}(\mu,1)^n,M}{M(X)} = \mu~~~\text{and}~~~\ex{X \gets \mathcal{N}(\mu,1)^n, M}{|M(X)-\mu|} < \infty.
    \]
    Then $M$ does not satisfy $(\eps,0)$-DP for any $\eps<\infty$.
\end{thm}

We also extend this impossibility result beyond Gaussians to exponential families and from pure DP to concentrated DP \citep{DworkR16,BunS16}.

\subsection{Our Techniques}
\label{sec:techniques}
We provide two different methods for proving lower bounds on the MSE of low-bias private estimators, such as Theorem~\ref{thm:main_trilemma_informal} and its generalization to higher moments. While this is redundant, we hope that offering multiple perspectives can provide deeper insight.

\paragraph{Negative Results via the Fingerprinting Method (Theorem~\ref{thm:main_trilemma_informal}, \S\ref{sec:trilemma}).}
The fingerprinting method \citep{BunUV14,DworkSSUV15} (alternatively called ``tracing attacks'' or ``membership-inference attacks'') is the main approach for proving optimal lower bounds on the error of differentially private estimation.  Our proof of Theorem~\ref{thm:main_trilemma_informal} is based on a refinement of this method that separately accounts for the bias and mean squared error of the estimator, and thus allows for us to prove tradeoffs between these two parameters.

To give intuition for the argument, we consider the case where $M$ is an \emph{unbiased} estimator, in which case the argument is similar to the proof of the Cram\'er-Rao bound.  Assume that we have a suitable family of distributions $P_\mu$ with mean $\exin{X \gets P_\mu}{X} = \mu$ and an unbiased estimator $M$ such that $\exin{X \gets P_\mu^n}{M(X)}=\mu$.  As in the proof of the Cram\'er-Rao bound, we take the derivative of the unbiasedness constraint, which gives 
\[
    1 = \frac{\mathrm{d}}{\mathrm{d}\mu} \left[\ex{X \gets P_\mu^n, M}{M(X)}\right] = \sum_{i=1}^n  \ex{X \gets P_\mu^n, M}{M(X) \cdot \frac{\mathrm{d}}{\mathrm{d}\mu} \log P_\mu(X_i)},
\] 
where $P_\mu(x)$ denotes the probability mass or density function of $P_\mu$ evaluated at $x$.
The $(\eps,\delta)$-differential privacy guarantee of $M$ says that $M(X)$ and $X_i$ are close to being independent where $\eps$ and $\delta$ quantify the distance from independence. Moreover, a straightforward calculation shows that $\exin{X \gets P_\mu}{\frac{\mathrm{d}}{\mathrm{d}\mu} \log P_\mu(X)} = 0$.  Thus, for all $i \in [n]$, we have 
\[
    \ex{X \gets P_\mu^n, M}{M(X) \cdot \frac{\mathrm{d}}{\mathrm{d}\mu} \log P_\mu(X_i)} \approx_{\eps,\delta} \ex{X \gets P_\mu^n, M}{M(X)} \cdot \ex{X \gets P_\mu^n}{\frac{\mathrm{d}}{\mathrm{d}\mu} \log P_\mu(X_i)} =  0.
\]
Intuitively, this leads to the contradiction
\[
    1 = \sum_{i=1}^n  \ex{X \gets P_\mu^n, M}{M(X) \cdot \frac{\mathrm{d}}{\mathrm{d}\mu} \log P_\mu(X_i)} \approx_{\eps,\delta} \sum_{i=1}^n 0.
\]
To make this argument precise, we must exactly quantify the approximation $\approx_{\eps,\delta}$, which depends both on the privacy parameters $\eps$ and $\delta$, as well as on the variances of $M(X)$ and of $\smash{\frac{\mathrm{d}}{\mathrm{d}\mu} \log P_\mu(X_i)}$. The variance of $M(X)$ is the quantity that we are trying to bound. The variance of $\smash{\frac{\mathrm{d}}{\mathrm{d}\mu} \log P_\mu(X_i)}$ (which is known as the Fisher information) is something we control by choosing the distribution $P_\mu$ to be a distribution supported on two points.

The above proof sketch applies to the unbiased case  ($\bias = 0$). The general case ($\bias>0$) introduces some additional complications to the proof. In particular, we cannot simply consider a single fixed value of the mean parameter $\mu$, as we must rule out the pathological algorithm that ignores its input sample and outputs $\mu$, which has somehow been hardcoded into the algorithm.  This pathological algorithm trivially satisfies privacy and is unbiased for the single distribution $P_\mu$. To rule out this algorithm, we consider a distribution over the parameter $\mu$ and average over this distribution, where the distribution's support is wider than the allowable bias $\beta$. While we can no longer assume that
\[
    1 = \frac{\mathrm{d}}{\mathrm{d}\mu} \left[\ex{X \gets P_\mu^n, M}{M(X)}\right],
\]
we can still argue that the derivative must be $\ge\Omega(1)$ \emph{on average} over the choice of $\mu$.

\paragraph{Negative Results via Amplification (Theorem~\ref{thm:main_trilemma_informal} Revisited, \S\ref{subsec:amplification}).} 
We present an alternative approach for proving MSE lower bounds on low-bias private estimators.  While this approach gives slightly weaker bounds than Theorem~\ref{thm:main_trilemma_informal}, it exploits less structure of the problem, and thus may be easier to adapt to other settings. This method is a proof by contradiction: We start by assuming the existence of a private estimator and we show that running such an algorithm on independent datasets and averaging the results would violate existing lower bounds on the mean squared error of private mean estimators~\citep{KamathSU20}.

We start with an $(\eps,\delta)$-DP private estimator $M : \R^n \to \R$ that takes $n$ samples from some distribution and estimates its mean with bias $\bias$ and variance $\sigma^2$. Then we construct a new estimator $A_{m}$ that takes $nm$ samples, randomly splits them into $m$ blocks of $n$ samples each, runs $M$ on each block, and averages the outputs. This averaging won't reduce the bias but will reduce the variance by a factor of $m$. Thus, the MSE of $A_{m}$ is $\bias^2 + \sigma^2 / m$. Moreover, \emph{privacy amplification by shuffling} \cite{ErlingssonFMRTT19,CheuSUZZ19,balle2019privacy,FeldmanMT22,FeldmanMT23}
shows that $A_{m}$ is $(\eps', \delta')$-DP for $\eps' = \wt{O}( \eps \cdot m^{-1/2} )$. 
To complete the proof, we can apply \emph{any} lower bound on the MSE of the private estimator $A_m$.  In particular, if we consider the class of distributions with bounded variance, then we can use the lower bound of \cite{KamathSU20}, which shows that the MSE of $A_m$ is $\bias^2 + \sigma^2/m \ge \Omega(1/nm\eps') \ge \wt\Omega(1/nm^{1/2}\eps)$. Setting $m$ appropriately yields a lower bound on $\sigma$ that roughly matches Theorem~\ref{thm:main_trilemma_informal}.

\paragraph{General-Purpose Low-Bias Mean Estimation (Theorem~\ref{thm:eps_delta_ub_mix}, \S\ref{sec:low_bias}).}
The estimator we construct in the proof of Theorem~\ref{thm:eps_delta_ub_mix} combines two well-known techniques: the noisy-clipped-mean that we already discussed and the so-called name-and-shame algorithm.  As discussed above, the noisy-clipped-mean satisfies pure DP, but it leads to a tradeoff between bias and accuracy.

Name-and-shame is a pathological algorithm, which with probability $\delta n$, outputs one random sample in the dataset without any privacy protection, and otherwise outputs nothing.  This algorithm satisfies $(0,\delta)$-DP and can be used as the basis for an unbiased estimator. Specifically, given a sample $X \gets P^n$, with probability $\delta n$, output $X_I/\delta n$ for a uniformly random $I \in [n]$, and otherwise output $0$.  This estimator is unbiased, but the variance scales with $1/\delta n$, which is impractically large for reasonable values of $\delta$.

Intuitively, our estimator breaks the distribution into the tail portion (far from the mean) and the head portion (close to the mean), and uses the noisy-clipped mean on the head and then uses name-and-shame on the tail to correct the bias.  That is, we estimate $\frac1n \sum_{i=1}^n \clip_{[a,b]}(X_i)$ with noise addition. Then we estimate the tail $\frac1n \sum_{i=1}^n (X_i - \clip_{[a,b]}(X_i)))$ using name-and-shame. Combining these two estimates yields an unbiased estimator, such that the variance of name-and-shame is greatly reduced because $X_i - \clip_{[a,b]}(X_i) = 0$ with high probability.

\paragraph{Unbiased Mean Estimation for Symmetric Distributions (Theorem~\ref{thm:positive_unbiased}, \S\ref{sec:symmetric}).}
Our unbiased private estimator for symmetric distributions is based on the estimator of \cite{KarwaV18}, with some modifications to ensure unbiasedness, so we begin by reviewing the key ideas in their estimator.  Their estimator has three steps:  First, obtain a coarse estimate $\wt\mu$.  Second, use this coarse estimate to compute a clipped mean $\wh\mu \coloneqq \frac1n \sum_{i=1}^n \clip_{[\wt\mu-c,\wt\mu+c]}(X_i)$.  Finally, add noise to this fine estimate $\wh\mu$, so $M(X) = \wh\mu + \Lap\left(\frac{2c}{n\eps}\right)$.  

The coarse estimate only needs to satisfy a minimal accuracy guarantee; roughly, we require $|\wt\mu-\mu(P)| \le O(\sigma)$ with high probability, where $\sigma^2$ is the variance of $P$.  We can compute such an estimate using a histogram where we split the real line into intervals of length $O(\sigma)$ and pick an interval that contains many samples. With high probability, there is at least one interval containing many samples and the midpoint of any such interval is a good coarse estimate. We can privately select such an interval with high probability -- even though there are infinitely many intervals to choose from -- using standard techniques in DP~\citep{KorolovaKMN09,Vadhan17}.

Intuitively, we modify the algorithm of Karwa and Vadhan so that the symmetry of the distribution $P$ is preserved in our estimates $\wt\mu$ and $\wh\mu$ of the mean.  The only part of the algorithm that breaks the symmetry is when we split the real line into intervals. We fix this by adding a random offset to the intervals. That is, up to scaling, our intervals are of the form $\{[\ell + T,\ell+1 + T) : \ell \in \mathbb{Z}\}$, where $T \in [-1/2,+1/2)$ is a uniformly random offset. Note that the distribution of this set is equivariant under translation, so we can proceed with the analysis as if we had translated it to be symmetric around the unknown mean.  

The key observation for the analysis is that, if the coarse estimate $\wt\mu$ has a symmetric distribution with center $\mu$, then the clipping does not introduce bias. This is because the clipping is equally likely to introduce positive bias or negative bias and this averages out. There are two additional technicalities in the algorithm:  First, we must ensure that the coarse estimate $\wt\mu$ is independent by using independent samples for the two stages. Second, the coarse estimation procedure may fail to produce an estimate $\wt\mu$ because no interval contains enough samples.  In this case, we fall back on a version of the name-and-shame algorithm that is unbiased but has high variance $O(\mu^2/\delta)$.  Since the probabiltiy of needing this fallback will be $O(\delta^2)$, this case does not contribute much to the overall variance.

\paragraph{Negative Result for Pure DP Unbiased Mean Estimation (Theorem~\ref{thm:main_packing}, \S\ref{sec:analyticity}).}
Suppose, for the sake of contradiction, that $M$ is an $\eps$-DP unbiased estimator for $\mu \in (a,b)$ given samples from $\mathcal{N}(\mu,1)$.
Consider the function
\[
    g(\mu) \coloneqq \ex{X \gets \mathcal{N}(\mu,1)^n,M}{M(X)}.
\]

We first observe that, if $M$ is unbiased for $\mu \in (a,b)$, then it must be unbiased for all $\mu \in \R$. 
Intuitively, this holds because of the smoothing induced by the Gaussian distribution -- if $M(X)$ were biased when the data is drawn from $\mathcal{N}(\mu^*,1)$ for some $\mu^*$, then the same data could occur under any other distribution $\mathcal{N}(\mu,1)$ with some tiny but non-zero probability, and thus $M(X)$ would also be biased for $\mu \in (a,b)$, as well. 
Formally, we show $g$ is analytic. Thus, its value on the whole real line can be determined from the interval $(a,b)$.

Second we observe that, if $M(X)$ has bounded mean absolute deviation $\exin{X \gets \mathcal{N}(\mu^*,1)^n}{|M(X) - \mu^*|}$ for some $\mu^*$, then $\ex{}{|M(x)|}$ must be uniformly bounded for all inputs $x$.  Namely, \[|g(\mu)| \le \ex{X \gets \cN(\mu,1)^n}{|M(X)|} \le |\mu^*| + \ex{X \gets \cN(\mu,1)^n}{|M(X)-\mu^*|} \le |\mu^*| + \exp(\eps n) \ex{X \gets \cN(\mu^*,1)^n}{|M(X)-\mu^*|}\] for all $\mu$ and some fixed $\mu^*$ that does not depend on $\mu$. This follows from the strong \emph{group privacy} property of $(\eps,0$)-DP algorithms -- changing the entire sample $X \gets \mathcal{N}(\mu^*,1)^n$ to any other sample $x$ can only change the distribution of $M(X)$ by a pointwise multiplicative factor of $\exp(\eps n)$.  

The first observation shows that $g(\mu)=\mu$ on the entire real line $\mu\in\R$, and the second observation shows that $g(\mu)$ is uniformly bounded on the entire real line. This is a contradiction.

This result extends beyond the Gaussian distribution to exponential families. Furthermore, it extends from pure $(\eps,0)$-DP to concentrated DP~\citep{DworkR16,BunS16} and other variants of differential privacy that satisfy a strong group privacy property.  However, it does not apply to approximate $(\eps,\delta)$-DP because the group privacy guarantee underpinning the second observation breaks when $\delta>0$.

\subsection{Related Work}

Unbiased estimators have long been a topic of interest in statistics.
For example, topics such as the minimum variance unbiased estimator (MVUE) and the best linear unbiased estimator (BLUE) are textbook.
A number of celebrated results derive properties of estimators with low or no bias, often proving certain estimators are optimal within this class.
Some examples include the Gauss-Markov theorem~\citep{Gauss23,Markov00}, the Lehman-Scheff\'e theorem~\citep{LehmannS11}, and the Cram\`er-Rao bound~\citep{Cramer93,Rao92}.
These results often focus on unbiased estimators for mathematical convenience: it is easier to prove optimality within this restricted class than for general estimators.

Within the context of differential privacy, relatively little work has considered the bias of private estimators separately from their overall mean squared error. 
A number of works~\citep{DuchiJW13,BarberD14,KarwaV18,KamathLSU19,KamathSU20} bound the bias of the clipped mean, though only to the ends of trying to minimize the overall error of the estimator.
\citet{AminKMV19} examine bias-variance tradeoffs of a similar procedure in the context of private empirical risk minimization.
\citet{KamathLZ22} employ the mean estimation approach of \citet{KamathSU20} as an oracle for stochastic first-order optimization, but, due to specifics of their setting, employ a different balance between bias and noise. 
They raise the question of whether unbiased algorithms for mean estimation exist.  \citet{BarrientosWSB24,BarrientosWSB21b} empirically measure the bias induced by various mean estimation algorithms.
Works by \citet{ZhuVF21,ZhuFV22,ZhuFVDT23} study bias induced by a variety of differentially private algorithms. 
\citet{EvansK21,EvansKST22,CovingtonHHK21} give methods for unbiased private estimation, though these rely upon strong assumptions or caveat their unbiasedness guarantees (e.g., guaranteeing a statistic is unbiased only with high probability).
\citet{FerrandoWS22} appeal to the parametric bootstrap to help reduce the bias introduced by data clipping in parameteric settings.
\citet{AsiD20} study instance-specific error bounds for private mechanisms (on fixed datasets), and prove lower bounds on the error of the class of unbiased mechanisms.
Concurrent and independent work by \citet{NikolovT24} shows that appropriate Gaussian noise addition is essentially an optimal unbiased private mechanism for mean estimation in certain cases.
Our setting and theirs are different. We focus on mean estimation with distributional moment assumptions and unbounded domain. They study mean estimation for arbitrary distributions (and fixed datasets) on a bounded domain.

Beyond considerations of bias, private statistical estimation has been a topic of much recent interest. 
Mean estimation is perhaps the most fundamental question in this space, enjoying significant attention (see, e.g.,~\citet{BarberD14, KarwaV18, BunS19, KamathLSU19, KamathSU20, WangXDX20, DuFMBG20, BiswasDKU20,CaiWZ21,  BrownGSUZ21, HuangLY21, LiuKKO21, LiuKO22, KamathLZ22, HopkinsKM22, KothariMV22, TsfadiaCKMS22, KuditipudiDH23,BrownHS23}).
Most relevant to our work are those that focus on estimation with bounds on only the low-order central moments of the underlying distribution~\citet{BarberD14,KamathSU20,HopkinsKM22}, as the bias introduced due to clipping is more significant.
Other related problems involve private covariance or density estimation~\citet{BunKSW19,AdenAliAK21,KamathMSSU22, AshtianiL22, AlabiKTVZ23,HopkinsKMN23}.
Beyond these settings, other works have examined statistical estimation under privacy constraints for mixtures of Gaussians~\citep{KamathSSU19, AdenAliAL21,ChenCDEIST23}, graphical models~\citep{ZhangKKW20}, discrete distributions~\citep{DiakonikolasHS15}, median estimation~\citep{AvellaMedinaB19, TzamosVZ20,RamsayC21,RamsayJC22,BenEliezerMZ22}, and more.
Several recent works explore connections between private and robust estimation~\citep{LiuKKO21, HopkinsKM22, GeorgievH22, LiuKO22, KothariMV22, AlabiKTVZ23, HopkinsKMN23,ChenCDEIST23,AsiUZ23} and between privacy and generalization \citep{HardtU14,DworkFHPRR15,SteinkeU15,BassilyNSSSU16,RogersRST16,FeldmanS17}.
Emerging directions of interest include guaranteeing privacy when one person may contribute multiple samples \citep{LiuSYKR20, LevySAKKMS21, GeorgeRST24,AgarwalKMMSU24,ZhaoLSLWL24}, a combination of local and central DP for different users~\citep{AventDK20}, or estimation with access to some public data~\citep{BieKS22,BenDavidBCKS23}.
See \citet{KamathU20} for more coverage of recent work on private estimation.

\section{Main Bias-Accuracy-Privacy Trilemma}
\label{sec:trilemma}


We begin by motivating
our main negative result with a simulation experiment and a complementary experiment involving a real-world dataset. In our experiments, we consider the Laplace mechanism for non-negative data points with clipping threshold $T$:
\begin{align*}
    M(X) := \frac{1}{n}\sum_{i=1}^{n} \clip_{[0, T]}(X_i) + \Lap\left(\frac{T}{\eps n}\right).
\end{align*}

As a first step to understanding the bias of this mechanism, consider its behavior on a skewed population as the clipping radius and privacy parameter are increased. Figure~\ref{fig:lognormal_laplace_metrics} shows that varying the clipping radius induces a tradeoff between bias and error.
The Laplace mechanism appears to only achieve low-bias estimation on skewed populations by choosing a very large clipping radius. This comes either at the expense of privacy or at the expense of accuracy by increasing the noise to preserve outlier privacy.

\begin{figure}
  \centering
  \includegraphics[width=0.8\textwidth]{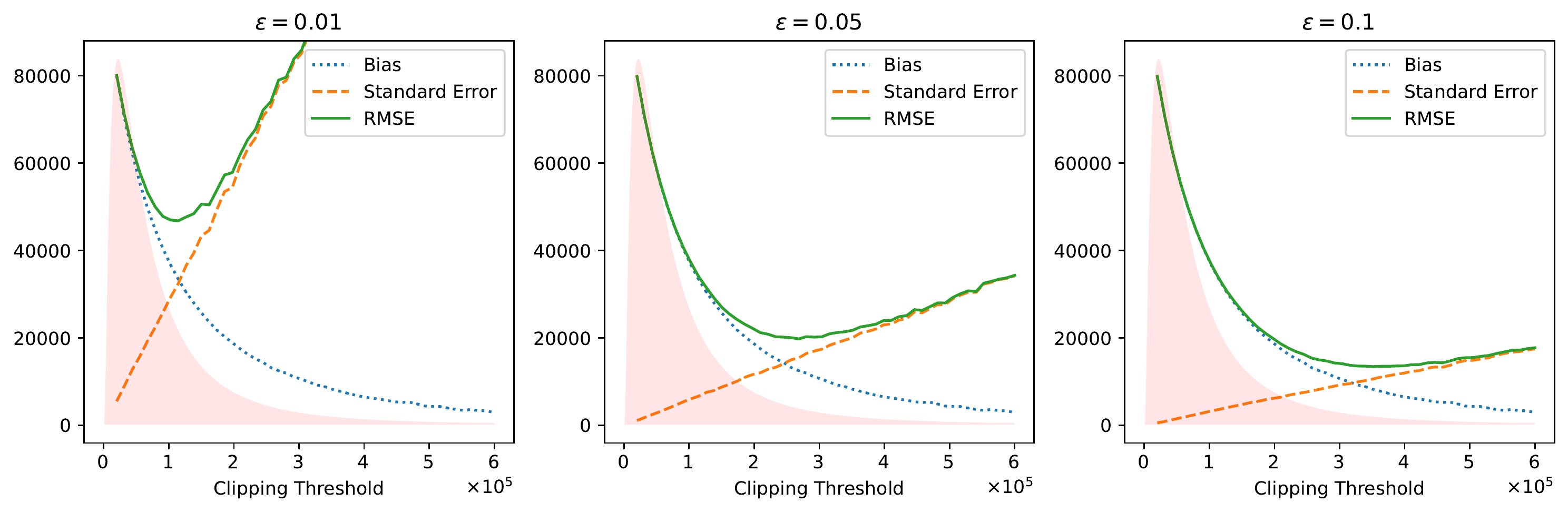}
  \caption{$n = 500$ samples are drawn from the log-normal distribution with median 60,000 and variance 1. A silhouette of the distribution is lightly shaded (y-axis not to scale). The data are clipped at a specified threshold $T$ and Laplace noise with parameter $\frac{T}{\epsilon n}$ is added.  The statistical bias due to clipping (dotted) and standard error introduced due to the noise and sampling (dashed) are plotted, as well as the root mean squared error (RMSE, solid).
  To minimize the RMSE, the clipping threshold must be chosen judiciously to balance the bias and the noise, and at a different value depending on the privacy parameter $\epsilon$.
  }
  \label{fig:lognormal_laplace_metrics}
\end{figure}

Next, we verify that this phenomenon holds for a real-world dataset as well. We conduct a similar experiment to \citet{DuchiJW17} using the University of California Report on 2011 Employee Pay \citep{UCal2011}. Though this dataset is publicly available, it is a reasonable proxy for a sensitive dataset drawn from a skewed population.

Regrettably, the 2010 version of the dataset used by \citet{DuchiJW17} is no longer publicly available. Nonetheless, summary statistics between 2010 and 2011 datasets are very similar. The 2011 dataset consists of $N = 259,043$ salaries. Mean, median, and maximum salaries are \$24,989 USD, \$40,865 USD, and \$2,884,880 USD respectively.

We remark that \citet{DuchiJW17} uses a clipping interval of $[-T, T]$. Their estimator is also designed for the local-DP setting and uses a noise multiplier that scales with an assumed population moment bound, so we cannot make direct comparisons.

In each round of the experiment, we randomly subsample 500 salaries from the dataset (treated as a population), apply clip-and-noise, and compare it to the population mean. We repeat the procedure $M = 5000$ times. Figure~\ref{fig:ucsalary_laplace_metrics} shows empirical estimates for the bias, standard error, and RMSE of the clip-and-noise Laplace mechanism. Figure~\ref{fig:ucsalary_laplace_optimal} shows the clipping threshold that minimizes RMSE at different settings of the privacy parameter $\epsilon$ together with the bias and standard error achieved by the optimal clipping threshold. Again, we observe the same tradeoff as in the simulated dataset. As privacy is relaxed, a larger clipping threshold may be safely chosen to reduce bias.

\begin{figure}[ht]
  \centering
  \includegraphics[width=0.8\textwidth]{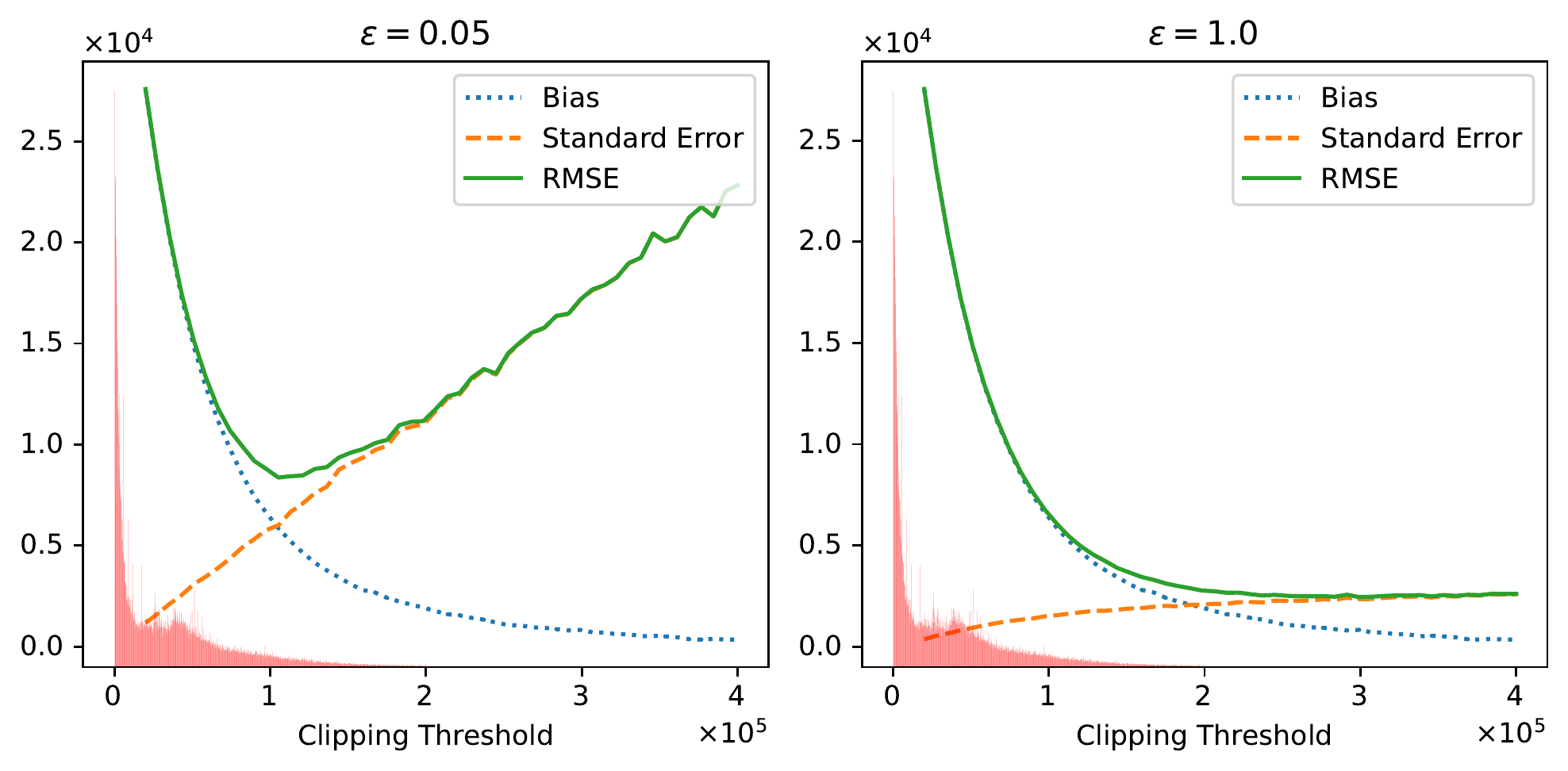}
  \caption{Bias, standard error, and root-mean-square error of the Laplace mechanism on the University of California Report on 2011 Employee Pay \citep{UCal2011}. A histogram of the dataset is indicated by the shaded region.}
  \label{fig:ucsalary_laplace_metrics}
\end{figure}

\begin{table}[ht]
    \centering
    \renewcommand{\arraystretch}{0.7}
    \begin{tabular}{l r r r r r r}
        \hline
        $\epsilon$ & 0.01 & 0.05 & 0.1 & 1 & 2 & 4 \\
        \hline
        Clipping threshold & 51,020 & 105,306 & 167,347 & 299,184 & 337,959 & 376,735 \\
        Bias & 14,909 & 5,825 & 2,565 & 708 & 534 & 349 \\
        Standard Error & 14,408 & 6,005 & 4,964 & 2,343 & 2,276 & 2,283 \\
        RMSE & 20,733 & 8,366 & 5,587 & 2,448 & 2,337 & 2,309 \\
        \hline
    \end{tabular}
    \caption{Clipping Thresholds that minimize RMSE of the Laplace Mechanism on the UC 2011 Salary Dataset.}
    \label{fig:ucsalary_laplace_optimal}
\end{table}

We now prove our main negative result, which is to show that this tradeoff is inherent and not unique to clip-and-noise. Informally, we show that if an algorithm is differentially private and has low bias, then it must have high error. There are, of course, other parameters that arise in the analysis, such as bounds on the tails of the unknown distribution $P$.

We provide two different proofs, which give slightly different results. The first proof directly applies the fingerprinting technique for lower bounds on differentially private estimation~\citep{BunUV14}, while the second proof is a ``black-box'' reduction. The black-box reduction yields a slightly weaker result, though the proof technique may be of more general interest.


\subsection{Negative Result via Fingerprinting}\label{subsec:fingerprinting}

We begin by stating our general result and, before giving the proof, we provide some remarks and corollaries to help interpret the result.

\begin{thm}[Bias-Accuracy-Privacy Tradeoff]\label{thm:flpb}
    Let $\eps,\delta,\bias,\rmse,\tau \ge 0$ and $\lambda > 1$.
    Let $M : \mathbb{R}^n \to \mathbb{R}$ be an $(\eps,\delta)$-DP algorithm that satisfies the following bias and accuracy properties.
    For any distribution $P$ on $\mathbb{R}$ with $\mu(P) \coloneqq \ex{X \gets P}{X} \in [0,1]$ and $\ex{X \gets P}{|X-\mu(P)|^\lambda} \le 1$, we have the following:
    \begin{align*}
        \left|\ex{X \gets P^n, M}{M(X)}-\mu(P)\right| &\le \bias,\\
        \ex{X \gets P^n, M}{\left|M(X)-\mu(P)\right|} &\le \rmse,\\
        \int_0^\infty \min\left\{\delta, \pr{X \gets P^n, M}{|M(X)-\mu(P)| > x} \right\} ~\mathrm{d}x &\le \rmse \cdot \tau.
    \end{align*}
    If $16\bias \le \gamma \le 1/5$, then
    $\rmse \ge \left(32n \cdot \sinh(\eps) \cdot \gamma^{1/(\lambda-1)} + 16n \cdot \tau \cdot \gamma^{-1}\right)^{-1}$.
\end{thm}

Note that, for small values of $\eps$, $\sinh(\eps) \approx \eps$, but, for large $\eps$, $\sinh(\eps) \approx \frac12 \exp(\eps)$.
Since this is the ``usual'' dependence on $\varepsilon$ in many such bounds under the constraint of DP, $\sinh$ allows us to capture behaviour in both regimes with a single function.

The first two accuracy conditions are not hard to interpret: 
The parameter $\bias$ bounds the bias of the algorithm, while $\rmse$ bounds the mean absolute deviation. By Jensen's inequality,
\[\left|\ex{X \gets P^n, M}{M(X)}-\mu(P)\right| \le \ex{X \gets P^n, M}{\left|M(X)-\mu(P)\right|} \le \sqrt{\ex{X \gets P^n, M}{\left(M(X)-\mu(P)\right)^2}}.\]
Thus, we can assume that $\bias \le \rmse$.
Furthermore, if an estimator $M(X)$ has a bound on the mean squared error of $\rmse^2$, it consequently also has a mean absolute error of at most $\rmse$.
Thus, our somewhat unconventional assumption controlling the mean absolute error only broadens the class of estimators against which our lower bound holds: for interpretability, one could instead replace this with $\rmse^2$ being the mean squared error of $M(X)$. 
    
The third property and the parameter $\tau$ is somewhat harder to interpret.
We note that this condition is implied by a bound on the MSE of the estimator via the following lemma.
It applies in more general circumstances as well, when we may have bounds on higher or lower moments of the estimator's error.

\begin{lem}[Setting $\tau=\delta^{1-1/\kappa}$ in Theorem~\ref{thm:flpb}]
    \label{lem:2nd_cond_implies_3rd}
     Let $\rmse,\delta \ge 0$ and $\kappa > 1$.
     Let $Y$ be a random variable satisfying $\ex{}{\left|Y\right|^\kappa} \le \rmse^\kappa$.
    Then $\int_0^\infty \min\{\delta, \pr{}{|Y| > x} \} ~\mathrm{d}x \le \rmse \cdot \delta^{1-1/\kappa}$.
\end{lem}

\begin{proof}[Proof of Lemma~\ref{lem:2nd_cond_implies_3rd}]
    We assume, without loss of generality, that $\delta<1$.
    Suppose, for now, there exists $c>0$ such that $\pr{}{|Y|>c}=\delta$.
    If $x\ge c$, then $\min\{\delta, \pr{}{|Y| > x} \} = \pr{}{|Y|>x} = \pr{}{|Y| \cdot \mathbb{I}[|Y|>c]>x}$.
    Likewise, if $x \le c$, then $\min\{\delta, \pr{}{|Y| > x} \} = \delta = \pr{}{|Y| \cdot \mathbb{I}[|Y|>c]>x}$.
    Thus,
    \begin{align*}
         \int_0^\infty \min\{\delta, \pr{}{|Y| > x} \} \mathrm{d}x 
        &= \int_0^\infty \pr{}{|Y|\cdot\mathbb{I}[|Y|>c] > x} \\
        &= \ex{}{|Y| \cdot \mathbb{I}[|Y|>c]}\\\
        &\le \ex{}{|Y|^\kappa}^{\frac{1}{\kappa}} \cdot \ex{}{\mathbb{I}[|Y|>c]^{\frac{\kappa}{\kappa-1}}}^{\frac{\kappa-1}{\kappa}} \tag{H\"older's Inequality} \\
        &= \ex{}{|Y|^\kappa}^{\frac{1}{\kappa}} \cdot \pr{}{|Y|>c}^{1-1/\kappa}\\
        &\le \alpha \cdot \delta^{1-1/\kappa}.
    \end{align*}
    If the distribution of $Y$ is continuous, then such a quantity $c$ is guaranteed to exist. In general, there exists $c\ge0$ such that $\pr{}{|Y|>c}\le\delta\le\pr{}{|Y|\ge c}$. We can define a random $I : \R \to \{0,1\}$ such that $\mathbb{I}[|Y|>c] \le I(|Y|) \le \mathbb{I}[|Y|\ge c]$ with probability $1$ and $\ex{}{I(|Y|)}=\delta$. The above proof carries through in general if we replace $\mathbb{I}[|Y|>c]$ with $I(|Y|)$.
\end{proof}

In particular, if we have a mean squared error bound for the estimator 
$\ex{}{(M(X)-\mu(P))^2}\le \alpha^2$, then the third condition of Theorem~\ref{thm:flpb} holds with $\tau = \sqrt{\delta}$. 
Larger values of $\kappa$ entail sharper tail bounds on the estimator, allowing us to set $\tau$ smaller (and thus implying stronger lower bounds), with $\tau \rightarrow \delta$ as $\kappa \to \infty$.

In general, Theorem~\ref{thm:flpb}'s lower bound on the error $\rmse$ is maximized by setting
\begin{equation}
    \gamma = \clip_{[16\beta, 1/5]}\left( \left(\frac{(\lambda-1)\tau}{2\sinh(\eps)}\right)^{1-1/\lambda} \right).
    \label{eq:optgamma}
\end{equation}
Combining this parameter setting for $\gamma$, along with the bound of $\tau = \delta^{1-1/\kappa} $ given by Lemma~\ref{lem:2nd_cond_implies_3rd}, and focusing on the most natural case of $\kappa = 2$ (i.e., we assume only that the estimator has bounded variance), gives the following result.

\begin{cor}[Combining Theorem~\ref{thm:flpb}, Lemma~\ref{lem:2nd_cond_implies_3rd} (with $\kappa=2$), and Equation~\ref{eq:optgamma}.]\label{cor:optgamma}
    Let $M : \mathbb{R}^n \to \mathbb{R}$ be $(\eps,\delta)$-DP and satisfy the following bias and accuracy properties.
    For any distribution $P$ on $\mathbb{R}$ with $\mu(P) \coloneqq \ex{X \gets P}{X} \in [0,1]$ and $\ex{X \gets P}{|X-\mu(P)|^\lambda} \le 1$, we have
    \[\left|\ex{X \gets P^n, M}{M(X)}-\mu(P)\right| \le \bias
    ~~~\text{ and }~~~
    \ex{X \gets P^n, M}{\left(M(X)-\mu(P)\right)^2} \le \rmse^2.\]
    If $\bias \le \frac{1}{80}$ and $\delta \le \left(\frac{2\cdot\sinh(\eps)}{5^{1+\frac{1}{\lambda-1}} \cdot (\lambda-1)}\right)^2$, then $$\rmse \ge \left(32 \cdot n \cdot \sinh(\eps) \cdot \frac{\lambda}{\lambda-1} \cdot \max \left\{ (16\bias)^{\frac{1}{\lambda-1}} , \left(\frac{(\lambda-1)\sqrt{\delta}}{2\cdot\sinh(\eps)} \right)^{1/\lambda} \right\}\right)^{-1}.$$
\end{cor}

We illustrate the representative case where the underlying distribution has bounded variance by further fixing $\lambda = 2$.
Combining the resulting lower bound with the non-private rate (Proposition~\ref{prop:nondp_mse}) gives the following result.

\begin{thm}[Setting $\lambda=2$ in Corollary \ref{cor:optgamma} to get Theorem~\ref{thm:main_trilemma_informal}]\label{thm:main_trilemma_formal}
    Let $M : \mathbb{R}^n \to \mathbb{R}$ be $(\eps,\delta)$-DP and satisfy the following bias and accuracy properties.
    For any distribution $P$ on $\mathbb{R}$ with $\mu(P) \coloneqq \ex{X \gets P}{X} \in [0,1]$ and $\ex{X \gets P}{(X-\mu(P))^2} \le 1$, we have
    \[\left|\ex{X \gets P^n, M}{M(X)}-\mu(P)\right| \le \bias
    ~~~\text{ and }~~~
    \ex{X \gets P^n, M}{\left(M(X)-\mu(P)\right)^2} \le \rmse^2.\]
    If $\bias \le 1/80$ and $\delta \le \left(\frac{2}{25} \sinh(\eps)\right)^2$, 
    then
    \[\rmse \ge \max\left\{ \frac{1}{\sqrt{6(n+2)}} , \frac{1}{64 \cdot n \cdot \sinh(\eps) \cdot \max\left\{ 16 \cdot \bias , \sqrt{\frac{\sqrt{\delta}}{2\cdot\sinh(\eps)}}\right\}} \right\} \ge \Omega\left( \frac{1}{\sqrt{n}} + \min\left\{ \frac{1}{n \eps \bias}, \frac{1}{n\sqrt{\eps \sqrt{\delta}}} \right\} \right).\]
\end{thm}

We now prove our general result. Our proof follows what is known as the fingerprinting approach. Fingerprinting codes were first studied by \citep{BonehS95} in the context of cryptographic traitor-tracing schemes. \citet{Tardos03} gave an optimal construction of fingerprinting codes. Finger-
printing codes were used to prove negative results for differentially private algorithms by \citet{BunUV14,DworkSSUV15}. Subsequently, many works have expanded this methodology \citep{SteinkeU15,BunSU17,SteinkeU17a,SteinkeU17b, CaiWZ20, CaiWZ21, KamathMS22, CaiWZ23}.

We will use the following lemma in the proof.
\begin{lem}[Fingerprinting Derivative Lemma {\cite[Lemma 9]{SteinkeU17b}}]\label{lem:fp_derivative}\nocite{Steinke2022}
    Let $f : \{0,1\}^n \to \mathbb{R}$ be an arbitrary function. Define $g : [0,1] \to \mathbb{R}$ by $g(p) = \exin{X \gets \mathsf{Bernoulli}(p)^n}{f(X)}$. Then, for all $p \in [0,1]$, we have
    \begin{align*}
        g'(p) \cdot p(1-p) = \ex{X \gets \mathsf{Bernoulli}(p)^n}{f(X) \cdot \sum_i^n (X_i-p)}.
    \end{align*}
\end{lem}

\begin{proof}[Proof of Theorem~\ref{thm:flpb}.]
    For $p \in [0,1]$ and $v>0$, define $\mathcal{D}_{v,p} = v \cdot \mathsf{Bernoulli}(p)$ -- i.e., a sample from $\mathcal{D}_{v,p}$ is $0$ with probability $1-p$ and $v$ with probability $p$. Then $\mu(\mathcal{D}_{v,p}) = \ex{X \gets \mathcal{D}_{v,p}}{X} = vp$ and
    \[\ex{X \gets \mathcal{D}_{v,p}}{|X-\mu(\mathcal{D}_{v,p})|^\lambda} = (1-p)(vp)^\lambda+p(v(1-p))^\lambda \le 2p v^\lambda.\]
    If we ensure $v \le (2p)^{-1/\lambda}\le 1/p$, then the $\lambda$-th absolute central moment is below $1$, and the mean is in the interval $[0,1]$, so the bias and accuracy guarantees of $M$ apply.

    For $v>0$, define $g_v : [0,1] \to \mathbb{R}$ by
    \[g_v(p) \coloneqq \ex{X \gets \mathcal{D}_{v,p}^n,M}{M(X)} .\]

    By Lemma~\ref{lem:fp_derivative},
    for all $v>0$ and $p \in [0,1]$, we have
    \begin{align}
        \ex{X \gets \mathcal{D}_{v,p}^n,M}{M(X) \cdot \sum_i^n \left(\frac1v X_i-p\right)} = p(1-p) g_v'(p).\label{eq:fingerprinting_bernoulli}
    \end{align}
    
    Fix $0<a<b\le1/2$ and $0< v \le (2b)^{-1/\lambda}$ (to be determined later). Now, let $P \in [a,b]$ be a random variable with density $\propto \frac{1}{P(1-P)}$ -- i.e., $\forall t \in [a,b], ~~ \pr{}{P \le t} = \tfrac{\int_a^t \frac{1}{x(1-x)} \mathrm{d}x}{\int_a^b \frac{1}{x(1-x)} \mathrm{d}x}$. Conditioned on $P$, let $X_1, \dots, X_n \in \mathbb{R}$ be independent samples from $\mathcal{D}_{v,P}$. Now,
    \begin{align*}
        \ex{P,X,M}{M(X) \cdot \sum_i^n \left(\frac1v X_i-P\right)} &= \ex{P}{P(1-P) g_v'(P) } \tag{Equation~\ref{eq:fingerprinting_bernoulli}}\\
        &= \frac{ \int_a^b g_v'(p) \mathrm{d}p }{ \int_a^b \frac{1}{x(1-x)} \mathrm{d}x } \\
        &= \frac{g_v(b)-g_v(a)}{\log(b/(1-b))-\log(a/(1-a))}.
    \end{align*}
    By our bias assumption, $|g_v(b)-vb| \le \bias$ and $|g_v(a)-va| \le \bias$. Thus, 
    \[\ex{P,X,M}{M(X) \cdot \sum_i^n \left(\frac1v X_i-P\right)} \ge \frac{v\cdot(b-a)-2\bias}{\log\left(\frac{b \cdot (1-a)}{a \cdot (1-b)}\right)}.\]
    Since $\ex{}{X_i} = vP$ for all $i$, we can center $M(X)$ and rearrange slightly:
    \[ \sum_i^n  \ex{P,X,M}{\left( M(X) - vP \right) \cdot \left(\frac1v X_i-P\right)} \ge \frac{v \cdot (b-a)-2\bias}{\log\left(\frac{b \cdot (1-a)}{a \cdot (1-b)}\right)}.\]

    Next, we will use differential privacy to prove an upper bound on this quantity. 
    Fix an arbitrary $i \in [n]$ and fix $P=p \in [a,b]$. Our goal is to upper bound $\ex{X \gets \mathcal{D}_{v,p}^n,M}{\left( M(X) - vp \right) \cdot \left(\tfrac1v X_i-p\right)}$.
    
    Since $M$ satisfies $(\eps,\delta)$-DP, the distribution of the pair $\left( M(X) , X_i \right)$ is $(\eps,\delta)$-indistinguishable\footnote{We say that distributions $P$ and $Q$ over a set $\cX$ are $(\eps,\delta)$-indistinguishable (denoted by $P \edequiv Q$), if $\exp(-\eps)(Q(E) - \delta) \leq P(E) \leq \exp(\eps)Q(E) + \delta$ for all measurable $E \subseteq \cX$.} from that of $\left( M(X_{-i},\tilde{X}_i) , X_i \right)$, where $(X_{-i},\tilde{X}_i)$ denotes the dataset $X$ with $X_i$ replaced by $\tilde{X}_i$; here $\tilde{X}_i \gets \mathcal{D}_{v,p}$ is a fresh sample from the distribution. Now $\tilde{X}_i$ and $X_i$ are interchangeable, this means the distribution of $\left( M(X_{-i},\tilde{X}_i) , X_i \right)$ is identical to that of $\left( M(X) , \tilde{X}_i \right)$. By transitivity, the distribution of $\left(M(X),X_i\right)$ is $(\eps,\delta)$-indistinguishable from that of $\left( M(X) , \tilde{X}_i \right)$. In particular\footnote{If two random variables $X$ and $Y$ have the same distribution, then we write $X \stackrel{d}{=} Y$.},
    \[\left( M(X) - vp \right) \cdot \left(\tfrac1v X_i-p\right) \edequiv \left( M(X_{-i},\tilde{X}_i) - vp \right) \cdot \left(\tfrac1v X_i-p\right) \stackrel{d}{=} \left( M(X) - vp \right) \cdot \left(\tfrac1v \tilde{X}_i-p\right).\]
    
    We also have $\left|\left( M(X) - vp \right) \cdot \left(\tfrac1v X_i-p\right)\right| \le \left| M(X) - vp \right|$ with probability $1$.
    
    Thus
    \begin{align*}
        \ex{}{\left( M(X) - vp \right) \cdot \left(\frac1v X_i-p\right)} &= \ex{}{\max\{\left( M(X) - vp \right) \cdot \left(\frac1v X_i-p\right),0\}}\\ &~~~~~~~~~~- \ex{}{\max\{-\left( M(X) - vp \right) \cdot \left(\frac1v X_i-p\right),0\}} \\
            &= \int\limits_{0}^{\infty}\pr{}{\left( M(X) - vp \right) \cdot \left(\frac1v X_i-p\right) > x} \mathrm{d}x\\ &~~~~~~~~~~- \int\limits_{0}^{\infty}\pr{}{\left( M(X) - vp \right) \cdot \left(\frac1v X_i-p\right) < -x} \mathrm{d}x .
    \end{align*}
    We have
    \[\pr{}{\left( M(X) - vp \right) \cdot \left(\frac1v X_i-p\right) > x} \le \pr{}{\left| M(X) - vp \right| > x}\]
    and simultaneously,
    \begin{align*}
        \pr{}{\left( M(X) - vp \right) \cdot \left(\frac1v X_i-p\right) > x} &\le \exp(\eps) \cdot \pr{}{\left( M(X) - vp \right) \cdot \left(\frac1v \tilde{X}_i-p\right) > x} + \delta\\
        &= \exp(\eps) \cdot p \cdot \pr{}{(M(X)-vp) \cdot (1-p) > x} \\&~~~+ \exp(\eps) \cdot (1-p) \cdot \pr{}{(M(X)-vp) \cdot (0-p) > x} + \delta.
    \end{align*}
    Define $\delta(x) \coloneqq \min\{ \delta, \pr{}{\left| M(X) - vp \right| > x} \}$. Then
    \begin{align*}
        \int\limits_0^\infty \pr{}{\left( M(X) - vp \right) \cdot \left(\frac1v X_i-p\right) > x} \mathrm{d}x
        &\le \int\limits_0^\infty \exp(\eps) \cdot p \cdot \pr{}{(M(X)-vp) \cdot (1-p) > x} \mathrm{d}x \\ &~~~~~~~+ \int\limits_0^\infty \exp(\eps) \cdot (1-p) \cdot \pr{}{(M(X)-vp) \cdot (0-p) > x} \mathrm{d}x \\ &~~~~~~~+ \int\limits_0^\infty\delta(x) \mathrm{d}x\\
        &= \ex{}{\exp(\eps) \cdot p \cdot \max\{(M(X)-vp) \cdot (1-p),0\}} \\&~~~~~~~+ \ex{}{ \exp(\eps) \cdot (1-p) \cdot \max\{(M(X)-vp) \cdot (0-p) , 0 \}} \\&~~~~~~~+ \int\limits_0^\infty \delta(x) \mathrm{d}x \\
        &\le \exp(\eps) \cdot p (1-p) \cdot \ex{}{|M(X)-vp|} + \rmse \cdot \tau.
    \end{align*}
    In the above, the final inequality holds because
    \begin{align*}
        \ex{}{\max\{M(X)-vp,0\}}+\ex{}{\max\{-M(X)+vp,0\}}
            &= \ex{}{|M(X)-vp|}
    \end{align*}
    and due to the third utility assumption in our theorem statement. Similarly,
    \begin{align*}
        \int\limits_0^\infty \mathbb{P}\bigg[(M(X) &- vp) \cdot \left(\frac1v X_i-p\right) < -x \bigg] \mathrm{d}x \\
            &\ge \int\limits_0^\infty \max\left\{\begin{array}{c}0,\\ \exp(-\eps) \left( \pr{}{\left( M(X) - vp \right) \cdot \left(\frac1v \tilde{X}_i-p\right) < -x} -\delta \right) \end{array}\right\} \mathrm{d}x \\
            &= \exp(-\eps) \cdot \int\limits_0^\infty \pr{}{\left( M(X) - vp \right) \cdot \left(\frac1v \tilde{X}_i-p\right) < -x}\mathrm{d}x \\ &~~~~~+ \exp(-\eps) \cdot \int\limits_0^\infty\max\left\{\begin{array}{c} -\delta, \\ - \pr{}{\left( M(X) - vp \right) \cdot \left(\frac1v \tilde{X}_i-p\right) < -x} \end{array} \right\} \mathrm{d}x \\
            &\ge \exp(-\eps) \cdot \int\limits_0^\infty \pr{}{\left( M(X) - vp \right) \cdot \left(\frac1v \tilde{X}_i-p\right) < -x} - \delta(x) \mathrm{d}x \\
            &= \exp(-\eps) \!\cdot\! \int\limits_0^\infty p \cdot \pr{}{\left( M(X) - vp \right) \cdot (1-p) < -x}\mathrm{d}x \\ &~~~~~+ \exp(-\eps) \!\cdot\! \int\limits_0^\infty (1-p) \cdot \pr{}{\left( M(X) - vp \right) \cdot (0-p) < -x} \mathrm{d}x \\ &~~~~~- \exp(-\eps) \!\cdot\! \int\limits_0^\infty \delta(x) \mathrm{d}x \\
            &\ge \exp(-\eps) \cdot  p(1-p) \cdot \ex{}{|M(X)-vp|} - \exp(-\eps) \cdot \rmse \cdot \tau.
    \end{align*}
    Putting these two pieces together, we have:
        \begin{align*}
        \ex{X_1, \cdots, X_n \gets \mathcal{D}_{v,p}}{\left( M(X) - vp \right) \cdot \left(\frac1v X_i-p\right)}
            &\le \int\limits_0^\infty \pr{}{\left( M(X) - vp \right) \cdot \left(\frac1v X_i-p\right) > x}\mathrm{d}x \\ &~~~~~~~~~~- \int\limits_0^\infty \pr{}{\left( M(X) - vp \right) \cdot \left(\frac1v X_i-p\right) < -x} \mathrm{d}x\\
            &\le \left(\exp(\eps) - \exp(-\eps) \right) \cdot  p(1-p) \cdot \ex{}{|M(X)-vp|} \\ &~~~~~+ (1+\exp(-\eps)) \cdot \rmse \cdot \tau \\
            &\le \left(\exp(\eps) - \exp(-\eps) \right) \cdot  b(1-b) \cdot \rmse + 2 \rmse \tau ,
    \end{align*}
    as $p \le b \le 1/2$.
    Then, combining this with our lower bound, we have
    \begin{align*}
        \frac{v \cdot (b-a)-2\bias}{\log\left(\frac{b \cdot (1-a)}{a \cdot (1-b)}\right)} 
            &\le \sum_i^n  \ex{P,X,M}{\left( M(X) - vP \right) \cdot \left(\frac1v X_i-P\right)} \\
            &\le n \cdot \left( \left(\exp(\eps) - \exp(-\eps) \right) \cdot  b(1-b) \cdot \rmse + 2\rmse\tau \right) \\
            &\le \rmse \cdot n \cdot 2 \cdot \left( \sinh(\eps) \cdot  b + \tau \right),
    \end{align*}
    which rearranges to
    \[\rmse \ge \frac{v \cdot (b-a)-2\bias}{ 2n \cdot \left( \sinh(\eps) \cdot  b  + \tau \right) \cdot \log\left(\frac{b \cdot (1-a)}{a \cdot (1-b)}\right)  } .\]
    It only remains to set the parameters subject to the constraints $0<a<b\le1/2$ and $0< v \le (2b)^{-1/\lambda}$.
    First, we set $b=2a$, and $v=(2b)^{-1/\lambda}=(4a)^{-1/\lambda}$ and assume $ (8\bias)^{\frac{\lambda}{\lambda-1}} \le a \le 1/5$, which simplifies the above expression to 
    \[\rmse \ge \frac{a^{1-1/\lambda} \cdot 4^{-1/\lambda} - 2\bias}{ 2n \cdot \left( \sinh(\eps) \cdot  2a  + \tau \right) \cdot \log\left(\frac{2 \cdot (1-a)}{1-2a}\right)  } \ge \frac{a^{1-1/\lambda}-8\bias}{8n \cdot (\sinh(\eps) \cdot 2a + \tau) }.\]
    We reparameterize $a = \gamma^{\frac{\lambda}{\lambda-1}}$ for some $16\bias \le \gamma \le 1/5$ to obtain
    \[\rmse \ge \frac{\gamma-8\bias}{8n \cdot (\sinh(\eps) \cdot 2\cdot\gamma^{\frac{\lambda}{\lambda-1}} + \tau) } \ge \frac{\gamma/2}{8n \cdot (\sinh(\eps) \cdot 2\cdot\gamma^{\frac{\lambda}{\lambda-1}} + \tau) } = \frac{1}{32n\sinh(\eps) \gamma^{1/(\lambda-1)} + 16n \tau \gamma^{-1}}.\]
    
    This completes our proof.
\end{proof}

\subsection{Alternate Proof of Trilemma via Amplification}
\label{subsec:amplification}

In this subsection, we show that known MSE lower bounds (without bias constraints) \citep{KamathSU20} combined with privacy amplification via shuffling \citep{ErlingssonFMRTT19,CheuSUZZ19,balle2019privacy,FeldmanMT22,FeldmanMT23} can also be used to derive qualitatively similar lower bounds on MSE for private estimators with low bias as those yielded by fingerprinting in Section~\ref{sec:trilemma}. Our reduction provides an alternative perspective on the bias-accuracy-privacy tradeoff, and could prove useful in future work as it is more ``generic'' than the fingerprinting approach.
Specifically, we will use the following lower bound on the MSE of a private estimator in a black-box manner.
\begin{thm}[{\cite[Theorem~3.8]{KamathSU20}}]\label{thm:ksu}
    Let $M : \R^n \to \R$ be $(\eps,\delta)$-DP. 
    Then, for some distribution $P$ with $\mu(P)\coloneqq\exin{X \gets P}{X}\in[-1,1]$ and $\exin{X \gets P}{(X-\mu(P))^2}\le1$,
    \[
        \ex{X \gets P^n}{(M(X)-\mu(P))^2} \ge \Omega\left(\frac{1}{n(\eps+\delta)}\right).
    \]
\end{thm}

The other ingredient in our proof is the following extension of the privacy amplification by subsampling result of \cite{FeldmanMT22}. Specifically we extend from the setting of local differential privacy (where each algorithm has one input) to the setting where a dataset is randomly partitioned into blocks of fixed size $n>1$, and these blocks are processed by a sequence of private mechanisms.
To be specific, we randomly partition the dataset as follows. We first arrange the dataset as a matrix with $m$ columns and $n$ rows. Then, for each row $i \in [n]$, we perform a uniformly random permutation of the $m$ elements in that row.

\begin{thm}[Extension of Privacy Amplification by Shuffling \citep{FeldmanMT22} to Larger Inputs]
    \label{thm:amplification_by_shuffling}
    Suppose we have a randomized function $L_i : \cY_1 \times \dots \times \cY_{i - 1} \times \cX^n \to \cY_i$ for each $i \in [m]$ such that $L_i(y, x)$ is $(\eps_0, \delta_0)$-DP in the parameter $x\in\cX^n$ for every fixed $y$. Consider $L_m \otimes \dots \otimes L_1 : (\cX^n)^m \to \cY_1 \times \dots \times \cY_m$ defined by
    \begin{align*}
        (L_m \otimes \dots \otimes L_1)(x_1, \dots, x_m) \coloneqq (y_1, \dots, y_m)
    \end{align*}
    where we recursively define $y_i \coloneqq L_i(y_1, \dots, y_{i - 1}, x_i)$. In addition, consider the shuffle operator $\Pi : (\cX^n)^m \to (\cX^n)^m$ given by
    \begin{align*}
        \Pi((x_1^1, \dots, x_1^n), \dots, (x_m^1, \dots, x_m^n))
            \coloneqq ((x_{\pi_1(1)}^1, \dots, x_{\pi_n(1)}^n), \dots, (x_{\pi_1(m)}^1, \dots, x_{\pi_n(m)}^n))
    \end{align*}
    where $\pi_1, \dots, \pi_n$ are uniform i.i.d. permutations of $[m]$.
    Then, for any $\delta_1 \in [2\exp(-\frac{m}{16\exp(\eps_0)}), 1]$, $L_m \otimes \dots \otimes L_1 \circ \Pi$ is $\left(\eps_1, \delta_1 + (\exp(\eps_1)+1)(\exp(-\eps_0)/2+1)m\delta_0\right)$-DP, where
    \begin{equation}
        \eps_1 := \log\left(1 + 8\frac{\exp(\eps_0)-1}{\exp(\eps_0)+1}\left(\sqrt{\frac{\exp(\eps_0)\log(4/\delta_1)}{m}} + \frac{\exp(\eps_0)}{m}\right)\right)
        .\label{eq:shuffle_eps}
    \end{equation}
\end{thm}

Note that, when $\eps_0 = O(1)$, we have $\eps_1 = O(\eps_0\sqrt{\log(1/\delta_1)/m})$.

Before we prove Theorem~\ref{thm:amplification_by_shuffling}, we will first show it can be used to prove a slightly weaker version of Theorem~\ref{thm:main_trilemma_informal} by reduction to Theorem~\ref{thm:ksu}.

\begin{thm}[Bias-Accuracy-Privacy Trilemma via Shuffling]\label{thm:heavy_tail_bias_bound}
    Let $M : \R^n \to \R$ be $(\eps, \delta)$-DP and satisfy the following bias and accuracy properties. For any distribution $P$ over $\R$ with $\mu(P) \coloneqq \ex{X \gets P}{X} \in [-1, 1]$ and $\ex{X \gets P}{(X-\mu(P))^2} \le 1$, we have 
        \[\left|\ex{X \gets P^n, M}{M(X)} - \mu(P)\right| \le \bias ~~~\text{ and }~~~
        \ex{X \gets P^n, M}{\left(M(X) - \mu(P)\right)^2} \le \rmse^2.\]
    If $\beta^2 \le \widetilde\Omega\left( \frac{1}{n\eps}\right)$ and $\delta \le O(n^3\eps^4\beta^6)$, then
    \(
        \rmse^2 \geq \Omega\p{\frac{1}{n^2\eps^2 \bias^2\log(1/n\eps^2\beta^2)}}.
    \)
\end{thm}

The proof relies on the well-known post-processing property of differential privacy.

\begin{lem}[Post-Processing~\citep{DworkMNS06}]\label{lem:post_processing}
   If $M : \cX^n \to \cY$ is $(\eps,\delta)$-DP and
   $P : \cY \to \cZ$ is any randomized function, then
   the algorithm $P \circ M$ is $(\eps,\delta)$-DP.
\end{lem}

\begin{proof}[Proof of Theorem~\ref{thm:heavy_tail_bias_bound}]
    Let $m \in \N$ (we delay our choice of $m$ until later). Consider $A_m : (\R^n)^m \to \R$ defined by
    \[
        \forall x_1,\dots,x_m \in \R^n,~~~ A_m((x_1,\dots,x_m)) = \frac{1}{m}\sum_{i = 1}^m M(x_i).
    \]
    Fix some distribution $P$ with mean and variance bounded by $1$.
    This gives us the following guarantee about the mean of $A_m$.
    \begin{align*}
        \check\mu \coloneqq \ex{(X_1,\dots,X_m)\gets (P^n)^m,A_m}{A_m((X_1,\dots,X_m))} &= \ex{(X_1,\dots,X_m)\gets (P^n)^m,M}{\frac{1}{m}\sum_{i=1}^{m}M(X_i)}\\
            &= \frac{1}{m}\sum_{i=1}^{m}\ex{X_i \gets P^n,M}{M(X_i)}\\
            &= \ex{X \gets P^n,M}{M(X)}
    \end{align*}
    Thus, the bias of $A_m$ is at most $\bias$, as we see from the following.
    \begin{align*}
        \abs{\check\mu-\mu(P)}=\abs{\ex{(X_1,\dots,X_m)\gets (P^n)^m,A_m}{A_m((X_1,\dots,X_m))}-\mu(P)} &= \abs{\ex{X \gets P^n,M}{M(X)}-\mu(P)} \leq \bias
    \end{align*}
    Similarly, the MSE of $A_m$ is
    \begin{align*}
        \ex{(X_1,\cdots,X_m) \gets (P^n)^m,A_m}{\left(A_m(X_1,\cdots,X_m)-\mu(P)\right)^2}
        &= (\check\mu-\mu(P))^2\\ &~~~~~~~~~~+ \ex{(X_1,\cdots,X_m) \gets (P^n)^m,A_m}{\left(A_m(X_1,\cdots,X_m)-\check\mu\right)^2}\\
        &= (\check\mu-\mu(P))^2\\ &~~~~~~~~~~+ \ex{(X_1,\cdots,X_m) \gets (P^n)^m,M}{\left(\frac1m \sum_{i=1}^m M(X_i)-\check\mu\right)^2}\\
        &= (\check\mu-\mu(P))^2 + \frac1m \cdot\ex{X \gets P^n,M}{\left(M(X)-\check\mu\right)^2}\\
        &\le \bias^2 + \frac{\rmse^2}{m}.
    \end{align*}

    Let $\Pi : (\cX^n)^m \to (\cX^n)^m$ be the shuffle operator described in Theorem~\ref{thm:amplification_by_shuffling}.
    Since any set of samples drawn i.i.d.~from a distribution is invariant under shuffling, $(A_m \circ \Pi)(X)$ has the same distribution as $A_m(X)$ when $X \gets (P^n)^m$. In particular, $A_m \circ \Pi$ has the same bias and MSE as $A_m$ on inputs from $(P^n)^m$. 
    Privacy amplification by shuffling (Theorem~\ref{thm:amplification_by_shuffling}) and post-processing (Lemma~\ref{lem:post_processing}), imply that $A_m \circ \Pi$ is $\left(\eps' \coloneqq O(\eps\sqrt{\log(1/\delta_1)/m}), \delta' \coloneqq \delta_1 + O(\delta m)\right)$-DP for all $\delta_1 \in [2\exp(-\frac{m}{16\exp(\eps)}), 1]$.
    
    Now, we apply Theorem~\ref{thm:ksu} \cite[Theorem~3.8]{KamathSU20} to $A_m$:
    There exists a distribution $P$ with mean and variance bounded by $1$, such that 
    \[\ex{(X_1,\cdots,X_m) \gets (P^n)^m,A_m}{\left(A_m(X_1,\cdots,X_m)-\mu(P)\right)^2} \ge \Omega\left(\frac{1}{nm(\eps'+\delta')}\right).\]
    Combining all these inequalities gives
    \begin{align*}
        \bias^2 + \frac{\rmse^2}{m} &\ge \ex{(X_1,\cdots,X_m) \gets (P^n)^m}{\left(A_m(X_1,\cdots,X_m)-\mu(P)\right)^2}\\
            &\ge \Omega\left(\frac{1}{nm(\eps'+\delta')}\right)\\
            &\ge \Omega\left( \frac{1}{nm(\eps\sqrt{\log(1/\delta_1)/m}+\delta_1 + \delta m)}\right).
    \end{align*}
    This rearranges to \[\rmse^2 \ge \Omega\left(\frac{1}{n\eps\sqrt{\log(1/\delta_1)/m}+n\delta_1+nm\delta}\right) - m\bias^2.\]
    It only remains to set $m \in \N$ and $\delta_1 \in [2\exp(-\frac{m}{16\exp(\eps)}), 1]$ to maximize this lower bound.

    Now, we assume that $\delta_1 \le O(\eps/\sqrt{m})$ and $\delta \le O(\eps/m^{3/2})$. Then the first term in the denominator dominates and we have
    \[\rmse^2 \ge \Omega\left(\frac{1}{n\eps\sqrt{\log(1/\delta_1)/m}}\right) - m\bias^2.\]
    Then setting $m=\Theta\left(\frac{1}{n^2\eps^2\bias^4\log(1/\delta_1)}\right)$ optimizes the expression giving
    \[\rmse^2 \ge \Omega\left(\frac{1}{n^2\eps^2\bias^2\log(1/\delta_1)}\right).\]
    We set $\delta_1 = n\eps^2\bias^2 \le O(\eps/\sqrt{m})$. This satisfies $\delta_1 \in [2\exp(-\frac{m}{16\exp(\eps)}), 1]$ as long as $\beta^2 \le 1/n\eps^2$ and $m=\Theta\left(\tfrac{1}{n^2\eps^2\bias^4\log(1/n\eps^2\beta^2)}\right) \ge O(\exp(\eps) \log(1/n\eps^2\beta^2))$. The latter constraint rearranges to $\beta^2\log(1/n\eps^2\beta^2) \le \Omega\left(\tfrac{\exp(-\eps)}{n\eps}\right)$.
    To conclude, we note that the assumption $\delta \le O(\eps/m^{3/2})$ is implied by $\delta \le O(n^3\eps^4\beta^6)$.
\end{proof}

We conclude this section by proving the extension of privacy amplification by shuffling (Theorem~\ref{thm:amplification_by_shuffling}). This proof is a direct reduction to the following result of \cite{FeldmanMT22}.

\begin{thm}[Local Privacy Amplification by Shuffling {\cite[Theorem~3.8]{FeldmanMT22}}]
    \label{thm:amplification_by_local_shuffling}
    Let $m \in \Z^+$, let $\cX$ be the data universe, and let $\cY_1,\dots,\cY_m$ be image spaces. Suppose for each $i \in [m]$, we have a randomized function $R_i : \cY_1 \times \dots \times \cY_{i - 1} \times \cX \to \cY_i$ such that $R_i(y, a)$ is $(\eps_0, \delta_0)$-DP in the parameter $a\in\cX$ for every fixed $y \in \cY_1 \times \dots \times \cY_{i - 1}$. Consider $R_m \otimes \dots \otimes R_1 : \cX^m \to \cY_1 \times \dots \times \cY_m$ defined by
    \begin{align*}
        (R_m \otimes \dots \otimes R_1)(x_1, \dots, x_m) \coloneqq (y_1, \dots, y_m)
    \end{align*}
    where we recursively define $y_i \coloneqq R_i(y_1, \dots, y_{i - 1}, x_i)$. In addition, consider the random shuffle operator $S : \cX^m \to \cX^m$ given by
    \begin{align*}
        S(x_1, \dots, x_m)
            \coloneqq (x_{\pi(1)}, \dots, x_{\pi(m)})
    \end{align*}
    where $\pi$ is a uniformly random permutation on $[m]$. Then, for any $\delta_1 \in [2\exp(-\frac{m}{16\exp(\eps_0)}), 1]$, the function $R_m \otimes \dots \otimes R_1 \circ S : \cX^m \to \cY_1 \times \cdots \cY_m$ is $\left(\eps_1, \delta_1 + (\exp(\eps_1)+1)(1+\exp(-\eps_0)/2)m\delta_0 \right)$-DP, where
    $\eps_1$ is as in Equation~\ref{eq:shuffle_eps}.
\end{thm}

\begin{proof}[Proof of Theorem~\ref{thm:amplification_by_shuffling}]
    Consider neighboring datasets $x = ((x_1^1, \dots, x_1^n), \dots, (x_m^1, \dots, x_m^n))$ and $x' = (({x'}_1^1, \dots, {x'}_1^n), \dots, ({x'}_m^1, \dots, {x'}_m^n))$ in $(\cX^n)^m$ and assume, without loss of generality, that they differ in only the first entry of the first block. That is, $x_i^j = {x'}_i^j$ for all $(i,j) \ne (1,1)$.

    Now, decompose the operator $\Pi = \Pi_1 \circ \Pi_{-1}$ as follows.
    \begin{align*}
        \Pi_1(x)_i^j
            \coloneqq \begin{cases}
                    \Pi(x)_i^j & \text{if } j = 1 \\
                    x_i^j & \text{otherwise}
               \end{cases}
        ~~~\text{and}~~~
        \Pi_{-1}(x)_i^j
            \coloneqq \begin{cases}
                    x_i^j      & \text{if } j = 1 \\
                    \Pi(x)_i^j & \text{otherwise}
               \end{cases}
    \end{align*}
    In other words, $\Pi_1$ applies the permutation $\pi_1$ to the first row and leaves the remaining $n-1$ rows fixed, whereas $\Pi_{-1}$ applies the permutations $\pi_2, \dots, \pi_n$ to every row except the first.

    We claim that
    \begin{align}
        (L_m \otimes \dots \otimes L_1 \circ \Pi_1)(x) \sim_{\eps',\delta'} (L_m \otimes \dots \otimes L_1 \circ \Pi_1)(x')\label{eq:shuffle_decomp_1}
    \end{align}
    with $\eps' = O(\eps_0\sqrt{\log(1/\delta)/m})$ and $\delta' = \delta_1 + O(\delta_0 m)$, as in the conclusion of Theorem~\ref{thm:amplification_by_local_shuffling}.
    
    To that end, consider the randomized function
    \begin{align*}
        R_i(y, a) \coloneqq L_i(y, (a, x_i^2, \dots, x_i^n))
    \end{align*}
    for $i \in [m]$. Since $(a, x_i^2, \dots, x_i^n)$ and $(a', x_i^2, \dots, x_i^n)$ are neighboring datasets for any $a, a' \in \cX$, $R_i$ must be $(\eps_0, \delta_0)$-DP in the parameter $a$ and hence $R_m \otimes \dots \otimes R_1 \circ S$ is $(\eps', \delta')$-DP by Theorem~\ref{thm:amplification_by_local_shuffling}. In particular, since $\hat{x} \coloneqq (x_1^1, \dots, x_m^1)$ and $\hat{x}' \coloneqq ({x'}_1^1, \dots, {x'}_m^1)$ are neighbors, $(R_m \otimes \dots \otimes R_1 \circ S)(\hat{x})$ must be $(\eps', \delta')$-indistinguishable from $(R_m \otimes \dots \otimes R_1 \circ S)(\hat{x}')$.
    
    Therefore, to prove our claim, it suffices to show that $(L_m \otimes \dots \otimes L_1 \circ \Pi_1)(x)$ is identically distributed to $(R_m \otimes \dots \otimes R_1 \circ S)(\hat{x})$, and likewise for $(L_m \otimes \dots \otimes L_1 \circ \Pi_1)(x')$ and $(R_m \otimes \dots \otimes R_1 \circ S)(\hat{x}')$. Indeed,
    \begin{align*}
        L_i(y, \Pi_1(x)_i)
            = L_i(y, (x_{\pi_1(i)}^1, x_i^2, \dots, x_i^n))
            = R_i(y, x_{\pi_1(i)}^1)
    \end{align*}
    for all $i$. So, it follows by induction that
    \begin{align*}
        (L_m \otimes \dots \otimes L_1 \circ \Pi_1)(x)
            = (R_m \otimes \dots \otimes R_1)(x_{\pi_1(1)}^1, \dots, x_{\pi_1(m)}^1)
            \stackrel{d}{=} (R_m \otimes \dots \otimes R_1 \circ S)(\hat{x}).
    \end{align*}
    Analogously, we get
    \[(L_m \otimes \dots \otimes L_1 \circ \Pi_1)(x') \stackrel{d}{=} (R_m \otimes \dots \otimes R_1 \circ S)(\hat{x}'),\]
    as desired.

    We can now leverage the decomposition $\Pi = \Pi_1 \circ \Pi_{-1}$ to prove the theorem. Fixing $\Pi_{-1}$, $\Pi_{-1}(x)$ and $\Pi_{-1}(x')$ are neighboring datasets differing only on the first element of the first block. So, by the claim that we proved above (Equivalence~\ref{eq:shuffle_decomp_1}), which used only the fact that $x$ and $x'$ differ at $x_1^1 \neq {x'}_1^1$, we have that
    \[
        (L_m \otimes \dots \otimes L_1 \circ \Pi_1)(\Pi_{-1}(x)) \sim_{\eps',\delta'} (L_m \otimes \dots \otimes L_1 \circ \Pi_1)(\Pi_{-1}(x')).
    \]
    But $\Pi_{-1}$ depends only on $\pi_2, \dots, \pi_n$ and is, thus, is independent of $L_m \otimes \dots \otimes L_1 \circ \Pi_1$. Therefore, it follows that
    \begin{align*}
        \pr{}{(L_m \otimes \dots \otimes L_1 \circ \Pi)(x) \in E}
            & = \ex{\Pi_{-1}}{\pr{}{(L_m \otimes \dots \otimes L_1 \circ \Pi_1)(\Pi_{-1}(x)) \in E \mid \Pi_{-1}}} \\
            & \leq \ex{\Pi_{-1}}{\delta' + \exp(\eps') \pr{}{(L_m \otimes \dots \otimes L_1 \circ \Pi_1)(\Pi_{-1}(x')) \in E \mid \Pi_{-1}}} \\
            & = \delta' + \exp(\eps')\pr{}{(L_m \otimes \dots \otimes L_1 \circ \Pi)(x') \in E}
    \end{align*}
    for any measurable $E$.
\end{proof}

\section{Low-Bias Estimators for General Distributions}\label{sec:low_bias}

In this section, we describe and analyze algorithms for private estimation with low or no bias.
We give three algorithms: an $(\varepsilon, 0)$-DP algorithm based on the clipped mean (Proposition~\ref{prop:eps_ub}), a $(0, \delta)$-DP algorithm based on a variant of the ``name-and-shame'' algorithm (Proposition~\ref{prop:delta_ub}), and an $(\varepsilon,\delta)$-DP algorithm obtained by combining the two (Proposition~\ref{prop:eps_delta_ub}).
By taking the best of the three resulting bounds, we get Theorem~\ref{thm:eps_delta_ub_mix}.

\subsection{Lemmata}

We will require the following technical lemmata in our analysis.

\begin{lem}\label{lem:max_moment}
    $\forall \lambda>1 ~ \forall c>0 ~ \forall t \in \R, ~~ \max\{0,t-c\} \le \frac{(\lambda-1)^{\lambda-1}}{\lambda^\lambda \cdot c^{\lambda-1}} \cdot |t|^\lambda.$
\end{lem}
\begin{proof}
    If $t=0$, the claim holds as an equality. Now, we assume that $t \in \R \setminus \{0\}$.
    Define $f : \R \setminus \{0\} \to \R$ by $f(t) = |t|^\lambda$. Then $f'(t)=\lambda \cdot |t|^{\lambda-1} \cdot \mathsf{sign}(t)$ and $f''(t) = \lambda(\lambda-1) \cdot |t|^{\lambda-2} \ge 0$ for all $t \in \R \setminus \{0\}$. Since $f$ is convex,
    \[\forall a\ge0 ~ \forall t \in \R \setminus \{0\}, ~~ f(t) \ge f(a) + f'(a) \cdot (t-a) = a^\lambda + \lambda \cdot a^{\lambda-1} \cdot (t-a) = \lambda \cdot a^{\lambda-1} \cdot \left( t - \frac{\lambda-1}{\lambda} \cdot a \right).\]
    Taking the maximum over $a=0$ and $a=\frac{c\lambda}{\lambda-1}$ and rearranging yields the result.
\end{proof}

The following lemma decomposes the mean squared error of the clipped mean into the sum of the sampling error and the (squared) population bias introduced (which is further bounded).

\begin{lem}\label{lem:clip_mse}
    Fix $\lambda>1$ and $a<b$.
    Let $P$ be a distribution with mean $\mu(P) \in (a,b)$. 
    Let $\mu_{[a,b]}(P) \coloneqq \ex{X \gets P}{\clip_{[a,b]}(X)} \in [a,b]$. 
    Let $X_1, \dots, X_n$ be independent samples from $P$.
    Then \[\ex{}{\left( \frac1n \sum_i^n \clip_{[a,b]}(X_i) - \mu(P) \right)^2} \le \frac{\ex{X \gets P}{(X-\mu(P))^2}}{n}  + (\mu_{[a,b]}(P)-\mu(P))^2\]
    and \[\left|\mu_{[a,b]}(P)-\mu(P)\right| \le \frac{(\lambda-1)^{\lambda-1}}{\lambda^\lambda} \cdot \frac{\ex{X \gets P}{|X-\mu(P)|^\lambda}}{\left(\min\{\mu(P)-a,b-\mu(P)\}\right)^{\lambda-1}}\le \frac{1}{\lambda} \cdot \frac{\ex{X \gets P}{|X-\mu(P)|^\lambda}}{\left(\min\{\mu(P)-a,b-\mu(P)\}\right)^{\lambda-1}}.\]
\end{lem}
\begin{proof}
    We have
    \begin{align*}
        \ex{}{\left( \frac1n \sum_i^n \clip_{[a,b]}(X_i) - \mu(P) \right)^2} 
        &= \ex{}{\left( \frac1n \sum_i^n \clip_{[a,b]}(X_i) - \mu_{[a,b]}(P) \right)^2} + (\mu_{[a,b]}(P)-\mu(P))^2\\
        &= \frac{1}{n^2} \sum_i^n \ex{}{\left(  \clip_{[a,b]}(X_i) - \mu_{[a,b]}(P) \right)^2} + (\mu_{[a,b]}(P)-\mu(P))^2\\
        &\le \frac{1}{n^2} \sum_i^n \ex{}{\left(  \clip_{[a,b]}(X_i) - \mu(P) \right)^2} + (\mu_{[a,b]}(P)-\mu(P))^2\\
        &\le \frac{1}{n^2} \sum_i^n \ex{}{\left(  X_i - \mu(P) \right)^2} + (\mu_{[a,b]}(P)-\mu(P))^2\\
        &= \frac{\ex{X \gets P}{(X-\mu(P))^2}}{n}  + (\mu_{[a,b]}(P)-\mu(P))^2.
    \end{align*}
    The first inequality follows from the fact that $\ex{X \gets P}{(X-\mu(P))^2} = \inf_{u \in \R} \ex{X \gets P}{(X-u)^2}$.
    The second inequality follows from the fact that $\mu(P) \in [a,b]$ and, hence, $(\clip_{[a,b]}(x)-\mu(P))^2 \le (x-\mu(P))^2$ for all $x \in \R$.

    It remains to bound $\mu_{[a,b]}(P)-\mu(P)$. We have
    \begin{align*}
        \mu_{[a,b]}(P)-\mu(P) &= \ex{X \gets P}{\clip_{[a,b]}(X)-X}\\
        &= \ex{X \gets P}{\mathbb{I}[X>b](b-X) + \mathbb{I}[X<a](a-X)}\\
        &= \ex{X \gets P}{\max\{a-X,0\}} - \ex{X \gets P}{\max\{X-b,0\}}.
    \end{align*}
    By Lemma~\ref{lem:max_moment},
    \begin{align*}
        0 \le \ex{X \gets P}{\max\{X-b,0\}} &= \ex{X \gets P}{\max\{(X-\mu(P))-(b-\mu(P)),0\}} \\
        &\le \ex{X \gets P}{\frac{(\lambda-1)^{\lambda-1}}{\lambda^\lambda \cdot (b-\mu(P))^{\lambda-1}} \cdot |X-\mu(P)|^\lambda}.
    \end{align*}
    Similarly,
    \[0 \le \ex{X \gets P}{\max\{a-X,0\}} \le \frac{(\lambda-1)^{\lambda-1}}{\lambda^\lambda \cdot (\mu(P)-a)^{\lambda-1}} \cdot \ex{X \gets P}{|X-\mu(P)|^\lambda}.\] 
    Thus,
    \[\frac{(\lambda-1)^{\lambda-1}}{\lambda^\lambda \cdot (b-\mu(P))^{\lambda-1}} \cdot \ex{X \gets P}{|X-\mu(P)|^\lambda} \le \mu_{[a,b]}(P)-\mu(P) \le \frac{(\lambda-1)^{\lambda-1}}{\lambda^\lambda \cdot (\mu(P)-a)^{\lambda-1}} \cdot \ex{X \gets P}{|X-\mu(P)|^\lambda}.\]
    Finally, note that $\frac{(\lambda-1)^{\lambda-1}}{\lambda^\lambda} = \left(1-\frac{1}{\lambda}\right)^{\lambda-1} \cdot \frac{1}{\lambda} \le \frac{\exp(-1+1/\lambda)}{\lambda} \le \frac{1}{\lambda}$.
\end{proof}

Finally, we will also require standard properties of the well-known Laplace mechanism.

\begin{defn}[Sensitivity]
    Let $f : \cX^n \to \R$ be a function, its \emph{sensitivity} is
    $$\Delta_f \coloneqq \sup_{x \sim x' \in \cX^n} \abs{f(x) - f(x')}.$$
\end{defn}

\begin{lem}[Laplace Mechanism]\label{lem:laplacedp}
    Let $f : \cX^n \to \R$ be a function
    with sensitivity $\Delta_f$.
    Then, denoting by $\Lap(b)$ the Laplace distribution with location $0$ and scale parameter $b$, the Laplace mechanism
    $M(x) \coloneqq f(x) + \Lap(\Delta_f/ \eps)$
    satisfies $\eps$-DP. Furthermore,
    $$\pr{}{\abs{M(x) - f(x)} \geq \frac{\Delta_f\cdot\ln(1/\beta)}{\eps}} \leq \beta.$$
\end{lem}

\subsection{Algorithms}

We first have a positive result based on clipping and adding noise, which satisfies pure DP.
The clipped and noised mean is folklore in differential privacy.
Analyzing such a procedure with bounded moments has been done in a few works~\citep{DuchiJW13,BarberD14,KamathSU20}.
These works generally set algorithm parameters to achieve a prescribed bias, towards the goal of minimizing the overall error.
As our goal is to explicitly quantify the bias, we leave it as a free variable.

\begin{prop}[$\eps$-DP Algorithm]\label{prop:eps_ub}
    Fix any $\eps,\bias>0$, $a<b$, $\lambda \ge 2$, and $n \in \mathbb{N}$.
    Consider the following $\eps$-DP mechanism $M : \mathbb{R}^n \to \mathbb{R}$:
        \[M(x) \coloneqq \left(\frac1n \sum_i^n \clip_{[\hat a,\hat b]}(x_i) \right) + \Lap\left(\frac{\hat b-\hat a}{\eps n}\right),\]
    where $\hat a \coloneqq a-\bias^{-1/(\lambda-1)}$ and $\hat b \coloneqq b+\bias^{-1/(\lambda-1)}$.
    Then, for any distribution $P$ on $\mathbb{R}$ with $\mu(P) \coloneqq \ex{X \gets P}{X} \in [a,b]$ and $\ex{X \gets P}{|X - \mu(P)|^\lambda} \le 1$, we have
    \begin{align*}
        \left| \ex{X \gets P^n, M}{M(X)}-\mu(P) \right| &\le \bias,\\
        \ex{X \gets P^n, M}{(M(X)-\mu(P))^2} &\le \frac{1}{n} + \beta^2 + \frac{2}{\eps^2 n^2} \left( b-a + \frac{2}{\bias^{1/(\lambda-1)}} \right)^2.
    \end{align*}
\end{prop}

\begin{proof}[Proof of Proposition~\ref{prop:delta_ub}]
    Since $A$ satisfies local $(0,\delta)$-DP, $M$ satisfies $(0,\delta)$-DP.
    Since $A$ is unbiased (i.e., $\forall x \in \mathbb{R} ~~ \ex{A}{A(x)}=x$), so is $M$.
    Finally, we calculate the mean squared error:
    \begin{align*}
        \ex{X \gets P^n, M}{(M(X)-\mu(P))^2}
        &= \frac1n \ex{X \gets P, A}{(A(X)-\mu(P))^2}\\
        &= \frac{\ex{X \gets P, A}{A(X)^2}-\mu(P)^2}{n}\\
        &= \frac{\ex{X \gets P}{0+\delta \cdot (X/\delta)^2}-\mu(P)^2}{n}\\
        &= \frac{\ex{X \gets P}{X^2} - \delta \cdot \mu(P)^2}{\delta \cdot n}\\
        &= \frac{\ex{X \gets P}{(X-\mu(P))^2} + (1 - \delta) \cdot \mu(P)^2}{\delta \cdot n}.
    \end{align*}
    We have the required result.
\end{proof}

Next, we give an algorithm based on the folklore ``name-and-shame'' procedure, which is $(0, \delta)$-DP.
The name-and-shame procedure is generally phrased as randomly selecting a point from a dataset and outputting it, sans any further privacy protection.
It is commonly used as an illustration of which values of $\delta$ are meaningful when it comes to informal uses of the word ``privacy'', and not as a serious algorithm.
However, we note that such a procedure gives an exactly unbiased estimate of the mean, which the previous $(\varepsilon,0)$-DP was unable to do. 
We thus use it to design an unbiased algorithm for mean estimation, albeit at a high price in the dependence on $\delta$, which we recall is usually chosen to be very small.

\begin{prop}[$(0,\delta)$-DP Algorithm]\label{prop:delta_ub}
    Fix any $\delta \in (0,1]$ and $n \in \mathbb{N}$. 
    Consider the following $(0,\delta)$-DP algorithm $M : \mathbb{R}^n \to \mathbb{R}$: $M(x) = \tfrac1n \sum_i^n A(x_i)$ where where each instantiation of $A(x_i)$ is independent and
    $A : \mathbb{R} \to \mathbb{R}$ is the randomized algorithm
    \[
        A(x) \coloneq
        \begin{cases}
            0 & \textrm{with probability } 1-\delta\\
            \frac{x}{\delta} & \textrm{with probability } \delta
        \end{cases}.
    \]    
    $M$ satisfies the following bias and accuracy properties.
    For any distribution $P$ on $\mathbb{R}$,
    \begin{align*}
        &\ex{X \gets P^n, M}{M(X)}=\mu(P), \tag{i.e., $M$ is unbiased}\\
        &\ex{X \gets P^n, M}{(M(X)-\mu(P))^2} = \frac{\ex{X \gets P}{(X-\mu(P))^2} + (1 - \delta) \cdot \mu(P)^2}{\delta \cdot n}.
    \end{align*}
\end{prop}

\begin{proof}[Proof of Proposition~\ref{prop:delta_ub}]
    Since $A$ satisfies local $(0,\delta)$-DP, $M$ satisfies $(0,\delta)$-DP.
    Since $A$ is unbiased (i.e., $\forall x \in \mathbb{R} ~~ \ex{A}{A(x)}=x$), so is $M$.
    Finally, we calculate the mean squared error:
    \begin{align*}
        \ex{X \gets P^n, M}{(M(X)-\mu(P))^2}
        &= \frac1n \ex{X \gets P, A}{(A(X)-\mu(P))^2}\\
        &= \frac{\ex{X \gets P, A}{A(X)^2}-\mu(P)^2}{n}\\
        &= \frac{\ex{X \gets P}{0+\delta \cdot (X/\delta)^2}-\mu(P)^2}{n}\\
        &= \frac{\ex{X \gets P}{X^2} - \delta \cdot \mu(P)^2}{\delta \cdot n}\\
        &= \frac{\ex{X \gets P}{(X-\mu(P))^2} + (1 - \delta) \cdot \mu(P)^2}{\delta \cdot n}.
    \end{align*}
    We have the required result.
\end{proof}

We can combine both of these methods into a new algorithm for $(\varepsilon,\delta)$-DP mean estimation.
Essentially, it decomposes a sample into non-tail and tail components, releasing the former via $(\varepsilon,0)$-DP clip-and-noise, and the latter via $(0,\delta)$-DP name-and-shame.
Note that we must consider a higher moment in our assumption on the distribution $P$.

\begin{prop}[$(\eps,\delta)$-DP Algorithm]\label{prop:eps_delta_ub}
    Fix any $\eps>0$, $\delta \in (0,1]$, $\psi>0$, $\lambda>2$, $a<b$, and $n \in \mathbb{N}$.
    Consider the following $(\eps,\delta)$-DP algorithm $M : \mathbb{R}^n \to \mathbb{R}$:
    \[
        M(x) \coloneqq
            \underbrace{
                \left(\frac1n \sum_i^n \clip_{[\hat a,\hat b]}(x_i) \right) +
                \Lap\left(\frac{\hat b-\hat a}{n\eps}\right)
            }_{\text{$(\eps,0)$-DP}} +
            \underbrace{
                \left(\frac1n \sum_i^n A\left(x_i - \clip_{[\hat a,\hat b]}(x_i)\right)\right)
            }_{\text{$(0,\delta)$-DP}},
    \]
    where $c \coloneqq \left( \frac{n\eps^2\psi^{\lambda}(\lambda-2)}{4 \lambda^2 \delta} \right)^{1/\lambda}$, $\hat a \coloneqq a - c$, and $\hat b \coloneqq b + c$, and the Laplace noise and all instantiations of $A$ (defined as in Proposition~\ref{prop:delta_ub}) are independent.
    $M$ satisfies the following bias and accuracy properties.
    For any distribution $P$ on $\mathbb{R}$ with $\mu(P) \coloneqq \ex{X \gets P}{X} \in [a,b]$ and $\ex{X \gets P}{(X-\mu(P))^2}\le1$ and $\ex{X \gets P}{|X-\mu(P)|^{\lambda}}\le\psi^{\lambda}$, we have
    \begin{align*}
        \ex{X \gets P^n, M}{M(X)}=\mu(P), \tag{i.e., $M$ is unbiased}\\
        \ex{X \gets P^n, M}{(M(X)-\mu(P))^2} &\le \frac2n + \frac{4(b-a)^2}{n^2\eps^2} + \frac{24\psi^2}{n^2\eps^2} \cdot \left( \frac{n\eps^2}{4\lambda\delta} \right)^{2/\lambda}.
    \end{align*}
\end{prop}


\begin{proof}[Proof of Proposition~\ref{prop:eps_delta_ub}]
    Since $\ex{A}{A(x)} = x$ for any $x \in \R$, the bias of the $(\eps, 0)$-DP and the $(0, \delta)$-DP components cancel and thus $M$ is unbiased -- i.e. $\forall x \in \mathbb{R}^n ~~ \ex{M}{M(x)}=\tfrac1n \sum_i^n x_i$.
    
    By composition and post-processing, $M$ satisfies $(\eps,\delta)$-DP.

    Now, we bound the mean squared error.
    Define $\mu_{[\hat a,\hat b]}(P) \coloneqq \ex{X \gets P}{\clip_{[\hat a,\hat b]}(X)}$. We have:
    \begin{align}
        \ex{X \gets P^n, M}{\left(M(X)-\mu(P)\right)^2} &= \ex{X \gets P^n \atop \xi \gets \Lap\left(\frac{\hat b - \hat a}{n\eps}\right), A}{\left( \begin{array}{c}\frac1n \sum_i^n \clip_{[\hat a, \hat b]}(X_i) + \xi \\ + \frac1n \sum_i^n A\left(X_i - \clip_{[\hat a, \hat b]}(X_i)\right) - \mu(P) \end{array} \right)^2} \nonumber\\
            &= \ex{X \gets P^n, A}{\left( \begin{array}{c} \frac1n \sum_i^n \clip_{[\hat a, \hat b]}(X_i) - \mu_{[\hat a,\hat b]}(P)\\ + \frac1n \sum_i^n A\left(X_i - \clip_{[\hat a, \hat b]}(X_i)\right) - (\mu(P) - \mu_{[\hat a,\hat b]}(P)) \end{array} \right)^2} \nonumber\\ &~~~~~~~~~~+ \ex{\xi \gets \Lap\left(\frac{\hat b - \hat a}{n\eps}\right)}{\xi^2} \nonumber\\
            &\le \frac2n \cdot \ex{X \gets P}{\left(\clip_{[\hat a,\hat b]}(X)-\mu_{[\hat a,\hat b]}(P)\right)^2} + 2\left(\frac{\hat b-\hat a}{n\eps}\right)^2 \nonumber\\ &~~~~~~~~~~+ \frac2n \cdot \ex{X \gets P, A}{\left(A\left(X-\clip_{[\hat a,\hat b]}(X)\right)-\left(\mu(P)-\mu_{[\hat a,\hat b]}(P)\right)\right)^2} \label{eq:mse_general}.
    \end{align}
    The final inequality uses the fact that for independent mean-zero random variables $U$ and $V$, we have $\ex{}{(U+V)^2} = \ex{}{U^2} + \ex{}{V^2}$.
    For the terms that are not independent, we apply the inequality $\ex{}{(U+V)^2} \le 2\ex{}{U^2}+2\ex{}{V^2}$.
    
    Since $\mu(P) \coloneqq \ex{X \gets P}{X} \in [a,b] \subset [\hat a, \hat b]$, we have
    \[\ex{X \gets P}{\left(\clip_{[\hat a,\hat b]}(X)-\mu_{[\hat a,\hat b]}(P)\right)^2} \le \ex{X \gets P}{\left(\clip_{[\hat a,\hat b]}(X)-\mu(P)\right)^2} \le \ex{X \gets P}{\left( X - \mu(P) \right)^2} \le 1.\]
    
    Finally we bound the last term:
    \begin{align*}
        \ex{X \gets P, A}{\left(A\left(X-\clip_{[a,b]}(X)\right)-\left(\mu(P)-\mu_{[\hat a,\hat b]}(P)\right)\right)^2}
        &\le \ex{X \gets P, A}{\left(A\left(X-\clip_{[\hat a, \hat b]}(X)\right)\right)^2}\\
        &= (1-\delta) \cdot 0 + \delta \cdot \ex{X \gets P}{\left(\frac1\delta\left(X-\clip_{[\hat a,\hat b]}(X)\right)\right)^2}\\
        &= \frac1\delta \cdot \ex{X \gets P}{\left(X-\clip_{[\hat a,\hat b]}(X)\right)^2}\\
        &= \frac1\delta \cdot \ex{X \gets P}{\left(\begin{array}{c}(X-\mu(P)) - \\ \clip_{[\hat a-\mu(P),\hat b-\mu(P)]}(X-\mu(P))\end{array}\right)^2}\\
        &\le \frac1\delta \cdot \ex{X \gets P}{\left((X-\mu(P))-\clip_{[-c,c]}(X-\mu(P))\right)^2}\\
        &= \frac1\delta \cdot \ex{X \gets P}{\left(\max\left\{0, |X-\mu(P)|-c\right\}\right)^2},
    \end{align*}
    where the final inequality holds because $c = \hat b - b \le \hat b-\mu(P)$ and $c = a - \hat a \le \mu(P)-a$.
    By Lemma~\ref{lem:max_moment}, \[\max\left\{0, |X-\mu(P)|-c\right\} \le \frac{1}{\lambda/2} \cdot \frac{|X-\mu(P)|^{\lambda/2}}{c^{\lambda/2-1}}.\]
    Thus,
    \[
        \ex{X \gets P}{\left(\max\left\{0, |X-\mu(P)|-c\right\}\right)^2} \le \left( \frac{1}{(\lambda/2) \cdot c^{\lambda/2-1}} \right)^2 \cdot \ex{X \gets P}{|X-\mu(P)|^{\lambda}} \le \frac{4}{\lambda^2} \cdot \frac{\psi^{\lambda}}{c^{\lambda-2}}.
    \]

    Now, we set parameters and assemble the bound from Inequality~\ref{eq:mse_general}:
    \begin{align*}
        \ex{X \gets P^n, M}{\left(M(X)-\mu(P)\right)^2}
        &\le \frac2n + 2 \left( \frac{b-a+2c}{n\eps} \right)^2 + \left(\frac2n \cdot \frac1\delta \cdot \frac{4}{\lambda^2} \cdot \frac{\psi^{\lambda}}{c^{\lambda-2}}\right)\\
        &\le \frac2n + \frac{4(b-a)^2}{n^2\eps^2} + \frac{16}{n^2\eps^2} \cdot c^2 + \frac{8\psi^{\lambda}}{n\delta\lambda^2} \cdot (c^2)^{1-\lambda/2}\\
        &= \frac2n + \frac{4(b-a)^2}{n^2\eps^2} + \frac{16\psi^2}{n^2\eps^2} \cdot \left( \frac{n\eps^2}{4\lambda\delta} \right)^{2/\lambda} \cdot \left(\frac{\lambda}{\lambda-2}\right)^{1-2/\lambda} \\
        &\le \frac2n + \frac{4(b-a)^2}{n^2\eps^2} + \frac{24\psi^2}{n^2\eps^2} \cdot \left( \frac{n\eps^2}{4\lambda\delta} \right)^{2/\lambda} ,
    \end{align*}
    where the final equality follows from setting $c = \left( \frac{n\eps^2\psi^{\lambda}(\lambda-2)}{4 \lambda^2 \delta} \right)^{1/\lambda}$ to minimize the expression, and the final inequality follows from the fact that $\left(\frac{\lambda}{\lambda-2}\right)^{1-2/\lambda} = (1-2/\lambda)^{-1+2/\lambda} \le \exp(\exp(-1)) < \frac32$.
\end{proof}

Combining Propositions~\ref{prop:eps_ub},~\ref{prop:delta_ub}, and~\ref{prop:eps_delta_ub} yields Theorem~\ref{thm:eps_delta_ub_mix}.

\section{Unbiased Estimators for Symmetric Distributions}
\label{sec:symmetric}

We now present our unbiased private mean
estimation algorithm for \emph{symmetric} distributions
over $\R$ that are weakly concentrated, i.e.,
those that have a bounded second moment.

\begin{defn}[Symmetric Distribution]\label{def:symmetric}
    We say that a distribution $P$ on $\mathbb{R}$ is \emph{symmetric} if there exists some $\mu(P) \in \mathbb{R}$, such that $\forall x \in \mathbb{R}, ~ \pr{X \gets P}{X-\mu(P) \le x} = \pr{X \gets P}{\mu(P)-X \le x}$.
    The value $\mu(P)$ is called the \emph{center} of the distribution $P$.
\end{defn}
Note that the center of the distribution is unique and coincides with the mean and the median (whenever these two quantities are well-defined).

Our algorithm is based on the approach of \citet{KarwaV18}, but with some modifications to ensure unbiasedness. First, we obtain a coarse estimate of the mean, and then we use this coarse estimate to perform clipping to obtain a precise estimate via noise addition. 
The key observation is that, if the coarse estimate we use for clipping is unbiased and symmetric (and also independent from the data used in the second step), then the clipping does not introduce bias.
We obtain the coarse estimate via a DP histogram, where each bucket in the histogram is an interval on the real line. To ensure that this is unbiased and symmetric, we simply need to apply a random offset to the bucket intervals.  

\subsection{Coarse Unbiased Estimation}


Our coarse estimator (Algorithm~\ref{alg:coarse}) is similar to that of \cite{KarwaV18}.
The key modification to ensure unbiasedness is adding a random offset to the histogram bins.
We define $\round : \mathbb{R} \to \Z$ to be the function that rounds real numbers to the nearest integer, i.e., for any $x \in \R$, we have $x \in [\round(x)-1/2,\round(x)+1/2)$.

\begin{algorithm}[h!]
  \caption{Unbiased DP Coarse Estimator $\coarse_{\eps, \delta}(x)$}
  \label{alg:coarse}
  \KwIn{Dataset $x = (x_1, \dots, x_n) \in \R^n$.}
  \KwOut{Estimate $\wt{\mu} \in \R \cup \{\bot\}$.}
  \vspace{5pt}
  
  Let $T$ be uniform on the interval $[-1/2,+1/2]$.\\

  Let $K = \{\round(x_i-T): i \in [n]\} \subset \mathbb{Z}$.\\

  For each $k \in K$, sample $\xi_k \gets \Lap(2/\varepsilon)$ independently.\\

  If $\max_{k \in K} |\{i \in [n] : \round(x_i-T)=k\}|+\xi_k \le 2+\frac{2\log(1/\delta)}{\varepsilon}$, output $\bot$.

  Otherwise, output $T + \argmax_{k \in K} |\{i \in [n] : \round(x_i-T)=k\}|+\xi_k$.
\end{algorithm}

First, it is easy to verify that Algorithm~\ref{alg:coarse} is private. Similar to previous work (see, e.g.,~\cite{Vadhan17}), privacy follows from the privacy of the stable histogram algorithm, plus post-processing via argmax.

\begin{prop}\label{prop:coarse_dp}
    Algorithm~\ref{alg:coarse} ($\coarse_{\varepsilon,\delta}$) satisfies $(\varepsilon,\delta)$-DP.
\end{prop}

We recall the group privacy property of differential privacy, which quantifies the privacy guaranteed by a DP algorithm for a group of individuals within a dataset.

\begin{lem}[Group Privacy \citep{DworkMNS06, DworkR14}]\label{lem:group_privacy}
  Let $A: \cX^n \rightarrow \cY$ be $(\eps,\delta)$-DP. Then for any integer $k \in \{0, \dots, n\}$, measurable subset $Y \subseteq \cY$,
  and pairs of datasets $x,x' \in \cX^n$ differing in $k$ elements,
  \[\pr{}{A(x) \in Y} \leq \exp(k\eps)\cdot\pr{}{A(x') \in Y}
    + \frac{\exp(k\eps)-1}{\exp(\eps)-1} \cdot \delta.\]
\end{lem}

\begin{proof}[Proof of Proposition~\ref{prop:coarse_dp}]
    We will prove that $\coarse_{\varepsilon,\delta}$ satisfies $(\eps/2,\delta/2\exp(\eps/2))$-DP with respect to addition or removal of one element of the dataset. By group privacy (Lemma~\ref{lem:group_privacy}), this implies $(\eps,\delta)$-DP for replacement of an element.

    Consider a fixed pair of datasets $x$ and $x'=x_{-i_*}$, where $x'$ is $x$ with $x_{i_*}$ removed for some $i_* \in [n]$.
    For the privacy analysis, we also consider the offset $T$ to be fixed -- i.e., $T$ is not needed to ensure privacy.
    
    By post-processing, we can consider an algorithm that outputs more information. Specifically, we can assume that for each $k \in \mathbb{Z}$, the algorithm outputs
    \[\nu_k(x) \coloneqq \left\{\begin{array}{cl} \max\left\{  |\{ i \in [n] : \round(x_i-T)=k\}|+\xi_k - 2 - \frac{2\log(1/\delta)}{\eps}, 0 \right\} & \text{ if } k \in K \\ 0 & \text{ if } k \in \mathbb{Z} \setminus K \end{array} \right\}.\]
    Note that only finitely many of these $\nu_k(x)$ values will be nonzero, so the algorithm can output a compressed version of this infinite vector of values. We can obtain the true output of $\coarse_{\varepsilon,\delta}$ by taking the argmax of this vector or outputting $\bot$ if this vector is all zeros.
    The advantage of this perspective is that each $\nu_k(x)$ is independent, as it depends only on the noise $\xi_k$ (the input $x$ and offset $T$ are fixed).
    
     The output distributions on the neighboring inputs are the same except for one $\nu_k(x)\not\stackrel{d}{=} \nu_k(x')$, namely $k = \round(x_{i_*}-T)$. Thus, we must simply show that this value satisfies $(\eps/2,\delta/2\exp(\eps/2))$-DP. That is, we must show $\nu_k(x) \sim_{\eps/2,\delta/2\exp(\eps/2)} \nu_k(x')$, where $\nu_k(x)$ and $\nu_k(x')$ denote the relevant random variables on the two different inputs. 
    There are two cases to consider: $|\{ i \in [n] : \round(x_i-T)=k\}|=1$ and $|\{ i \in [n] : \round(x_i-T)=k\}| \ge 2$. (Note that $|\{ i \in [n] : \round(x_i-T)=k\}|=0$ is ruled out because $k = \round(x_{i_*}-T)$.)
    
    Suppose $|\{ i \in [n] : \round(x_i-T)=k\}|=1$. Then $\nu_k(x') = 0$ deterministically. Therefore, it suffices to prove that $\pr{}{\nu_k(x)=0} \ge 1-\delta/2\exp(\eps/2)$. We have 
    \begin{align*}
        \pr{}{\nu_k(x) \ne 0} &= \pr{}{|\{ i \in [n] : \round(x_i-T)=k\}|+\xi_k - 2 - \frac{2\log(1/\delta)}{\eps} > 0} \\
        &= \pr{}{1 + \xi_k -2 - \frac{2\log(1/\delta)}{\eps} > 0}\\
        &= \pr{}{\xi_k > 1 + \frac{2\log(1/\delta)}{\eps}}\\
        &= \frac12 \exp\left(-\frac{\eps}{2} \cdot \left( 1 + \frac{2\log(1/\delta)}{\eps} \right) \right)\\
        &= \frac{\delta}{2\exp(\eps/2)},
    \end{align*}
    where the penultimate equality follows from the fact that $\xi_k \gets \Lap(2/\eps)$ (Lemma~\ref{lem:laplacedp}).

    Now suppose $|\{ i \in [n] : \round(x_i-T)=k\}|\ge 2$. Then $\nu_k(x)$ and $\nu_k(x')$ are post-processings of $|\{ i \in [n] : \round(x_i-T)=k\}|+\xi_k$ and, respectively, $|\{ i \in [n] : \round(x_i-T)=k\}|-1+\xi_k$. Thus, by the properties of Laplace noise, we have $ \exp(-\eps/2) \pr{}{\nu_k(x') \in S} \le \pr{}{\nu_k(x) \in S} \le \exp(\eps/2) \pr{}{\nu_k(x') \in S}$ for all $S$, as required.
\end{proof}

Now we turn to the utility analysis, which consists of two parts.
First, conditioned on not outputting $\bot$, the estimate is symmetric and unbiased (Proposition~\ref{prop:coarse_symmetric}). 
Second, we show that the probability of outputting $\bot$ is low for appropriately concentrated distributions, and that the MSE is bounded, as well (Proposition~\ref{prop:coarse_success}).

\begin{prop}[Conditional Symmetry of $\coarse$]\label{prop:coarse_symmetric}
    Let $P$ be a symmetric distribution with center $\mu(P)$.
    Let $X_1, \cdots, X_n \in \mathbb{R}$ be independent samples from $P$. Let $\wt\mu = \coarse_{\varepsilon,\delta}(X_1, \cdots, X_n)$. Let $Q$ be the distribution of $\wt\mu$ conditioned on $\wt\mu \ne \bot$. Then $Q$ is symmetric with the same center as $P$ -- i.e., $\mu(P)=\mu(Q)$.
\end{prop}

To show that our estimator preserves symmetry, note that, unlike the static histogram bucket approach of prior work, the introduction of the uniformly random offset $T \in [\pm 1/2]$ in Algorithm~\ref{alg:coarse} endows $\coarse_{\eps, \delta}$ with equivariance under translation.

\begin{lem}\label{lem:translation_equivariance}
  For any $x = (x_1,\dots,x_n) \in \R^n$ and $c \in \R$, we have
  \[
    \coarse_{\eps, \delta}(x + c) \stackrel{d}{=} \coarse_{\eps, \delta}(x) + c
  \]
  where $\bot + c \coloneqq \bot$ and $x + c \coloneqq (x_1+c,\dots,x_n+c)$.
\end{lem}

\begin{proof}[Proof of Lemma~\ref{lem:translation_equivariance}]
  It will be easier to proceed by rewriting Algorithm~\ref{alg:coarse} in a non-algorithmic form. To that end, we define the following notations.
  \begin{itemize}
    \item For $r \in \R$, set $\wh{p}_r(x) \coloneqq \frac{1}{n}\abs{i \in [n] : x_i \in [r \pm 1/2)}$ and sample $T \gets \mathcal{U}[\pm 1/2]$, $\wt{p}_r(x) \gets \wh{p}_r(x) + \Lap(0, 2/(\eps n))$ such that $T$ and $\{\wt{p}_r(x)\}_{r \in \R}$ are all mutually independent.
    \item For $t \in \R$ and $S \subseteq \R$, we define $S+t \coloneqq \{s+t : s \in S\}$, and denote by $S + T$ the distribution over the set of sets $\{S+t : t \in [\pm 1/2]\}$ induced by the randomness of $T$.
    \item Set $R(x) \coloneqq \{r \in \Z + T : \wh{p}_r(x) > 0\}$ and put $R^*(x) \coloneqq \argmax_{r \in R(x)} \wt{p}_r(x)$, provided there is an $r \in R(x)$ for which $\wt{p}_r(x) > \frac{2\ln(2/\delta)}{\eps n} + \frac{2}{n} \eqqcolon \eta$, otherwise $R^*(x) \coloneqq \bot$.
  \end{itemize}
  Essentially, we have reparameterized the terms of $\coarse_{\eps, \delta}(x)$ so that $R(x) = K + T$ and $\wh{p}_{k + T}(x) = \frac{1}{n}|\{i \in [n] : \round(x_i - T) = k\}|$ hold, so it follows that $R^*(x) \stackrel{d}{=} \coarse_{\eps, \delta}(x)$.

  Now, notice that
  \begin{align*}
    \wh{p}_r(x + c)
      = \frac{1}{n}\abs{\{i \in [n] : x_i + c \in [r \pm 1/2)\}}
      = \frac{1}{n}\abs{\{i \in [n] : x_i \in [r - c \pm 1/2)\}}
      = \wh{p}_{r - c}(x),
  \end{align*}
  so in particular, we have that $\wt{p}_r(x + c)$ is identically distributed to $\wt{p}_{r - c}(x)$ for any $r \in \R$. Moreover, $\Z + T$ is identically distributed to $\Z + T - c$, so it follows that
  \begin{align*}
    R(x + c)
      & = \{r \in \Z + T : \wh{p}_r(x + c) > 0\} \\
      & = \{r \in \Z + T : \wh{p}_{r - c}(x) > 0\} \\
      & = \{r \in \Z + T - c : \wh{p}_r(x) > 0\} + c \\
      & \stackrel{d}{=} \{r \in \Z + T : \wh{p}_r(x) > 0\} + c \\
      & = R(x) + c.
  \end{align*}
  As $T$ and the Laplace noise were all sampled in a mutually independent manner, these distributional equivalences hold jointly, i.e., $((\wt{p}_r(x + c))_{r \in \R}, R(x + c)) \stackrel{d}{=} ((\wt{p}_{r - c}(x))_{r \in \R}, R(x) + c)$. Hence,
  \begin{align*}
    R^*(x + c)
      & = \begin{cases}
            \argmax_{r \in R(x + c)} \wt{p}_r(x + c) & \text{if } \exists r \in R(x + c), \wt{p}_r(x + c) > \eta \\
            \bot & \text{otherwise}.
          \end{cases} \\
      & \stackrel{d}{=}
        \begin{cases}
            \argmax_{r \in R(x) + c} \wt{p}_{r - c}(x) & \text{if } \exists r \in R(x) + c, \wt{p}_{r - c}(x) > \eta \\
            \bot & \text{otherwise}.
          \end{cases} \\
      & = \begin{cases}
            \argmax_{r' \in R(x)} \wt{p}_{r'}(x) + c & \text{if } \exists r' \in R(x), \wt{p}_{r'}(x) > \eta \\
            \bot & \text{otherwise}.
          \end{cases} \tag{$r = r' + c$} \\
      & = R^*(x) + c \tag{$\bot = \bot + c$}.
  \end{align*}
  The equivalence of $R^*(x)$ and $\coarse_{\eps,\delta}(x)$ gives us the desired result.
\end{proof}

The random offsets also endow our estimator with equivariance under reflection.

\begin{lem}\label{lem:symmetry}
    For any $x = (x_1,\dots,x_n) \in \R^n$, we have
    \[
        \coarse_{\eps, \delta}(-x) \stackrel{d}{=} -\coarse_{\eps, \delta}(x)
    \]
    where $-\bot \coloneqq \bot$ and $-x \coloneqq (-x_1,\dots,-x_n)$.
\end{lem}

\begin{proof}[Proof of Lemma \ref{lem:symmetry}]
  Recall the notation from the proof of Lemma~\ref{lem:translation_equivariance}. Then, we have that
  \begin{align*} 
    \wh{p}_r(-x)
      = \frac{1}{n}\abs{\{i \in [n] : -x_i \in [r \pm 1/2)\}}
      = \frac{1}{n}\abs{\{i \in [n] : x_i \in [-r \pm 1/2)\}}
      = \wh{p}_{-r}(x),
  \end{align*}
  so it follows that for any $r \in \R$, $\wt{p}_r(-x)$ is identically distributed to $\wt{p}_{-r}(x)$. Moreover,
  \begin{align*}
    R(-x)
      & = \{r \in \Z + T : \wh{p}_r(-x) > 0\} \\
      & = \{r \in \Z + T : \wh{p}_{-r}(x) > 0\} \\
      & = -\{r \in -(\Z + T) : \wh{p}_r(x) > 0\} \\
      & \stackrel{d}{=} -\{r \in \Z + T : \wh{p}_r(x) > 0\} \\
      & = -R(x).
  \end{align*}
  As $T$ and all of the Laplace noise is sampled independently, these distributional equivalences hold simultaneously, namely $((\wt{p}_r(-x))_{r \in \R}, R(-x)) \stackrel{d}{=} ((\wt{p}_{-r}(x))_{r \in \R}, -R(x))$.
  Combining these with $-\bot = \bot$, we obtain $R^*(-x) \stackrel{d}{=} -R^*(x)$ by the same argument as the one we used to prove Lemma~\ref{lem:translation_equivariance}.
\end{proof}

\begin{proof}[Proof of Proposition~\ref{prop:coarse_symmetric}]
    Due to Lemma~\ref{lem:translation_equivariance}, we may assume without loss of generality that $P$ has center $0$. Then, since $P$ is symmetric, $X \gets P^n$ is identically distributed to $-X$. So, for any $a \geq 0$,
    \begin{align*}
        \pr{X \gets P^n}{\coarse_{\eps, \delta}(X) \in [a, \infty)}
            & = \ex{X \gets P^n}{\pr{}{\coarse_{\eps, \delta}(X) \in [a, \infty) | X}} \\
            & = \ex{X \gets P^n}{\pr{}{\coarse_{\eps, \delta}(-X) \in (-\infty, -a] | X}} \tag{Lemma~\ref{lem:symmetry}} \\
            & = \pr{X \gets P^n}{\coarse_{\eps, \delta}(-X) \in (-\infty, -a]} \\
            & = \pr{X \gets P^n}{\coarse_{\eps, \delta}(X) \in (-\infty, -a]} \tag{$X \stackrel{d}{=} -X$}.
    \end{align*}
    In particular,
    \begin{align*}
        Q([a, \infty))
            = \frac{\pr{}{\wt\mu \in [a, \infty)}}{\pr{}{\wt\mu \neq \bot}}
            = \frac{\pr{}{\wt\mu \in (-\infty, -a]}}{\pr{}{\wt\mu \neq \bot}}
            = Q((-\infty, a])
    \end{align*}
    for all $a \geq 0$, so $Q$ must also be symmetric with center $0$.
\end{proof}

\begin{prop}[Accuracy of $\coarse$]\label{prop:coarse_success}
    Let $P$ be a distribution over $\R$ with mean $\mu(P)$ and variance $\ex{X \gets P}{(X-\mu(P))^2} < 1/64$ and $\ex{X \gets P}{|X-\mu(P)|^{\lambda}} \le \psi^{\lambda}$ for some $\lambda \ge 2$ and $\psi>0$.
    Let $X=(X_1,\dots,X_n)$ be independent samples from $P$, and $\wt\mu \gets \coarse_{\varepsilon,\delta}(X)$. If $n \ge 7 + \frac7\eps \log(1/\delta)$, then
    \begin{align*}
        \pr{}{\wt\mu \ne \bot \wedge |\wt\mu-\mu(P)| \le 1} &\ge 1 - \exp(-n/128) - \frac{n}{2} \exp(-n\eps/16),
        \\
        \ex{}{\mathbb{I}[\wt\mu\ne\bot]\cdot|\wt\mu-\mu(P)|^{\lambda}} &\le \frac12 + n\cdot 2^{\lambda-1} \cdot \psi^{\lambda}.
    \end{align*}
\end{prop}

In particular, for $\gamma>0$, to ensure $\pr{}{\wt\mu \ne \bot \wedge |\wt\mu-\mu(P)| \le 1} \ge 1 - \gamma$, it suffices to set \begin{equation} n \ge \max\left\{ 7 + \frac7\eps \log(1/\delta), 128 \log(2/\gamma), \frac{16}{\eps} \log( n/\gamma)  \right\} = O(\log(n/\gamma\delta)/\eps).\label{eq:n_coarse}\end{equation}

\begin{proof}[Proof of Proposition~\ref{prop:coarse_success}]
    By Lemma~\ref{lem:translation_equivariance}, we assume, without loss of generality, that $\mu(P) = 0$.

    Let $X \gets P^n$ be the input to $\coarse_{\eps,\delta}$ and let $\wt\mu \in \mathbb{R} \cup \{\bot\}$ be the output. Let $T \in [\pm\frac12]$, $K \subset \mathbb{Z}$, and $\xi_k \gets \Lap(2/\eps)$ be as in the algorithm (and define $\xi_k=0$ for $k \notin K$). For $k \in \mathbb{Z}$, define
    \[C_k \coloneqq|\{i \in [n] : \round(X_i-T)=k\}|.\]
    Recall, from Algorithm \ref{alg:coarse}, that $k \in K \iff C_k \ge 1$ and that $\wt\mu = \bot \iff \max_{k \in K} C_k+\xi_k \le 2 + \frac2\eps\log(1/\delta)$ and, otherwise, $\wt\mu=T+\argmax_{k \in K} C_k + \xi_k$.
    
    We begin by showing $\wt\mu \in [\pm 1]$ with high probabiliy. 
    
    Define \[k_+ = \round\left(\frac12-T\right) ~~~ \text{ and } ~~~ k_- = \round\left(-\frac12-T\right).\] Note that $k_+=k_-+1$ and $k_+,k_- \in (-T-1,-T+1]$. Thus, if $\argmax_{k \in K} C_k + \xi_k \in \{k_+,k_-\}$ (and $\max_{k \in K} C_k + \xi_k > 2 + \frac2\eps \log(1/\delta)$), then $\wt\mu \in \{k_++T,k_-+T\} \subset (-1,+1]$, as required.
    In other words, it suffices for us to show that, with high probability, $C_{k_+}+\xi_{k_+}$ or $C_{k_-}+\xi_{k_-}$ are large and all other $C_k+\xi_k$ values are small.

    For any $x \in \left(-\frac12,+\frac12\right)$, we have $\round(x-T) \in \{k_-,k_+\}$. Thus \[C_{k_-} + C_{k_+} \ge \sum_i^n \mathbb{I}\left[X_i \in \left(-\frac12,+\frac12\right)\right].\]
        
    Due to our assumption that $\ex{X \gets P}{(X-\mu(P))^2} < 1/64$, Chebyshev's inequality yields $\pr{X \gets P}{X \in \left(\pm\frac12\right)} \ge \frac{15}{16}$. Furthermore, by Hoeffding's inequality, we have
    $$\pr{X \gets P^n}{\sum_i^n \mathbb{I}\left[X_i \in \left(-\frac12,+\frac12\right)\right] \ge \frac{15}{16}n - s} \ge 1 - \exp(-2s^2/n)$$
    for all $s \ge 0$.
    In particular, \[\pr{X \gets P^n}{\sum_i^n \mathbb{I}\left[X_i \in \left(-\frac12,+\frac12\right)\right] \ge \frac78n} \ge 1 - \exp(-n/128).\]
    
    This means $\pr{}{C_{k_-} + C_{k_+} \ge \frac78 n} \ge 1 - \exp(-n/128)$.
    Define $k_* \coloneqq \argmax_{k \in \{k_+,k_-\}} C_k$ (breaking ties arbitrarily). If $C_{k_-} + C_{k_+} \ge \frac78 n$, then $C_{k_*} \ge \frac{7}{16}n$, while $C_k \le \frac18 n$ for all $k \notin \{k_+,k_-\}$. Thus \[\pr{}{C_{k_*} \ge \frac{7}{16}n \wedge \max_{k \in K \setminus \{k_+,k_-\}} C_k \le \frac18n} \ge 1 - \exp(-n/128).\]

    The next step is to analyze the noise. For all $k \in K$ and $r \ge 0$, $\pr{}{\xi_k \ge r} = \pr{}{\xi_k \le -r} = \frac12 \exp(-r\eps/2)$. Note that $|K| \le n$. Setting $r=n/8$ and taking a union bound over $k \in K$, we have \[\pr{}{\xi_{k_*} \ge \frac{-n}{8} \wedge \max_{k \in K \setminus \{k_+,k_-\}} \xi_k \le \frac{n}{8}} \ge 1 - \frac{n}{2} \exp(-n\eps/16).\]
    Combining the high probability bounds on the noise bound and the data, we have \[\pr{}{C_{k_*} + \xi_{k_*} \ge \frac{5}{16}n \wedge \max_{k \in K \setminus \{k_+,k_-\}} C_k +\xi_k \le \frac14n} \ge 1 - \exp(-n/128) - \frac{n}{2} \exp(-n\eps/16).\]
    Since $\frac{5}{16} n \ge 2  + \frac2\eps\log(1/\delta)$, the event $C_{k_*} + \xi_{k_*}\ge \frac{5}{16}n \wedge \max_{k \in K \setminus \{k_+,k_-\}} C_k + \xi_k \le \frac14n$ implies $\wt\mu \in \{T + k_+, T + k_-\} \subset [-1,+1]$, as required.

    Finally, we bound $\ex{}{\mathbb{I}[\wt\mu\ne\bot]\cdot|\wt\mu-\mu(P)|^{\lambda}}$.
    Observe that $\wt\mu = T+k$ for some $T \in [-1/2,+1/2]$ and $k = \round(X_i-T)$ for some $i \in [n]$. Thus, $|\wt\mu-X_i|\le 1/2$ for some $i \in [n]$ and, hence, $|\wt\mu-\mu(P)| \le 1/2 + \max_{i\in[n]} |X_i-\mu(P)|$.
    It follows that
    \begin{align*}
        \ex{}{\mathbb{I}[\wt\mu\ne\bot]\cdot|\wt\mu-\mu(P)|^\lambda} &\le \ex{}{\left(\frac12 + \max_{i\in[n]} |X_i-\mu(P)|\right)^\lambda} \\
        &\overset{\textrm{(A)}}{\le} \ex{}{\frac12 + 2^{\lambda-1} \cdot \max_{i\in[n]} |X_i-\mu(P)|^\lambda}\\
        &\overset{\textrm{(B)}}{\le} \frac12 + 2^{\lambda-1} \cdot \sum_{i \in [n]} \ex{}{|X_i-\mu(P)|^\lambda}\\
        &\le \frac12 + n \cdot 2^{\lambda-1} \cdot \psi^{\lambda},
    \end{align*}
    where Inequality A follows from the fact that $\forall p \ge 1 ~ \forall x,y\ge 0, ~~(x+y)^p \le (x^p+y^p) \cdot 2^{p-1}$, and Inequality B holds because the maximum among a set of non-negative real numbers should be at most the sum of those numbers.
    This completes our proof.
\end{proof}

\subsection{Final Algorithm}

Now, we present our main algorithm (Algorithm~\ref{alg:fine}) for unbiased mean estimation of symmetric distributions under (approximate) DP.
The idea is simple: invoke our coarse estimator (Algorithm~\ref{alg:coarse}) to get a symmetric, unbiased, rough estimate of the mean privately; then apply the standard clip-average-noise technique. 
The second step will not create any new bias because the clipping is performed around a symmetric, unbiased estimate that is independent of the data we are clipping and averaging, and the added noise has mean $0$. 
There is an additional hiccup though: the coarse estimator may fail. In this case, we fall back to a $(0,\delta)$-DP algorithm that does not require a coarse estimate, as in Proposition~\ref{prop:delta_ub}. 

\begin{algorithm}[h!]
    \caption{Unbiased DP Estimator $\fine_{\eps, \delta, c,\sigma,n_1,n_2}(x)$}
    \label{alg:fine}
  
    \KwIn{Dataset $x = (x_1, \dots, x_{n_1}, x_{n_1 + 1}, \dots, x_{n_1 + n_2}) \in \R^{n_1 + n_2}$. Parameters: Privacy: $\eps,\delta>0$. Clipping \& Scale: $c, \sigma > 0$. Dataset split: $n_1,n_2 \in \mathbb{N}$.}
    \KwOut{Estimate $\wh{\mu} \in \R$.}

  $\wt{\mu} \gets \sigma \cdot \coarse_{\eps, \delta}\left(\frac{x_1}{\sigma}, \dots, \frac{x_{n_1}}{\sigma}\right)$.\\

    \If{$\wt{\mu} = \bot$     \tcp*[f]{When the coarse estimator fails.}
}{
        Let $\xi_1,\xi_2,\cdots,\xi_{n_2} \in \{0,1\}$ be independent samples from $\mathsf{Bernoulli}(\delta)$.\\
        Let $\wh\mu = \frac{1}{n_2\delta} \sum_{i=1}^{n_2} x_{n_1+i} \cdot \xi_i$.
    }
    {\bf Else} \tcp*[f]{When the coarse estimator outputs $\wt{\mu} \in \R$.}\\

        \Indp Let $\wh{\mu} = \left(\frac{1}{n_2}\sum\limits_{i = 1}^{ n_2} \clip_{[\wt{\mu}-c, \wt{\mu}+c]}(x_{n_1 +i}) \right) + \Lap\left(\frac{2c}{n_2\eps}\right)$.

   \Indm \Return $\wh\mu$.
\end{algorithm}

The following privacy and utility guarantee is the more general version of Theorem~\ref{thm:positive_unbiased}. 

\begin{thm}[Unbiased DP Estimator]\label{thm:unbiased_final}
    Fix $\eps,\delta\in(0,1)$, $n_2 \in\mathbb{N}$, $\psi \ge 1$, and $\lambda \ge 2$.
    Set $\gamma=\delta^2$, $\sigma=10$, $c=\sigma + \psi \cdot (n_2\eps)^{1/\lambda}$, $n_1 = O(\log(n_1/\gamma\delta)/\eps)$ (as in Equation~\eqref{eq:n_coarse}), and $n=n_1+n_2$.
    Algorithm~\ref{alg:fine} ($\fine_{\eps,\delta,c,\sigma,n_1,n_2}$) satisfies $(\eps,\delta)$-DP and the following.
    Let $P$ be a symmetric distribution with center $\mu(P)$, $\ex{X \gets P}{(X-\mu(P))^2} \le 1$, and $\ex{X \gets P}{|X-\mu(P)|^\lambda}\le \psi^\lambda$.
     Let $X=(X_1, \cdots, X_{n}) \gets P^n$ and $\wt\mu$ and $\wh\mu$ be as in $\fine_{\eps,\delta,c,\sigma,n_1,n_2}(X)$. Then
    \begin{align*}
        \ex{}{\wh\mu} &= \mu(P),\\
        \ex{}{\left(\wh\mu-\mu(P)\right)^2} &\le \frac{1}{n_2} + O\left( \frac{\psi^2}{(n_2\eps)^{2-2/\lambda}} + \delta \cdot \frac{\mu(P)^2}{n_2} + \delta^{2-4/\lambda} \cdot (n_1+n_2\eps)^{2/\lambda} \cdot \psi^2\right),\\
        \pr{}{\wt\mu\ne\bot} &\ge 1-\gamma = 1-\delta^2,\\
        \ex{}{\wh\mu \mid \wt\mu\ne\bot} &= \ex{}{\wt\mu \mid \wt\mu\ne\bot} = \mu(P),\\
        \ex{}{\left(\wh\mu-\mu(P)\right)^2 \mid \wt\mu \ne \bot} &\le \frac{1}{n_2} + O\left( \frac{\psi^2}{(n_2\eps)^{2-2/\lambda}} + \delta^{2-4/\lambda} \cdot (n_1+n_2\eps)^{2/(\lambda-1)} \cdot \psi^2\right).
    \end{align*}
\end{thm}

In particular, we can apply Theorem~\ref{thm:unbiased_final} to (sub-)Gaussians. Using the bound $\ex{X \gets \cN(0,1)}{|X|^\lambda} = O(\sqrt{\lambda})^\lambda$ and setting $\lambda = \Theta(\log n)$ yields the following. Note that we restrict $|\mu|\le\delta^{-1/2}$ to remove the $\delta \cdot \mu^2/n$ term.

\begin{cor}[Unbiased Gaussian Mean Estimation]
    \label{cor:unbiased_gauss}
    Let $\eps\in(0,1)$, $\delta\in(0,1/n)$, and $n \ge O(\log(1/\delta)/\eps)$. 
    Let $M=\fine_{\eps,\delta,c,\sigma,n_1,n_2}$ be as in Algorithm~\ref{alg:fine} with appropriate settings of parameters.
    Then, for all $\mu \in [-\delta^{-1/2},\delta^{-1/2}]$,
    \[\ex{X \gets \cN(\mu,1)^n, M}{M(X)}=\mu ~~~\text{ and }~~~ \ex{X \gets \cN(\mu,1)^n,M}{(M(X)-\mu)^2} \le O\left(\frac1n + \frac{\log n}{n^2\eps^2}\right).\]
\end{cor}

To prove Theorem~\ref{thm:unbiased_final}, we use the following lemma that characterizes the symmetry of a clipped random variable from a symmetric distribution under special circumstances.
\begin{lem}\label{lem:clip_symmetry}
    Let $P$ and $Q$ be symmetric distributions with the same center $\mu(P)=\mu(Q)$. Let $c>0$. Define a distribution $R$ to be $\clip_{[Y-c,Y+c]}(X)$ where $X \gets P$ and $Y \gets Q$ are independent. Then $R$ is symmetric with the same center $\mu(R)=\mu(P)=\mu(Q)$.
\end{lem}
\begin{proof}
    Assume, without loss of generality, that $\mu(P)=\mu(Q)=0$.
    Let $X \gets P$ and $Y \gets Q$ be independent.
    Let $Z = \clip_{[Y-c,Y+c]}(X)$.

    We claim that
    \[\forall x,y\in\mathbb{R} ~~~ \clip_{[(-y)-c,(-y)+c]}(-x) = -\clip_{[y-c,y+c]}(x).\]
    This can be verified by analyzing the following cases: (1) $x < y-c$; (2) $x \in [y-c,y+c]$; and (3) $x > y+c$.

    Since $P$ and $Q$ are symmetric, $\clip_{[(-Y)-c,(-Y)+c]}(-X)$ has the same distribution as $Z$.
    By the claim, this is simply $-Z$. Ergo, the distribution of $Z$ is symmetric and $0$-centered.
\end{proof}


\begin{proof}[Proof of Theorem~\ref{thm:unbiased_final}]
    The privacy of Algorithm~\ref{alg:fine} follows from parallel composition, as we split the dataset in two, and apply $(\eps,\delta)$-DP algorithms to each half. Computing $\wt\mu$ is $(\eps,\delta)$-DP by Proposition~\ref{prop:coarse_dp}.
    If $\wt\mu=\bot$, then we compute $\wh\mu$ in a $(0,\delta)$-DP manner by sampling a $\delta$ fraction of the data points. If $\wt\mu\ne\bot$, then we compute $\wh\mu$ in a $(\eps,0)$-DP manner using clipping and Laplace noise addition (Lemma~\ref{lem:laplacedp}).

    Note that $\wt\mu$ is independent from $X_{n_1+1}, \dots, X_{n_1+n_2}$, which are the data points used to compute $\wh\mu$. If $\wt\mu=\bot$, then we compute $\wh\mu$ in an unbiased manner:
    \[ \ex{}{\wh\mu \mid \wt\mu=\bot} = \ex{}{\frac{1}{n_2\delta} \sum_{i=n_1+1}^{n_1+n_2} X_i \xi_i} =  \frac{1}{n_2\delta} \sum_{i=n_1+1}^{n_1+n_2} \ex{}{X_i} \ex{}{\xi_i} = \frac{1}{n_2\delta} \sum_{i=n_1+1}^{n_1+n_2} \mu(P) \delta = \mu(P). \]
    Now, condition on $\wt\mu\ne\bot$.
    By Proposition~\ref{prop:coarse_symmetric}, $\wt\mu$ has a symmetric distribution with center $\mu(P)$.
    By Lemma~\ref{lem:clip_symmetry}, $\ex{}{\clip_{[\wt\mu-c,\wt\mu+c]}(X_i) \mid \wt\mu\ne\bot} = \mu(P)$, which implies that $\ex{}{\wh\mu\mid\wt\mu\ne\bot}=\mu(P)$ because the Laplace noise has expected value $0$.
    Combining these two cases implies $\ex{}{\wh\mu}=\mu(P)$.

    Finally, we analyze the variance:
    \[
        \ex{}{(\wh\mu-\mu(P))^2} = \pr{}{\wt\mu=\bot} \cdot \ex{}{(\wh\mu-\mu(P))^2 \mid \wt\mu=\bot} + \pr{}{\wt\mu\ne\bot} \cdot \ex{}{(\wh\mu-\mu(P))^2 \mid \wt\mu\ne\bot}.
    \]
    We bound the two terms for $\wt\mu=\bot$ and $\wt\mu\ne\bot$ separately. For the first term, Proposition~\ref{prop:coarse_success} gives us $\pr{}{\wt\mu=\bot} \le \gamma$. Then we have the following.
    \begin{align*}
        \pr{}{\wt\mu=\bot} \cdot \ex{}{(\wh\mu-\mu(P))^2 \mid \wt\mu=\bot} &= \pr{}{\wt\mu=\bot} \cdot \ex{}{\left(\frac{1}{n_2\delta} \sum_{i=n_1+1}^{n_1+n_2} X_i \xi_i-\mu(P)\right)^2}\\
            &= \pr{}{\wt\mu=\bot} \cdot \frac{1}{n_2^2\delta^2} \sum_{i=n_1+1}^{n_1+n_2} \ex{}{\left( X_i \xi_i-\mu(P)\right)^2}\\
            &\le \pr{}{\wt\mu=\bot} \cdot \frac{1}{n_2^2\delta^2} \sum_{i=n_1+1}^{n_1+n_2} \ex{}{\left( X_i \xi_i \right)^2}\\
            &= \pr{}{\wt\mu=\bot} \cdot \frac{\mu(P)^2 + \ex{X \gets P}{(X-\mu(P))^2}}{n_2\delta}\\
            &\leq \gamma \cdot \frac{\mu(P)^2 + 1}{n_2\delta}.
    \end{align*}
    Now, we bound the second term: $\pr{}{\wt\mu\ne\bot} \cdot \ex{}{(\wh\mu-\mu(P))^2 \mid \wt\mu\ne\bot} = \ex{}{\mathbb{I}[\wt\mu\ne\bot] \cdot (\wh\mu-\mu(P))^2}$. We split this into two cases: $A \coloneqq \left[ \wt\mu \in [\mu(P) - \sigma, \mu(P) + \sigma] \right]$ and $ B \coloneqq \left[ \wt\mu \in \R \setminus [\mu(P) - \sigma, \mu(P) + \sigma] \right]$. Note that the event $A \wedge \wt\mu\ne\bot$ is equivalent to $A$ because $A$ cannot happen if $\wt\mu=\bot$, because $\bot\notin\R$. Similarly, $B \wedge \wt{\mu}\ne\bot$ is equivalent to $B$. Note that $\wt\mu\ne\bot \implies A \vee B$. Thus, we have
    \begin{align}
        \pr{}{\wt\mu\ne\bot} \cdot \ex{}{(\wh\mu-\mu(P))^2 \mid \wt\mu\ne\bot} &= \pr{}{\wt\mu \neq \bot}\cdot\pr{}{A \mid \wt\mu\ne\bot} \cdot \ex{}{(\wh\mu-\mu(P))^2 \mid \wt\mu\ne\bot \wedge A} \nonumber \\ &~~~~~ + \pr{}{\wt\mu \neq \bot}\cdot\pr{}{B \mid \wt\mu\ne\bot} \cdot \ex{}{(\wh\mu-\mu(P))^2 \mid \wt\mu\ne\bot \wedge B} \nonumber \\
            &= \pr{}{\wt\mu \neq \bot \wedge A} \cdot \ex{}{(\wh\mu-\mu(P))^2 \mid \wt\mu\ne\bot \wedge A} \nonumber \\ &~~~~~ + \pr{}{\wt\mu \neq \bot \wedge B} \cdot \ex{}{(\wh\mu-\mu(P))^2 \mid \wt\mu\ne\bot \wedge B} \nonumber \\
            &= \pr{}{\wt\mu\ne\bot\wedge A} \cdot \ex{}{(\wh\mu-\mu(P))^2 \mid \wt\mu\ne\bot \wedge A} \nonumber \\&~~~~~ + \ex{}{\mathbb{I}[\wt\mu\ne\bot] \cdot \mathbb{I}\left[B\right] \cdot (\wh\mu-\mu(P))^2} \nonumber \\
            &\le \ex{}{(\wh\mu-\mu(P))^2 \mid \wt\mu\ne\bot \wedge A} + \ex{}{\mathbb{I}[\wt\mu\ne\bot] \cdot \mathbb{I}\left[B\right] \cdot (\wh\mu-\mu(P))^2} \nonumber\\
            &= \ex{}{(\wh\mu-\mu(P))^2 \mid A} + \ex{}{\mathbb{I}\left[B\right] \cdot (\wh\mu-\mu(P))^2} \label{eq:mse_coarse_not_bot}.
    \end{align}
    If $A$ holds (i.e., $\wt\mu \in [\mu(P) - \sigma, \mu(P) + \sigma]$), then $\mu(P) \in [\wt\mu - \sigma, \wt\mu + \sigma]$, so we can bound the first term of the last line in Inequality~\ref{eq:mse_coarse_not_bot} as follows.
    \begin{align*}
        \ex{}{(\wh\mu-\mu(P))^2 \mid A}
        &= \ex{}{\left(\frac{1}{n_2} \sum_{i=n_1+1}^{n_1+n_2} \clip_{[\wt\mu-c,\wt\mu+c]}(X_i) + \Lap\left(\frac{2c}{n_2\eps}\right)-\mu(P)\right)^2 \mid A} \\
        &\overset{\textrm{(A)}}{=} \ex{}{\left(\frac{1}{n_2} \sum_{i=n_1+1}^{n_1+n_2} \clip_{[\wt\mu-c,\wt\mu+c]}(X_i) -\mu(P)\right)^2 \mid A} + 2\left(\frac{2c}{n_2\eps}\right)^2\\
        &\overset{\textrm{(B)}}{\le} \ex{}{\frac{\ex{X \gets P}{(X-\mu(P))^2}}{n_2} \mid A} \\ &~~~~~~~~~~+ \ex{}{\left(\frac{ \ex{X \gets P}{|X-\mu(P)|^\lambda}}{\lambda \cdot \left(\min\{ \mu(P) - (\wt\mu-c), (\wt\mu+c)-\mu(P) \}\right)^{\lambda-1}}\right)^2 \mid A} + \frac{8c^2}{n_2^2\eps^2} \\
        &\le \frac{\ex{X \gets P}{(X-\mu(P))^2}}{n_2} + \left(\frac{1}{\lambda} \cdot \frac{\ex{X \gets P}{|X-\mu(P)|^\lambda}}{\left(c - \sigma \right)^{\lambda-1}}\right)^2 + \frac{8c^2}{n_2^2\eps^2}\\
        &\le \frac{1}{n_2} + \frac{\psi^{2\lambda}}{\lambda^2 \cdot (c - \sigma)^{2(\lambda-1)} } + \frac{8c^2}{n_2^2\eps^2}\\
        &\overset{\textrm{(C)}}{=} \frac{1}{n_2} + \frac{\psi^{2\lambda}}{\lambda^2 \cdot \psi^{2(\lambda-1)} \cdot (n_2\eps)^{2-2/\lambda}} + \frac{8c^2}{n_2^2\eps^2}\\
        &= \frac{1}{n_2} + \frac{8c^2+\psi^2\cdot(n_2\eps)^{2/\lambda}\cdot\lambda^{-2}}{n_2^2\eps^2}\\
        &\le \frac{1}{n_2} + \frac{8c^2+\psi^2\cdot(n_2\eps)^{2/\lambda}}{n_2^2\eps^2}.
    \end{align*}
    In the above: Equality A follows from the fact that the Laplace noise is independent from everything else.
    Inequality B follows from Lemma~\ref{lem:clip_mse} and linearity of expectations; and Equality C follows from the setting of $c=\sigma + \psi \cdot (n_2\eps)^{1/\lambda}$.\\
    Next, we bound the second term in the last line of Inequality~\ref{eq:mse_coarse_not_bot}. We use the fact that \[\forall x \in \R ~~ |\clip_{[\wt\mu-c,\wt\mu+c]}(x)-\mu(P)| \le |\wt\mu-\mu(P)|+c.\]
    We have
    \begin{align*}
        \ex{}{\mathbb{I}\left[B\right] \cdot (\wh\mu-\mu(P))^2} &= \ex{}{\mathbb{I}\left[B\right] \cdot \left(\frac{1}{n_2} \sum_{i=n_1+1}^{n_1+n_2} \clip_{[\wt\mu-c,\wt\mu+c]}(X_i) + \Lap\left(\frac{2c}{n_2\eps}\right) -\mu(P)\right)^2 } \\
            &\overset{\textrm{(D)}}{=} \ex{}{\mathbb{I}\left[B\right] \cdot \left(\frac{1}{n_2} \sum_{i=n_1+1}^{n_1+n_2} \clip_{[\wt\mu-c,\wt\mu+c]}(X_i) -\mu(P)\right)^2 } + 2 \left(\frac{2c}{n_2\eps}\right)^2 \cdot \pr{}{B}\\
            &\le \ex{}{\mathbb{I}\left[B\right] \cdot \left(|\wt\mu-\mu(P)| + c\right)^2 } +  \frac{8c^2}{n_2^2\eps^2}  \cdot \pr{}{B}\\
            &\overset{\textrm{(E)}}{\le} \ex{}{\mathbb{I}\left[B\right]^{\frac{\lambda}{\lambda-2}}}^{\frac{\lambda-2}{\lambda}} \cdot \ex{}{\mathbb{I}[\wt\mu\ne\bot]\cdot\left(|\wt\mu-\mu(P)| + c\right)^{\lambda} }^{2/\lambda} + \frac{8c^2}{n_2^2\eps^2}  \cdot \pr{}{B} \\
            &\overset{\textrm{(F)}}{\le} \ex{}{\mathbb{I}\left[B\right]}^{\frac{\lambda-2}{\lambda}}  \cdot \ex{}{\mathbb{I}[\wt\mu\ne\bot]\cdot\left(|\wt\mu-\mu(P)|^{\lambda} + c^{\lambda}\right) \cdot 2^{\lambda-1} }^{2/\lambda}  +  \frac{8c^2}{n_2^2\eps^2}  \cdot \pr{}{B} \\
            &\overset{\textrm{(G)}}{\le} \gamma^{\frac{\lambda-2}{\lambda}} \cdot \left(\frac12 + n_1 \cdot 2^{\lambda-1} \cdot \psi^{\lambda} + c^{\lambda}\right)^{2/\lambda} \cdot 2^{2-2/\lambda}  +  \frac{8c^2}{n_2^2\eps^2}  \cdot \gamma \\
            &\overset{\textrm{(H)}}{\le} \gamma^{\frac{\lambda-2}{\lambda}} \cdot \left(2^{-2/\lambda} + n_1^{2/\lambda} \cdot 2^{2-2/\lambda} \cdot \psi^2  + c^{2}\right) \cdot 2^{2-2/\lambda}  +  \frac{8c^2}{n_2^2\eps^2}  \cdot \gamma \\
            &\le 4\gamma^{\frac{\lambda-2}{\lambda}}\left(1 + 4n_1^{2/\lambda}\psi^2  + c^{2}\right)  +  \frac{8c^2}{n_2^2\eps^2}  \cdot \gamma.
    \end{align*}
    In the above: Equality D follows from 
    the independence of the Laplace noise; Inequality E follows from H\"older's inequality; Inequality F holds because $\forall p \ge 1 ~ \forall x,y\ge 0 ~~(x+y)^p \le (x^p+y^p) \cdot 2^{p-1}$; Inequality G follows from Proposition~\ref{prop:coarse_success}; and Inequality H holds because $\forall p \in (0,1] ~ \forall x,y \ge 0 ~~ (x+y)^p \le x^p + y^p$.
    
    Finally, we can combine all the pieces, and use our parameter settings $\gamma=\delta^2\le1$ and $c^2=(10 + \psi \cdot (n_2\eps)^{1/\lambda})^2 \le {2\psi^2(n_2\eps)^{2/\lambda}+200}$, to get the following.
    \begin{align*}
        \ex{}{(\wh\mu-\mu(P))^2}
        &\le \gamma\left(\frac{\mu(P)^2 + 1}{n_2\delta}\right)
        + \left(\frac{1}{n_2} +  \frac{8c^2+\psi^2(n_2\eps)^{2/\lambda}}{n_2^2\eps^2}\right)
        + \left(4\gamma^{\frac{\lambda-2}{\lambda}}\left(1 + 4n_1^{2/\lambda}\psi^2  + c^{2}\right)  +  \frac{8c^2}{n_2^2\eps^2}  \cdot \gamma\right)\\
        &\le \frac{1}{n_2} + \frac{16c^2+\psi^2\cdot(n_2\eps)^{2/\lambda}}{n_2^2\eps^2} + \delta \cdot \frac{1+\mu(P)^2}{n_2}  + \delta^{2-4/\lambda} \cdot 4\left(4n_1^{2/\lambda} \cdot \psi^2 +c^2 +1 \right)\\
        &\le \frac{1}{n_2} + \frac{33\psi^2(n_2\eps)^{2/\lambda}+3200}{n_2^2\eps^2} + \delta \cdot \frac{1+\mu(P)^2}{n_2} + \delta^{2-4/\lambda} \cdot \left(16\psi^2n_1^{2/\lambda} + 8\psi^2(n_2\varepsilon)^{2/\lambda} + 804\right)\\
        &= \frac{1}{n_2} + O\left( \frac{\psi^2}{(n_2\eps)^{2-2/\lambda}} + \delta \cdot \frac{\mu(P)^2}{n_2} + \delta^{2-4/\lambda} \cdot (n_1+n_2\eps)^{2/\lambda} \cdot \psi^2\right).
    \end{align*}
    
    Our proof is now complete.
\end{proof}

\subsection{Empirical Comparison to Karwa--Vadhan Estimator}
\label{subsec:kv_vs_unbiased}

We conduct experiments comparing the sampling distribution and bias properties of our estimator to that of \cite{KarwaV18}. Their estimator, which we call the \textit{KV estimator}, functions like ours except that coarse estimation does not use random bin offsets $T \sim \mathcal{U}[-1/2, 1/2]$.

In Figure~\ref{fig:kv_vs_unbiased}, we plot the bias of the KV estimator on synthetic $\mathcal{N}(\mu, 1)$ data as $\mu$ varies. In this case, we see that the bias of the KV estimator varies cyclically with respect to the true mean. This arises because the KV estimator is forced to clip around an integer-valued coarse estimate, whereas our estimator can more naturally track data between integers.

To illustrate this phenomenon more clearly, Figure~\ref{fig:kv_vs_unbiased_real} shows the sampling distribution of our bias-adjusted estimator compared to the KV estimator on a human height dataset \citep{Dinov2008}. The KV estimator exhibits multimodal behaviour clustered around integer multiples of the dataset deviation. The dataset is approximately symmetric with a skewness coefficient of -0.006. Therefore, we can expect near unbiasedness from our estimator. Indeed, our estimator has an estimated bias of $0.0045$ inches, whereas the KV estimator has an order of magnitude larger bias $0.030$ inches (both estimates $\pm 0.0006$ 95\% CI).

\begin{figure}[ht]
    \centering
    \subfloat[
        Datasets of 400 samples from $\mathcal{N}(\mu, 1)$ are collected. We apply our bias-corrected estimator and the biased estimator of \citet{KarwaV18} as $\mu$ varies. The bias of both is approximated by drawing many datasets and computing the signed bias. 95\% CI is lightly shaded.
    ]{{
        \includegraphics[width=7cm]{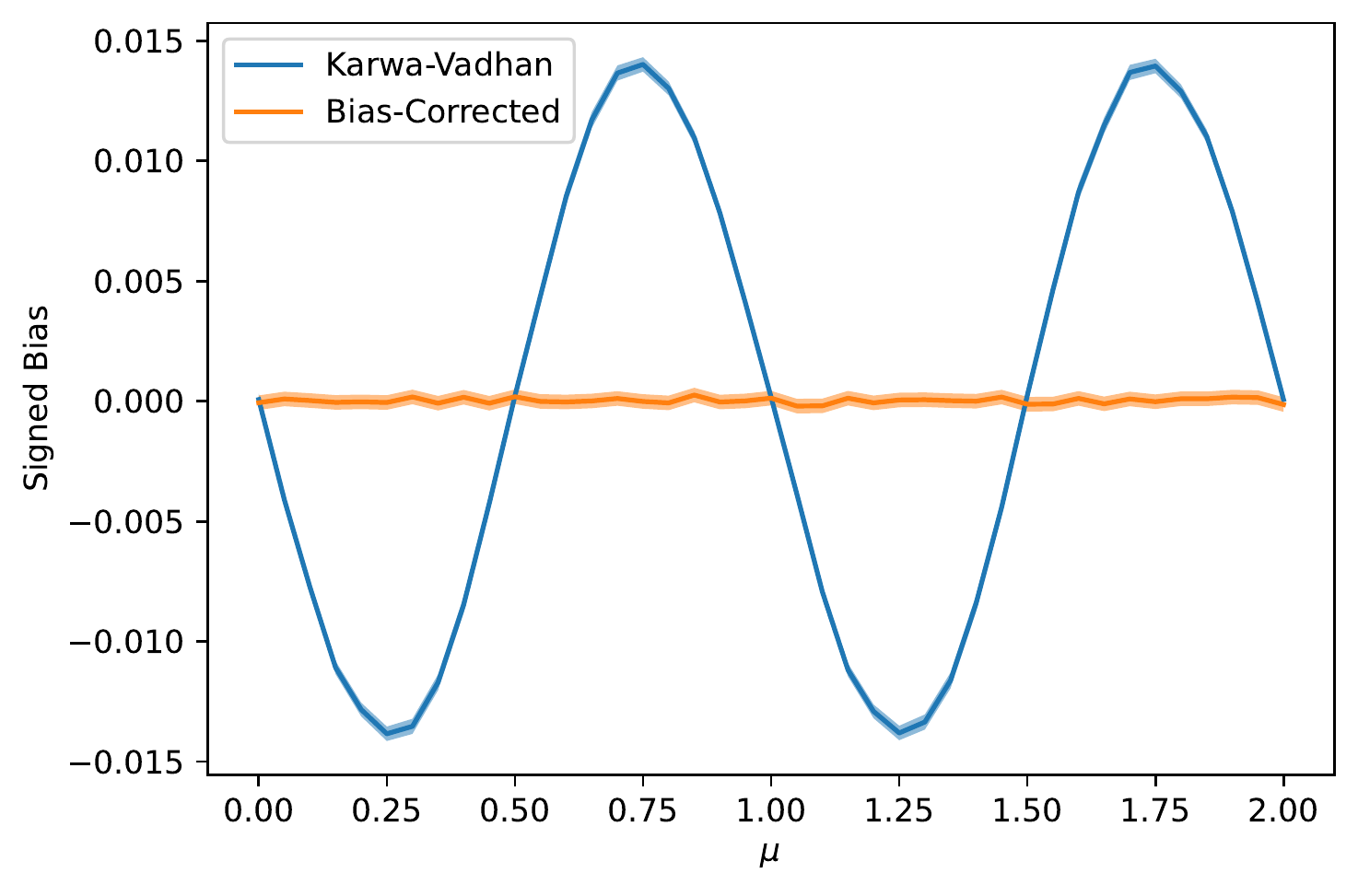}
        \label{fig:kv_vs_unbiased}
    }}
    \qquad
    \subfloat[
        Datasets of size 400 are repeatedly subsampled from a human adolescent height database \citep{Dinov2008}. The sampling distributions of the \citet{KarwaV18} estimator and our bias-corrected estimator are shown as histograms. True mean is indicated by dashed line.
    ]{{
        \includegraphics[width=7cm]{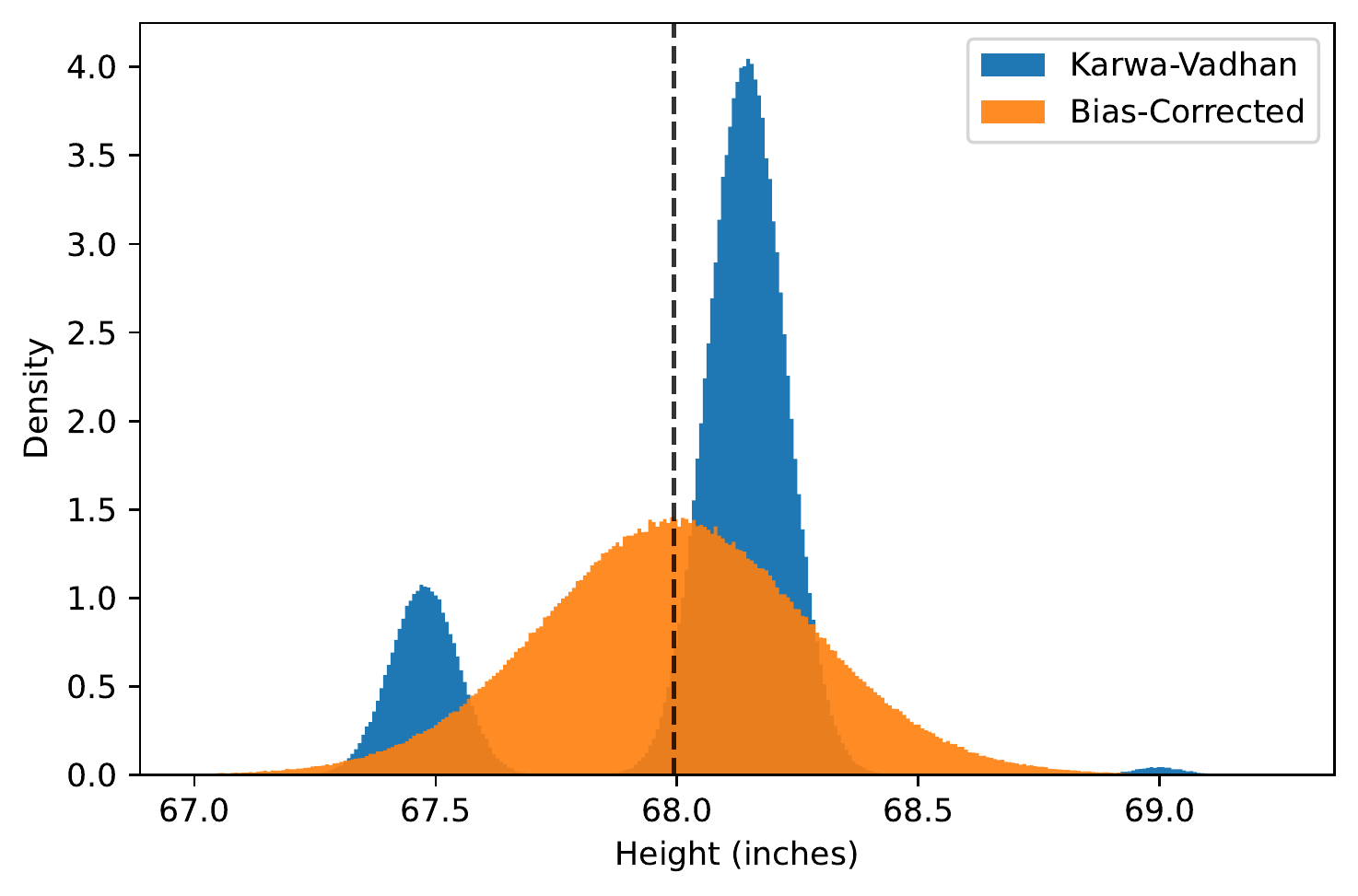}
        \label{fig:kv_vs_unbiased_real}
    }}
    \caption{Bias-corrected vs KV estimator on synthetic and real data}
    \label{fig:example}
\end{figure}

\section{Impossibility of Pure DP Unbiased Estimation}
\label{sec:analyticity}

In Section~\ref{sec:symmetric}, we showed that it is possible to perform unbiased mean estimation for symmetric distributions. However, this result only provides approximate DP (i.e., $(\eps,\delta)$-DP with $\delta>0$). We now show that this is inherent.
We show that unbiased estimation is impossible under pure DP (i.e., $(\eps,0)$-DP) for a large class of sufficiently smooth distributions.

For technical expedience, we focus on the case when the data come from an \emph{exponential family}, though the following results can be proved in a more general setting. We refer to the calculations in Subsection~\ref{subsec:expectation_is_analytic} for the recipe to generalize beyond exponential families. In any case, exponential families include a wide range of common distributions, including Gaussians, exponential distributions, Laplace distributions with fixed mean, and Gamma distributions. A notable exception is the Laplace distribution with unknown mean. It is an open and interesting question to investigate whether our impossibility result applies to mean-estimation on Laplacian data. We do not expect our method to be applicable.

\begin{defn}[Exponential Family]\label{def:exponential_family}
    Let $D \subseteq \R^n$, $h \colon D \to [0,\infty)$, and $T \colon D \to \R^k$. The exponential family with \emph{carrier measure} and \emph{sufficient statistic} $h$ and $T$ respectively is the collection of probability measures $P_\eta$ with density
    \(
        f_{T,h,\eta}(x) = h(x) \exp\left(\eta^{\top} T(x) - Z(\eta)\right),
    \)
    where $Z(\eta) = \ln\left(\int\limits_S h(x) \exp\left(\eta^{\top} T(x)\right) \, dx\right)$ is called the \emph{log-partition function} of the family. The family is defined over all values $\eta \in \R^k$ for which $Z(\eta) < \infty$. The set of these values, which we denote $U$, is called the family's \emph{range of natural parameters}.
\end{defn}

\begin{thm}[Impossibility of Pure DP Unbiased Estimation for Exponential Families]
    \label{thm:unbiased_ef_lb}

    Let $U \subseteq \R$ be an interval of infinite length, let $\{P_\eta : \eta \in U\}$ be an exponential family, and let $I \subseteq U$ be any interval of positive length. Then, for any $\eps \geq 0$ and $n \geq 0$, there exists no $(\eps,0)$-DP algorithm $M:\R^n \rightarrow \R$ satisfying $\exin{X \gets P_\eta^n,M}{M(X)} = \eta$ for all $\eta \in I$.
\end{thm}

We remark that the interval over which the algorithm is unbiased must have positive length. It is easy to construct a pathological estimator that is unbiased at a single point $\eta_0 \in U$ but not anywhere else, e.g. by
setting $M(x) = \eta_0$ for all $x \in \R^n$. On the other hand, the theorem implies that there can be no pure DP estimator that gives an unbiased estimate for the mean of $\textrm{Exponential}(\lambda)$ for all $\lambda \in (0, 0.1)$.
The range of natural parameters having infinite length is also essential to our analysis. We emphasize that this is a property of the distribution and not of the algorithm; that is, the algorithm does not need to ``know'' about $U$.
Note that the family of distributions $\{\mathsf{Bernoulli}(p) : p \in [0,1]\}$ is an exponential family\footnote{Definition~\ref{def:exponential_family} is stated in terms of densities, but it can be extended to discrete distributions.} and it is possible to estimate the mean $p$ under pure DP. In this case $U=[0,1]$ has finite length, so we see that the assumption that $U$ has infinite length is also necessary.

Proving this result requires tools from complex analysis and measure theory, a review of which can be found in Appendices~\ref{app:complex_analysis} and~\ref{app:measure_theory}, respectively. We first show that, for an estimator $\phi : \R^n \to \R$ and an exponential family $\{P_\eta : \eta\}$, the expected value of the estimator $\exin{X \gets P_\eta^n}{\phi(X)}$ is an analytic function in $\eta$. We then apply the identity theorem for analytic functions to argue that if $\phi$ is locally unbiased, i.e., unbiased when $\eta$ lies in some small set, then $\phi$ must also be globally unbiased, i.e., unbiased for all choices of $\eta$. On the other hand, we will argue that global unbiasedness over an infinite interval is impossible for pure DP estimators as a consequence of the strong group privacy properties of pure DP.

\subsection{Locally Unbiased Estimators Are Globally Unbiased}
\label{subsec:expectation_is_analytic}

We first argue that the expectation of the estimator $\exin{X \gets P_\eta}{\phi(X)}$ must be an analytic function of the parameter $\eta$. 
A function being analytic means that its Taylor series provides an exact representation of the function.

\begin{prop}[Analyticity under Exponential Families]
    \label{prop:expectation_is_analytic}
    Let $\{P_\eta : \eta \in U\}$ be an exponential family on $\R^n$ in canonical form (recall Definition~\ref{def:exponential_family}) and let $\phi : \mathcal{X}^n \to \R$ be any well-defined estimator for $\{P_\eta : \eta \in U\}$, i.e., $\exin{X \gets P_\eta}{|\phi(X)|}$ is finite for all $\eta\in U$. Then $g : U \to \R$ defined by $g(\eta) \coloneqq \exin{X \gets P_\eta}{\phi(X)}$ is an analytic function.
\end{prop}

An immediate and important consequence if Proposition~\ref{prop:expectation_is_analytic} is that any estimator for the parameter of an exponential family that is locally unbiased is also globally unbiased. That is, if we have an unbiased estimator for a restricted range of parameter values (of nonzero length), then we have an unbiased estimator for the entire range.

\begin{prop}[Locally Unbiased Implies Globally Unbiased]
    \label{prop:local_to_global}
    Let $\phi : \R^n \to \R$ be any well-defined estimator for an exponential family $\{P_\eta : \eta \in U\}$. Let $I \subseteq U$ be an interval of nonzero length. If $\exin{X \gets P_\eta}{\phi(X)}=\eta$ for all $\eta \in I$, then $\exin{X \gets P_\eta}{\phi(X)}=\eta$ for all $\eta \in U$.
\end{prop}

Intuitively, the result holds because if an analytic function is linear in some nontrivial interval, we can compute the Taylor series at an interior point of that interval to deduce that the function is linear globally. In any case, the result follows immediately by applying the identity theorem (Theorem~\ref{thm:identity}) to Proposition~\ref{prop:expectation_is_analytic}.

Now, before we prove Proposition~\ref{prop:expectation_is_analytic}, we note that the product distribution where each marginal has the same distribution from an exponential family is also an exponential family. See e.g. Theorem 3.13 of \citet{bedbur2021multivariate}.

\begin{lem}\label{lemma:ef_dataset}
    Let $\{P_\eta : \eta \in U\}$ be an exponential family with support
    $D \subseteq \R$. Then for any $n \in \N$, the family
    of distributions
    $\left\{P_\eta^n: \eta \in U\right\}$ is an exponential
    family over $D^n$ with the same natural parameters as well as carrier measure and sufficient statistic given, respectively, by
    \begin{align*}
        h_n(x_1, \dots, x_n) = \prod\limits_{i=1}^{n}{h(x_i)}~~~ \text{and}~~~
        T_n(x_1, \dots, x_n) = \sum\limits_{i=1}^{n}{T(x_i)}.
    \end{align*}
\end{lem}

With this detail out of the way, the main idea behind the proof of Proposition~\ref{prop:expectation_is_analytic} is that analyticity is preserved under integration under certain circumstances, which we show next. Although a proof for the real plane is possible, it will be technically convenient to pass to the complex plane where we can wield Morera's theorem (Theorem~\ref{thm:morera}).

\begin{lem}\label{lemma:analyticity}
  Let $\Omega$ be a $\sigma$-finite measure space with measure $\nu$, let $V \subseteq \C$ be open, and let $f : \Omega \times V \to \C$. Assume that $f(\omega, \eta)$ is analytic in $\eta$ for every fixed $\omega \in \Omega$ and that, for every compact $K \subseteq V$, there is a $\nu$-integrable function (see Section~\ref{app:measure_theory}) $G : \Omega \to [0, \infty)$ for which $\abs{f(\omega, \eta)} \leq G(\omega)$ for all $\eta \in K$. Then
  $g(\eta) \coloneqq \int\limits_\Omega f(\omega, \eta) \, d\nu(\omega)$
  is analytic, as well.
\end{lem}
\begin{proof}
  Our plan is to apply Morera's theorem (Theorem~\ref{thm:morera}). To that end, we must first show that $g$ is continuous,
  so let $(\eta_n)_{n \in \N}$ be any sequence with $\eta_n \to \eta$ as $n \to \infty$.
  By our assumption, there is a $\nu$-integrable $G : \Omega \to [0, \infty)$ such that $\abs{f(\omega, \eta_n)} \leq G(\omega)$ for all $n \in \N$ and $\omega \in \Omega$. So, by the dominated convergence theorem (Theorem~\ref{thm:dominated_convergence}), $g(\eta_n) \to g(\eta)$ as $n \to \infty$.

  Now, let $\gamma : [0, 1] \to \C$ be any closed contour lying in a simply connected (see Section~\ref{app:complex_analysis}) subset of $V$, and let $\gamma'$ denote its first derivative. Then, $\abs{\gamma'}$ must be bounded by some $C > 0$, so
  \begin{align*}
    \int\limits_0^1 \int\limits_\Omega \abs{f(\omega, \gamma(t))\gamma'(t)} \, d\nu(\omega) \, dt
      \leq \int\limits_0^1 \int\limits_\Omega G(\omega)C \, d\nu(\omega) \, dt
      = C\int\limits_\Omega G \, d\nu
      < \infty
  \end{align*}
  and thus Fubini's theorem (Theorem~\ref{thm:fubini}) implies that
  \begin{align*}
    \oint\limits_\gamma g(\eta) \, d\eta
      & = \int\limits_0^1 \int\limits_\Omega f(\omega, \gamma(t))\gamma'(t) \, d\nu(\omega) \, dt \\
      & = \int\limits_\Omega \int\limits_0^1 f(\omega, \gamma(t))\gamma'(t) \, dt \, d\nu(\omega) \\
      & = \int\limits_\Omega \oint\limits_\gamma f(\omega, \eta) \, d\eta \, d\nu(\omega) \\
      & = \int\limits_\Omega 0 \, d\nu(\omega)
          \tag{Theorem~\ref{thm:cauchy}} \\
      & = 0.
  \end{align*}
  As $\gamma$ was arbitrary, $g$ must be analytic by Morera's theorem.
\end{proof}

\begin{proof}[Proof of Proposition~\ref{prop:expectation_is_analytic}]
    Our main goal is to show that $g(\eta) \coloneqq \exin{X \gets P_\eta}{\phi(X)}$ is analytic. To that end, let $h$, $T$, and $Z$ be the carrier measure, the sufficient statistic, and the log-partition function of $\{P_\eta : \eta \in U\}$, respectively.

    We first show that $\exp(Z(\eta))$ is analytic by way of Lemma~\ref{lemma:analyticity}. Indeed, $r(x, \eta) \coloneqq h(x)\exp(\eta T(x))$ is entire (see Section~\ref{app:complex_analysis}) in $\eta \in \C$ for each fixed $x \in \R^n$. Let $K \subseteq \C$ be an arbitrary compact set, and let $m$ and $M$ be the minimum and the maximum real coordinates among the points within $K$, respectively. Then for any $x \in \R^n$ and $\eta \in K$,
        \[T(x) < 0 \implies \abs{r(x, \eta)} = h(x)\exp(\Real(\eta) T(x)) \leq h(x)\exp(m T(x))\]
    and
        \[T(x) \geq 0 \implies \abs{r(x, \eta)} \leq h(x) \exp(M T(x)),\]
    so we have
        \[\abs{r(x, \eta)} \leq h(x)\exp(m T(x)) + h(x) \exp(M T(x)).\]  
    But $\int\limits_{\R^n} h(x)\exp(m T(x)) + h(x) \exp(M T(x)) \, dx = \exp(Z(m)) + \exp(Z(M)) < \infty$, so, since $K$ was arbitrary, $\exp(Z(\eta)) = \int\limits_{\R^n} r(x, \eta) \, dx$ must be entire by Lemma~\ref{lemma:analyticity}.

    As a consequence, $h(x)\exp(\eta T(x) - Z(\eta))$ is analytic in $\eta$ for every fixed $x \in \R^n$, so we can apply nearly the same argument to $h(x)\exp(\eta T(x) - Z(\eta))$ in order to conclude that
    \begin{align*}
        g(\eta) = \ex{X \gets P_\eta}{\phi(X)}
            = \int\limits_\Omega \phi(x)h(x)\exp(\eta T(x) - Z(\eta)) \, dx
    \end{align*}
    is analytic, as well.
\end{proof}

\subsection{Pure DP Estimators Are Uniformly Bounded}

We now show by strong group privacy that, if the expectation of a pure DP estimator is bounded locally, then it is uniformly bounded globally.

\begin{prop}[Pure DP Estimators Are Uniformly Bounded]
    \label{prop:deprivatization_is_bounded}
    Let $A : \cX^n \to \R$ be a randomized algorithm. If $A$ is $(\eps,0)$-DP, then for all $x,x^* \in \cX^n$,
    $\abs{\ex{A}{A(x)}} \leq \exp(\eps n) \ex{A}{\abs{A(x^*)}}$.
\end{prop}

Intuitively, this holds because modifying a single entry of $x$ changes the distribution of $A(x)$ by at most a factor of $\exp(\epsilon)$, so modifying all $n$ points should result in a change by a factor of at most $\exp(\epsilon n)$.
We emphasize that the bound on $\abs{\ex{A}{A(x)}}$ is \emph{uniform} because the result holds for any $x,x^*$.

\begin{proof}[Proof of Proposition~\ref{prop:deprivatization_is_bounded}]
    Recall that $\ex{}{Y} = \int\limits_0^\infty \pr{}{Y \geq t} \, dt$ for any non-negative random variable $Y$.
    For any $x \in \cX^n$, we have
    \begin{align*}
        \abs{\ex{A}{A(x)}} \overset{\mathrm{(a)}}{\leq} \ex{A}{\abs{A(x)}}
            = \int\limits_0^\infty {\pr{A}{\abs{A(x)} \geq t}\,dt}
            \overset{\mathrm{(b)}}{\leq} \exp(\eps n) \int\limits_0^\infty
                {\pr{A}{\abs{A(x^*)} \geq t}\,dt}
            = \exp(\eps n) \ex{A}{\abs{A(x^*)}},
    \end{align*}
    where inequalities (a) and (b) follow from Jensen's inequality and group privacy (Lemma~\ref{lem:group_privacy}), respectively.
\end{proof}

Our impossibility result for exponential families now follows by stringing together the tools we have collected so far.

\begin{proof}[Proof of Theorem~\ref{thm:unbiased_ef_lb}]
    Suppose, for the sake of contradiction, there exist $\eps \geq 0$, $n \geq 0$, and an $\eps$-DP algorithm $M:\R^n \rightarrow \R$ for which $\exin{X \gets P_\eta^n,M}{M(X)} = \eta$ when $\eta \in I$. By Proposition~\ref{lemma:ef_dataset}, $\{P_\eta^n : \eta\}$ is an exponential family, such that for every $P \in \{P_\eta : \eta\}$, the natural parameter of $P^n$ is the same as that of $P$. Therefore, by Proposition~\ref{prop:local_to_global}, we have $\exin{X \gets P_{\eta}^n,M}{{M}(X)} = \eta$ for all $\eta \in U$. In particular, since $U$ is unbounded, $\exin{X \gets P_{\eta}^n,M}{{M}(X)}$ must be an unbounded function of $\eta$, which contradicts Proposition~\ref{prop:deprivatization_is_bounded}.
\end{proof}


\section*{Discussion and Future Work}


Comparing our lower and upper bounds gives strong evidence that the bias-accuracy-privacy tradeoff bifurcates around $\delta \approx \beta^4 \epsilon^2$. In the regime where $\delta$ is small relative to the bias, namely $\delta \ll \beta^4 \epsilon^2$, Theorems \ref{thm:main_trilemma_informal} and \ref{thm:eps_delta_ub_mix} yield matching bounds and thus the trilemma is tight. However, a technical gap in our results emerges when $\delta \gg \beta^4 \epsilon^2$, in which case our lower and upper bounds differ by a factor of $\sqrt{n}$. A direction for future work is to close this gap and determine whether crossing the threshold impacts the dependence on $n$.

Another direction is to consider higher-order moment estimation problems such as covariance estimation. We expect a similar tradeoff to hold if light distributional assumptions are imposed. However, just as symmetry enables unbiased private mean estimation, we conjecture that there are mild conditions that lead to unbiased private covariance estimation.

A final direction to consider is the algorithmic implications of our impossibility result (Theorem \ref{thm:unbiased_ef_lb}). 
The rates of certain tasks, such as stochastic gradient descent, depend on the bias (see, e.g., discussion in~\cite{KamathLZ22}).
Thus it is prudent to investigate implications of our results for privatizing these tasks.

\section*{Acknowledgements}
\addcontentsline{toc}{section}{Acknowledgements}

We thank Yu-Xiang Wang for asking interesting questions that are answered in Appendix~\ref{app:n} and Kelly Ramsay for helpful feedback on our manuscript. We also thank Weijie Su and Marco Avella Medina for their advice.

\section*{Funding}

GK was supported by an NSERC Discovery Grant, unrestricted gifts from Google and Apple, and a University of Waterloo startup grant. AM was supported by an NSERC Discovery Grant and a David R.\ Cheriton Graduate Scholarship. MR was supported by a Vector Scholarship in AI and an NSERC CGS-M. VS was supported by an NSERC Discovery Grant. JU was supported by NSF grants CCF-1750640 and CNS-2120603.

\addcontentsline{toc}{section}{References}
\bibliographystyle{alpha}
\bibliography{main}

\clearpage


\appendix

\section{Non-Private Error of Mean Estimation}

Given independent samples $X_1, \cdots, X_n \in \mathbb{R}$ from an unknown distribution $P$, the empirical mean $\hat\mu(X) \coloneqq \frac1n \sum_i^n X_i$ is an unbiased estimator of the distribution mean $\mu(P) \coloneqq \ex{X \gets P}{X}$ and its mean squared error is
\[\mse \coloneqq \ex{X \gets P^n}{\left( \hat\mu(X) - \mu(P) \right)^2} = \frac{\ex{X \gets P}{(X-\mu(P))^2}}{n} = O(1/n).\]
This mean squared error is asymptotically optimal in a minimax sense and is optimal for the univariate Gaussian case $P=\cN(\mu,\sigma^2)$. 

We have the following well-known result which shows that the empirical mean is asymptotically optimal even for the simple case of Bernoulli data.

\begin{prop}\label{prop:nondp_mse}
    Let $M : \{0,1\}^n \to \mathbb{R}$ be an estimator satisfying
    \[\forall p \in [0,1] ~~ \ex{X \gets \mathsf{Bernoulli}(p)^n}{(M(X)-p)^2} \le \alpha^2.\]
    Then $\alpha^2 \ge \frac{1}{6(n+2)}$.
\end{prop}

The empirical mean attains MSE $\ex{X \gets \mathsf{Bernoulli}(p)^n}{(\hat\mu(X)-p)^2} = \frac{p(1-p)}{n}$. However, this is not the minimax optimal estimator of the mean of a Bernoulli distribution, rather it is the biased estimator \[\check\mu(X) \coloneqq \frac{1}{n+\sqrt{n}} \left( \frac{\sqrt{n}}{2} + \sum_i^n X_i \right),\] which has MSE \[\ex{X \gets \mathsf{Bernoulli}(p)^n}{(\check\mu(X)-p)^2} = \frac{1}{4(\sqrt{n}+1)^2}\] for all $p \in [0,1]$ \citep{HodgesL50}.

\begin{proof}[Proof of Proposition \ref{prop:nondp_mse}.]
    Let $P \in [0,1]$ be uniform and, conditioned on $P$, let $X \gets \mathsf{Bernoulli}(P)^n$ be $n$ independent bits, each with conditional expectation $P$. Note that the marginal distribution of $\sum_i^n X_i$ is uniform on $\{0,1,\cdots,n\}$.
    
    Given $X=x$, the conditional distribution of $P$ is \[P|_{X=x} \stackrel{d}{=} \mathsf{Beta}\left(1+\sum_i^n x_i,1+\sum_i^n (1-x_i)\right).\] 
    In terms of mean squared error, the best estimator of $P$ is simply the mean of this conditional distribution. That is, the function $f : \{0,1\}^n \to \R$ that minimizes $\ex{P,X}{(P-f(X))^2}$ is the conditional expectation $f(x)=\ex{}{P\mid X=x}$.
    Indeed, this is the \emph{definition} of conditional expectation in the general measure theoretic setting. 
    Consequently, the best possible mean squared error of an estimator of $P$ given $X$ is the variance of this conditional distribution $P|X$.
    
    The distribution $\mathsf{Beta}(a,b)$ has mean $\frac{a}{a+b}$ and variance $\frac{ab}{(a+b)^2(a+b+1)}$.
    Now we have
    \begin{align*}
        \alpha^2 &\ge \ex{P \gets [0,1], X \gets \mathsf{Bernoulli}(P)^n}{(M(X)-P)^2} \\
        &\ge \ex{X}{\ex{P}{\left(\ex{P}{P\mid X}-P\right)^2 \mid X}} \\
        &= \ex{P \gets [0,1], X \gets \mathsf{Bernoulli}(P)^n}{\left( \frac{1+\sum_i^nX_i}{2+n} - P \right)^2} \\
        &= \ex{P \gets [0,1], X \gets \mathsf{Bernoulli}(P)^n}{\frac{(1+\sum_i^nX_i)(1+\sum_i^n(1-X_i))}{(n+2)^2(n+3)}} \\
        &= \frac{1}{n+1} \sum_{k=0}^n \frac{(1+k)(1+n-k)}{(n+2)^2(n+3)}
        = \frac{1}{(n+1)(n+2)^2(n+3)} \sum_{k=0}^n (1+n) + n \cdot k - k^2\\
        &= \frac{1}{(n+1)(n+2)^2(n+3)} \left( (1+n) \cdot (n+1) + n \cdot \frac{n(n+1)}{2} - \frac{n(n+1)(2n+1)}{6} \right)\\
        &= \frac{6(n+1)^2 + 3n^2(n+1) - n(n+1)(2n+1)}{6(n+1)(n+2)^2(n+3)}
        = \frac{6(n+1) + 3n^2 - n(2n+1)}{6(n+2)^2(n+3)}\\
        &= \frac{5n+6 + n^2}{6(n+2)^2(n+3)}
        = \frac{(n+2)(n+3)}{6(n+2)^2(n+3)}
        = \frac{1}{6(n+2)}.
    \end{align*}
    This completes the proof.
\end{proof}

If we change the distribution of $P \in [0,1]$ from uniform to $\mathsf{Beta}(\sqrt{n}/2,\sqrt{n}/2)$ in the above proof, then we obtain the stronger conclusion $\alpha^2 \ge \frac{1}{4(\sqrt{n}+1)^2}$, which is exactly optimal. However, this requires a more complicated calculation.

\section{Background on Complex Analysis}\label{app:complex_analysis}

The primary objects of interest in complex analysis are the \emph{holomorphic} functions in the complex plane, namely those functions $f : U \to \C$ that are differentiable at every point $z \in U$. Many familiar functions, such as the polynomials, are in fact holomorphic or may be extended to a holomorphic function. Note that when $U = \C$, i.e., $f$ is differentiable on the whole complex plane, we say that $f$ is an \emph{entire} function.

A basic result of complex analysis asserts that a function $f : U \to \C$ is holomorphic exactly when it is \emph{analytic}, i.e., its Taylor series expansion around any point $z_0 \in U$ converges to $f$ in some neighborhood of $z_0$. For this reason, holomorphic functions are typically referred to as analytic functions.
We consider analyticity in our work as there exist useful mathematical tools to check when functions are analytic, and even more useful tools for constraining functions that we have established to be analytic.

For our purposes, we define a \emph{closed contour} in a region $D \subseteq \C$ to be a continuously differentiable map $\gamma : [0, 1] \to D$ with $\gamma(0) = \gamma(1)$. Informally, we say that a region in the plane is \emph{simply connected} if it contains no holes. For instance, the disk $\{z \in \C : \abs{z} \leq 3\}$ is simply connected, whereas the annulus $\{z \in \C : \abs{z} \in [1, 3]\}$ is not.

A thorough review of the language of complex analysis with the precise definitions of the above (which are not necessary for the understanding of our application) is outside the scope of this work, so we recommend the textbook by \cite{Ahlfors53} for a more comprehensive background.

A useful property of analytic functions is that their closed contour integrals vanish in simply connected regions. The following theorem characterises this more formally.
\begin{thm}[Cauchy's Theorem]\label{thm:cauchy}
  Let $U$ be an open, simply connected subset of $\C$ and let $f : U \to \C$ be analytic. Then, for any closed contour $\gamma$ in $U$, we have $\oint\limits_\gamma f(z) \, dz = 0$.
\end{thm}

The converse is true, as well, and is a convenient technique for establishing analyticity.
\begin{thm}[Morera's Theorem]\label{thm:morera}
  Let $U \subseteq \C$ be open and let $f : U \to \C$ be continuous. Suppose that, for all simply connected $D \subseteq U$ and any closed contour $\gamma$ in $D$, we have $\oint\limits_\gamma f(z) \, dz = 0$. Then $f$ is analytic.
\end{thm}

Next, for functions $f_1,f_2: U \to \C$ and any $L \subseteq U$, we write $f_1|L \equiv f_2|L$, if for all $x \in L$, $f_1(x) = f_2(x)$. Additionally, we write $f_1 \equiv f_2$, if $f_1|U \equiv f_2|U$. Finally, we define the \emph{limit points} of a set.
\begin{defn}[Limit Point of a Set]\label{def:limit_point}
    Given a topological space $\cX$ and $S \subseteq \cX$, we say that $x \in \cX$ is a limit point of $S$, if for every neighbourhood $B \subseteq \cX$ of $x$ (with respect to the topology of $\cX$), there exists a point $y \in B$, such that $y \in S$ and $y \neq x$.
\end{defn}
In other words, a limit point $x$ of $S$ can be ``approximated by points in $S$.''
The main property of analytic functions that we exploit is the fact that any two analytic functions that agree locally must, in fact, agree globally, as we show next.
\begin{thm}[Identity Theorem]\label{thm:identity}
  Let $U \subseteq \C$ be open, and $f_1, f_2 : U \to \C$ be analytic. Suppose there is a set $L \subseteq U$ with a limit point in $U$, such that $f_1|L \equiv f_2|L$. Then $f_1 \equiv f_2$.
\end{thm}

\section{Background on Measure Theory}\label{app:measure_theory}

Recall that a \emph{measure space} is the combination of a set $\cX$ with a collection $\Sigma$ of subsets of $\mathcal{X}$, which are closed under complement and countable unions, as well as a function $\mu : \Sigma \to [0, \infty]$ satisfying $\mu(\emptyset) = 0$ and $\mu\p{\bigcup_{i = 1}^\infty A_i} = \sum_{i = 1}^\infty \mu(A_i)$ for disjoint $A_1, A_2, \dots \in \Sigma$. The subsets making up $\Sigma$ are called the \emph{measurable subsets} of $\cX$ and $\mu$ is called a \emph{measure} on $\cX$. We say that $\mathcal{X}$ is \emph{$\sigma$-finite} when it can be decomposed as $\mathcal{X} = \bigcup_{i = 1}^\infty A_i$ where $A_1, A_2, \dots \in \Sigma$ are all of finite measure $\mu(A_i) < \infty$. A function $f : \cX \to \C$ is said to be \emph{measurable} if $f^{-1}(U)$ is a measurable subset of $\cX$ for any open $U \subseteq \C$. In this case, we say that $f : \cX \to \R$ is \emph{$\mu$-integrable} if $\int_\cX \abs{f} \, d\mu$, the Lebesgue integral of $\abs{f}$ with respect to $\mu$, exists and is finite.

Now, in order to apply Morera's theorem, we will require some standard integral-limit interchange theorems. The first is the dominated convergence theorem, which asserts that pointwise convergence of a sequence of functions may be interchanged with integration, provided that the sequence is uniformly bounded by an integrable function.
\begin{thm}[Dominated Convergence Theorem]\label{thm:dominated_convergence}
  Let $\mathcal{X}$ be a measure space. Suppose that $(f_n)_{n \in \N}$ is a sequence of measurable functions $\mathcal{X} \to \C$ converging pointwise to some $f$, i.e., $f_n(x) \to f(x)$ for all $x \in \mathcal{X}$ as $n \to \infty$. Suppose further that there is some measurable $G : \mathcal{X} \to [0, \infty)$ such that $\int\limits_\cX G \, d\mu < \infty$ and $\abs{f_n(x)} \leq G(x)$ for all $x \in \cX$ and $n \in \N$. Then $f$ is integrable such that
  \begin{align*}
    \lim_{n \to \infty} \int\limits_\cX f_n \, d\mu = \int\limits_\cX f \, d\mu.
  \end{align*}
\end{thm}

Switching the order of integration is a very useful operation that is permitted under fairly general measure-theoretic conditions. We describe it as follows.
\begin{thm}[Fubini's Theorem]\label{thm:fubini}
  Let $\mathcal{X}$ and $\mathcal{Y}$ be $\sigma$-finite measure spaces and suppose that $f : \mathcal{X} \times \mathcal{Y} \to \R$ is measurable such that
  \begin{align*}
    \int\limits_{\mathcal{X}} \int\limits_{\mathcal{Y}} \abs{f(x, y)} \, dy \, dx < \infty.
  \end{align*}
  Then
  \begin{align*}
    \int\limits_{\mathcal{X}} \int\limits_{\mathcal{Y}} f(x, y) \, dy \, dx
      = \int\limits_{\mathcal{Y}} \int\limits_{\mathcal{X}} f(x, y) \, dx \, dy.
  \end{align*}
\end{thm}

\section{Impossibility Result for Concentrated DP}
\label{app:cdp}

We extend the impossibility result for unbiased estimation under pure DP in Section~\ref{sec:analyticity} to concentrated DP.
But first we briefly introduce concentrated DP.

Concentrated DP is a variant of DP that has particularly nice composition properties. It was introduced by \cite{DworkR16}, but we use a slightly different definition due to \cite{BunS16} (see also \cite{Steinke2022}).

\newcommand{\dr}[3]{\mathrm{D}_{#1}\left(#2\middle\|#3\right)}
\begin{defn}[Concentrated DP]
    A randomized algorithm $M : \mathcal{X}^n \to \mathcal{Y}$ satisfies $\rho$-zCDP if, for all neighboring inputs $x,x'\in\mathcal{X}$, \[ \forall t > 0 ~~~ \dr{t+1}{M(x)}{M(x')} \coloneqq \frac1t \log \left( \ex{Y \gets M(x)}{\left(\frac{\pr{M}{M(x)=Y}}{\pr{M}{M(x')=Y}}\right)^t} \right)\le (t+1)\rho.\]
\end{defn}
The quantity $\dr{t+1}{\cdot}{\cdot}$ is the R\'enyi divergence of order $t+1$. The above definition applies when the distributions of $M(x)$ and $M(x')$ are discrete; in the continuous case, we replace $\frac{\pr{M}{M(x)=Y}}{\pr{M}{M(x')=Y}}$ with the Radon-Nikodym derivative.

Concentrated DP is intermediate between pure DP and approximate DP. Specifically, $(\eps,0)$-DP implies $\frac12\eps^2$-zCDP and $\rho$-zCDP implies $(\rho + 2\sqrt{\rho \cdot \log(1/\delta)},\delta)$-DP for all $\delta>0$.
Concentrated DP captures most common DP algorithms, including Laplace and Gaussian noise addition and the exponential mechanism. Thus our impossibility result for concentrated DP provides a barrier against a wide range of techniques.

Concentrated DP also has strong group privacy properties (cf.~Lemma~\ref{lem:group_privacy}), which forms the basis of our impossibility result.
\begin{lem}[Group Privacy for Concentrated DP]
      Let $A: \cX^n \rightarrow \cY$ be $\rho$-zCDP. Then for any integer $k \in \{0, \dots, n\}$ and pairs of datasets $x,x' \in \cX^n$ differing in $k$ entries,
    \[\forall t>0 ~~~ \dr{t+1}{A(x)}{A(x')} \le (t+1) k^2 \rho.\]
\end{lem}

We can also use R\'enyi divergences to bound expectations.
\begin{lem}[{\cite[Lemma C.2]{BunS16}}]
    Let $X$ and $Y$ be random variables with $\dr{2}{X}{Y}$ and $\ex{}{Y^2}$ being finite. Then \[ |\ex{}{X}| \le \sqrt{\ex{}{Y^2} \cdot (\exp(\dr{2}{X}{Y})-1)}.\]
\end{lem}

Thus, we can give an analog of Proposition~\ref{prop:deprivatization_is_bounded}.
\begin{prop}[Concentrated DP Estimators Are Uniformly Bounded]
    \label{prop:cdp_bounded}
    Let $A : \cX^n \to \R$ be a randomized algorithm. If $A$ is $\rho$-zCDP, then for all $x,x^* \in \cX^n$,
    \[
      \abs{\ex{A}{A(x)}} \leq \sqrt{ (\exp(2 n^2 \rho)-1)\cdot\ex{A}{{A(x^*)}^2} }.
    \]
\end{prop}

This yields an extension of Theorem~\ref{thm:unbiased_ef_lb}.
\begin{thm}[Impossibility of Concentrated DP Unbiased Estimation for Exponential Families ]
    Let $U \subseteq \R$ be an interval of infinite length, let $\{P_\eta : \eta \in U\}$ be an exponential family, and let $I \subseteq U$ be any interval of positive length. Then, for any $\rho \geq 0$ and $n \geq 0$, there exists no $\rho$-zCDP algorithm $M:\R^n \rightarrow \R$ satisfying $\exin{X \gets P_\eta^n,M}{M(X)} = \eta$ and $\exin{X \gets P_\eta^n,M}{M(X)^2} < \infty$ for all $\eta \in I$.
\end{thm}

\section{Unknown Dataset Size}\label{app:n}

Our results consider the setting in which the dataset size $n$ is known and does not need to be kept private.
In particular, this means we assume neighboring datsets are the same size (see Definition~\ref{defn:dp}), so that we replace an element, rather than adding or removing an element.
The differential privacy literature often glosses over this distinction, as it is usually not important. But the distinction can matter, such as in privacy amplification by subsampling \cite[\S6.2]{Steinke2022}, and it could also matter when we are concerned about bias.

Our negative results automatically extend to the unknown dataset size setting, as this setting is only more general.
We remark that if we define differential privacy with respect to addition or removal of an element, then this implies differential privacy with respect to replacement. That is, if an algorithm is $(\eps,\delta)$-DP with respect to addition or removal, then it is $(2\eps,(1+\exp{\eps})\delta)$-DP with respect to replacement. This is because one replacement can be achieved by the combination of one removal and one addition and thus we can apply group privacy (Lemma~\ref{lem:group_privacy}). Hence, our negative results apply in the more general setting, up to a small loss in parameters.

If we define differential privacy with respect to replacement, then we can perform a removal by simply replacing the removed element with some default or null value. However, this default value could introduce bias. Thus, our positive results do not automatically extend to the unknown dataset size setting.
Nevertheless, we show that our positive results do extend to the more general setting where the dataset size is unknown and neighboring datasets are allowed to add, remove, or replace elements. Throughout this section, we use our ``$\sim$'' notation to denote neighboring datasets, but this time, two datasets could be neighboring if they differ on at most one entry via addition or removal or replacement of a point, as opposed to just replacement.

\subsection{Generic Reduction from Unknown to Known Dataset Size}

We first give a generic reduction that takes an algorithm $M_n : \mathcal{X}^n \to \mathcal{Y}$ that works for a known dataset size $n$ and produces an algorithm $M_* : \mathcal{X}^* \to \mathcal{Y}$ that works for unknown dataset size.\footnote{Here $\mathcal{X}^* = \bigcup_{n=0}^\infty \mathcal{X}^n$.} The reduction preserves the bias and accuracy properties, but the privacy parameters degrade, and we also sacrifice some of the dataset.

The natural approach to perform such a reduction is to first have $M_*$ privately obtain an estimate $N$ of the value of $n$, and then run $M_N$ on the dataset. 
This is essentially how our method works, but special care is needed to ensure that $0 \le N \le n$. And we also need to ensure that this does not introduce bias; e.g., even if $N$ is an unbiased estimate of $n$, computing $\tfrac{1}{N}\sum_i^n X_i$ could be a biased estimate of the mean.

\begin{prop}\label{prop:adp_n_reduction}
    For each $n \ge n_0$, let $M_n : \mathcal{X}^n \to \mathcal{Y}$ be an $(\eps,\delta)$-DP algorithm (with respect to replacement of one element of its input).
    Then there exists an algorithm $M_* : \mathcal{X}^* \to \mathcal{Y}$ that is $(2\varepsilon,2\delta)$-DP  (with respect to addition, removal, or replacement of one element of its input) and has the following form.
    For any input $x \in \mathcal{X}^n$ of size $n$, $M_*(x) = M_N(\widehat{x})$, where $N$ is random with $n \ge N \ge n-2\lceil\log(2/\delta)/\eps\rceil$ and $\widehat{x} \in \mathcal{X}^N$ is a uniformly random subset of $N$ elements of $x$; but, if $N<n_0$, then $M_*$ simply outputs $\bot$.
\end{prop}

Recall that, for distributions $P$ and $Q$ over a set $\cX$, we say that $P$ and $Q$ are $(\eps,\delta)$-indistinguishable (denoted by $P \edequiv Q$), if for all measurable $E \subseteq \cX$,
\begin{align*}
    \exp(-\eps)(Q(E) - \delta)
        \leq P(E)
        \leq \exp(\eps)Q(E) + \delta
\end{align*}

\begin{proof}
    The value $N(x)$ is an (under)estimate of the unknown size $n=|x|$ of the input dataset $x$. Specifically, \[N(x) \coloneqq |x| + \clip_{[-v,v]}(Z) - v,\] where $Z \in \mathbb{Z}$ is a random variable sampled according to a discrete Laplace distribution \citep{ghosh2009universally,CanonneKS20} with scale parameter $1/\varepsilon$ and $v = \lceil \log(2/\delta)/\eps \rceil$. {Namely, $\pr{}{Z=z} = \exp(-\eps |z|) \cdot \frac{\exp(\eps)-1}{\exp(\eps)+1}$ for all $z \in \mathbb{Z}$.}

    We show that $N : \mathcal{X}^* \to \mathbb{Z}$ is $(\eps,\delta)$-DP: 
    Let $x \in \mathcal{X}^n$ and $x' \in \mathcal{X}^{n'}$ be neighboring datasets.
    The dataset size has sensitivity 1 -- i.e. $|n-n'|\le1$. Hence $n+Z-v$ and $n'+Z-v$ are $\eps$-indistinguishable.
    By the tail properties of the discrete Laplace distribution \cite[Lemma 30]{CanonneKS20}, \[\pr{}{\clip_{[-v,v]}(Z) \ne Z} = \pr{}{|Z| \ge v+1} = \frac{2 \cdot \exp(-\eps v)}{\exp(\eps)+1} \le  \frac{\delta}{\exp(\eps)+1}.\]
    Thus, for all $S \subset \mathbb{Z}$, we have
    \begin{align*}
        \pr{}{N(x') \in S} &= \pr{}{n' + \clip_{[-v,v]}(Z) - v \in S}\\
        &\le \pr{}{n' + Z - v \in S} + \pr{}{\clip_{[-v,v]}(Z) \ne Z} \\
        &\le \exp(\eps) \cdot \pr{}{n + Z - v \in S} + \pr{}{\clip_{[-v,v]}(Z) \ne Z} \\
        &\le \exp(\eps) \cdot \left( \pr{}{n + \clip_{[-v,v]}(Z) - v \in S} + \pr{}{\clip_{[-v,v]}(Z) \ne Z} \right) + \pr{}{\clip_{[-v,v]}(Z) \ne Z}\\
        &= \exp(\eps) \cdot \pr{}{n + \clip_{[-v,v]}(Z) - v \in S} + (\exp(\eps)+1) \cdot \pr{}{\clip_{[-v,v]}(Z) \ne Z}\\
        &\le \exp(\eps) \cdot \pr{}{n + \clip_{[-v,v]}(Z) - v \in S} + \delta\\
        &= \exp(\eps) \cdot \pr{}{ N(x) \in S} + \delta.
    \end{align*}
    Thus $N(x)=n + \clip_{[-v,v]}(Z) - v$ and $N(x')=n' + \clip_{[-v,v]}(Z) - v$ are $(\eps,\delta)$-indistinguishable, as required.

    Now we define $M_* : \mathcal{X}^* \to \mathcal{Y}$ as the composition of $N : \mathcal{X}^* \to \mathbb{Z}$ and $A : \mathbb{Z} \times \mathcal{X}^* \to \mathcal{Y}$. That is $M_*(x) = A(N(x),x)$, where $A$ is defined as follows.
    On an input $N$ and $x \in \mathcal{X}^n$ of unknown size $n$, the algorithm $A(N,x)$ does the following. 
    If $N<n_0$, then $A$ simply outputs some fixed value $\bot$ indicating failure.
    First, $A$ samples a uniformly random $U \subset [\max\{n,N\}]$ with $|U|=N$. Second, it sets $\widehat{x}_U \in \mathcal{X}^N$ be the elements of $x$ indexed by $U$ -- i.e., if $U=\{i_1, i_2, \cdots, i_N\}$, then $\widehat{x}_U=(x_{i_1}, x_{i_2}, \cdots, x_{i_N})$. But, if $N>n$, then $A$ pads $x$ by adding copies of some default value $a \in \mathcal{X}$.\footnote{In $M_*$, we have $N \le n$ with probability $1$, so we never actually need to pad. But the composition theorem requires us to ensure $A$ is well-defined (and DP) for all $N$, so we must handle the $N>n$ case.} That is, we define $x_i=a$ for $i>n$. 
    Third and finally, it runs $M_N(\widehat{x}_U)$ and returns that output.


    Now we show that $A(N,\cdot) : \mathcal{X}^* \to \mathcal{Y}$ is $(\eps,\delta)$-DP for all $N \in \mathbb{Z}$. By composition, this means $M_*$ satisfies $(2\eps,2\delta)$-DP.

    Let $x \in \mathcal{X}^n$ and $x' \in \mathcal{X}^{n'}$ be neighboring datasets and let $N \in \mathbb{Z}$.
    If $N<n_0$, then $A(N,x)=A(N,x')=\bot$. So we can assume $N\ge n_0$.
    
    We can couple $\widehat{x}_U$ and $\widehat{x}'_{U'}$ so that they are always neighboring (in terms of replacement):
    Let $U \subset [\max\{n,N\}]$ and $U' \subset [\max\{n',N\}]$ denote the sets of indices of $x$ and $x'$ included in $\widehat{x}_U$ and $\widehat{x}'_{U'}$ respectively.
    If $x$ and $x'$ differ by the replacement of one element (so $n=n'$), then we can define the coupling by $U'=U$. Now suppose $x'$ is $x$ with one element removed (so $n'<n$). For notational simplicity, we assume, without loss of generality, that the last element of $x=(x_1, \cdots, x_n)$ is removed, so $x'=(x_1, \cdots, x_{n-1})$.  Then we define the coupling by first sampling $U$ from its marginal distribution and then setting $U'=U$ if $n \notin U$ and, if $n \in U$, then $U'=\{I\} \cup U \setminus \{n\}$, where $I \in [\max\{n,N\}] \setminus U$ is uniformly random. That is, we replace $x_n$ with $x_I$ in the coupling.

    

    Given this coupling and the $(\eps,\delta)$-DP guarantee of $M_N$, we have
    \[
        \pr{A}{A(N,x')\in S}
            = \ex{U,U'}{\pr{M_N}{M_N(\widehat{x}'_{U'}) \in S}}
            \le \ex{U,U'}{e^\eps \pr{M_N}{M_N(\widehat{x}_U) \in S}+\delta}
            = e^\eps \pr{A}{A(N,x)\in S} + \delta.
    \]

    Finally, $n-2v \le N \leq n$ with probability 1. This ensures that $M_*$ never requires padding and that, if $n \ge n_0 + 2v$, then the algorithm will not fail and output $\bot$.
\end{proof}

Proposition~\ref{prop:adp_n_reduction} does not directly guarantee anything about utility. But it is easy to see from the form of $M_*$ that it inherits the utility guarantees of $M_n$. Let $X$ consist of $n$ independent samples from some distribution. Then, for any realization of $N$, $\wh{X}$ consists of $N$ independent samples from the same distribution. Thus we can appeal to the utility properties of $M_N$ and note that $N \ge n-O(\log(1/\delta)/\eps)$.

In particular, if each $M_n$ is an unbiased estimator, then $M_*$ is also an unbiased estimator: \[\ex{X \gets P^n, M_*}{M_*(X)} = \ex{N, \widehat{X} \gets P^N, M_N}{M_N(\widehat{X})} = \ex{N}{\mu(P)} = \mu(P).\]
In the most natural case when the utility (e.g., MSE) of $M_m$ improves as $m$ increases, the utility of $M_*$ is at least as good as $M_{n - O(\log (1/\delta)/\varepsilon)}$.
Thus, we can achieve the same utility as in the known-dataset-size case at an additive cost of $O(\log(1/\delta)/\eps)$ in the sample size and a factor of two in the privacy parameters.



\subsection{Pure DP Algorithm for Bounded Distributions}

The reduction in the previous subsection inherently requires approximate DP because we use truncated noise.
If we use un-truncated Laplace noise, then there is a non-zero probability that $N>n$ or $N<n_0$ and hence the reduction fails to provide an unbiased estimator.
Thus, it is natural to wonder whether the unknown dataset size setting inherently requires approximate DP to ensure unbiasedness even for bounded distributions. (Recall that Theorem \ref{thm:main_packing} shows that approximate DP is necessary to ensure unbiasedness for Gaussian data.) We show that this is not the case.

Instead of a generic reduction like Proposition \ref{prop:adp_n_reduction}, we extend the noisy clipped mean approach to the unknown dataset size setting in a manner that does not add bias and remains pure DP.
This allows us to extend all our positive results for pure DP to the unknown dataset size setting.

A natural approach would be to estimate $n$ and $\sum_i^n x_i$ separately, and output the ratio of the estimates. However, the expectation of a ratio is, in general, not the same as the ratio of the expectations; thus, this natural approach can be biased.
Instead, we use the \emph{smooth sensitivity} framework of \cite{NissimRS07} to allow us to add unbiased noise directly to $\frac1n \sum_i x_i$ that scales with $\tfrac1n$ even though this value is unknown.
\cite[\S A]{NissimSV16} provide a similar algorithm for this task, but their approach does not satisfy pure DP because it is based on the propose-test-release framework.

We first provide some background on the smooth sensitivity framework, starting with the idea of \emph{local sensitivity}. Unlike global sensitivity, local sensitivity is pertinent to a fixed individual dataset, and is the worst-case change in the value of a function when we add or remove or replace one point in that particular dataset.
\begin{defn}[Local Sensitivity]\label{def:ls}
    Let $\cX$ be a data universe and $f:\cX^* \to \R$. Then for any $x \in \cX^*$, we define the \emph{local sensitivity} of $f$ at $x$ as:
    \[
        \ls_{f}(x) \coloneqq \max_{x' \in \cX^*, x' \sim x} \abs{f(x)-f(x')}.
    \]
\end{defn}

Next, we define a category of functions that upper bounds the local sensitivity of a function in a \emph{smooth} way. Note that global sensitivity does satisfy the following definition, but it is a very conservative upper bound.
\begin{defn}[$\beta$-Smooth Upper Bound]\label{def:bsub}
    Let $\cX$ be a data universe, $f:\cX^* \to \R$, and $\beta>0$. For any $x,x' \in \cX^*$, denote the Hamming distance between $x$ and $x'$ by $d(x,x')$ (i.e., $d(x,x')=1 \iff x \sim x'$). Then a function $S:\cX^* \to \R$ is a \emph{$\beta$-smooth upper bound} on $\ls_f$ if:
    \begin{align*}
        \forall x \in \cX^*,~~~ S(x) &\geq \ls_f(x),\\
        \forall x \sim x' \in \cX^*,~~~ S(x) &\leq \exp(\beta)\cdot S(x').
    \end{align*}
\end{defn}

We now define the \emph{smooth sensitivity} of a function, which is the smallest smooth upper bound.
\begin{defn}[Smooth Sensitivity]\label{def:ss}
    Let $\cX$ be a data universe, $f:\cX^* \to \R$, and $\eps \geq 0$. For any $x,x' \in \cX^*$, denote the Hamming distance between $x$ and $x'$ by $d(x,x')$. Then for any $x \in \cX^*$, we define the \emph{$\beta$-smooth sensitivity} of $f$ at $x$ as:
    \[
        \smooths_{f}^{\beta}(x) \coloneqq \sup_{x' \in \cX^*}\left\{\ls_{f}(x')\cdot \exp(-\beta d(x,x'))\right\}.
    \]
\end{defn}

Finally, state how these quantities relate to differential privacy.
\begin{lem}[Pure DP via Smooth Sensitivity \citep{NissimRS07,BunS19}]\label{lem:dp_smooth}
    Let $f:\cX^* \to \R$. Then we have the following.
    \begin{enumerate}
        \item $\smooths_{f}^{\beta}$ is a $\beta$-smooth upper bound on $\ls_{f}$.
        \item Let $S$ be a $\beta$-smooth upper bound on $\ls_{f}$ for some $\beta \ge 0$. Define $M:\cX^* \to \R$ by
        \[
            M(x) \coloneqq f(x) + \tau \cdot {S(x)}\cdot Z,
        \]
        where $Z$ is sampled from Student's $t$-distribution with $d$ degrees of freedom.\footnote{Student's $t$-distribution with $d$ degrees of freedom has Lebesgue density $\propto (d+z^2)^{-\frac{d+1}{2}}$ at $z \in \mathbb{R}$.} Then $M$ satisfies $\eps$-DP for \[\eps = \beta \cdot (d+1) + \frac1\tau \cdot \frac{d+1}{2\sqrt{d}}.\] Furthermore, $\ex{M}{(M(x)-f(x))^2} = \frac{d}{d-2} \cdot \tau^2 \cdot S(x)^2$.
    \end{enumerate}
\end{lem}

We use the above result to construct our algorithm for unbiased, pure DP mean estimation of distributions with bounded support.
\begin{thm}[Unbiased Pure DP Mean Estimation for Bounded Distributions]\label{thm:unbiased_puredp_ub}
    For every $\eps >0$ and $a \le 0 \le b$, there exists an algorithm $M : [a,b]^* \to \mathbb{R}$ that is $\eps$-DP (with respect to addition, removal, or replacement of an element) such that the following holds.
    Let $P$ be a distribution supported on $[a,b]$ with $\ex{X \gets P}{(X-\mu(P))^2} \le 1$, where $\mu(P) \coloneqq \ex{X \gets P}{X}$. For all $n \ge 1$, we have
    \begin{align*}
        \ex{X \gets P^n,M}{M(X)} &= \mu(P),\\
        \ex{X \gets P^n,M}{(M(X)-\mu(P))^2} &\leq \frac1n + \frac{9}{\eps^2} (b-a)^2 \max\left\{ \exp(-\eps(n-1)/6), \frac{1}{n^2} \right\}.
    \end{align*}
\end{thm}
\begin{proof}
    We construct a pure DP algorithm based on the smooth sensitivity framework (Lemma~\ref{lem:dp_smooth}). 
    Define $f:[a,b]^* \to [a,b]$ by
    \[
        f(x) = \begin{cases}
            0 & \text{if } n = |x| = 0\\
            \frac{1}{n}\sum\limits_{i=1}^{n}{x_i} & \text{if } n = |x| \ge 1
        \end{cases}.
    \]
    
    We first bound the local sensitivity of $f$: Consider neighboring inputs $x'\sim x$.
    If $n=|x|=0$ or $n'=|x'|=0$, then, since $f(x),f(x')\in[a,b]$, we have $|f(x)-f(x')| \le b-a = \frac{b-a}{\max\{n,n',1\}}$. 
    Now assume $n=|x| \ge 1$ and $n'=|x'| \ge 1$.
    If $n=n'$, then there exists some $i \in [n]$ such that $|f(x)-f(x')| = \left|\frac{x_i-x'_i}{n}\right| \le \frac{b-a}{n}$.
    By symmetry, we can assume that the remaining case is $n'=n+1 \ge 2$. Then there exists some $i \in [n']$ such that $x'$ is $x$ with $x'_i$ added. Hence 
    \begin{align*}
        \left|f(x)-f(x')\right| &= \left|\frac{1}{n} \sum_{j \in [n]} x_j -\frac{1}{n+1} \sum_{j \in [n'] } x_j' \right| \\
        &= \left|\frac{1}{n} \sum_{j \in [n]} x_j -\frac{1}{n+1} \left(x'_i + \sum_{j \in [n]} x_j\right)\right| \\
        &= \left|\left(\frac{1}{n} - \frac{1}{n+1} \right)\sum_{j \in [n]} x_j -\frac{x'_i}{n+1} \right| \\
        &= \frac{1}{n+1}\left|\frac{1}{n} \sum_{j \in [n]} x_j - x'_i \right| \\
        &\le \frac{1}{n+1}\frac{1}{n} \sum_{j \in [n]} \left| x_j - x'_i \right| \\
        &\le \frac{b-a}{n+1}.
    \end{align*}
    Combining these cases, we have $|f(x)-f(x')| \le \frac{b-a}{\max\{n,n',1\}}$. From this, we conclude $\ls_f(x) \le \frac{b-a}{\max\{n,1\}}$, where $n=|x|$.
    
    We now bound the smooth sensitivity of $f$. Note that for any $x,x' \in \cX^*$,
    \[
        d(x,x') \geq \abs{|x|-|x'|}.
    \]
    From Definition~\ref{def:ss}, we have the following.
    \begin{align}
        \smooths_{f}^{\beta}(x) &= \max_{x' \in \cX^*}\left\{\ls_{f}(x')\cdot \exp(-\beta d(x,x'))\right\}\nonumber\\
            &\leq \max_{x' \in \cX^*}\left\{\frac{b-a}{\max\{|x'|,1\}}\cdot \exp(-\beta \abs{|x|-|x'|})\right\}\label{eq:ss}
    \end{align}
    We consider three cases:
    \begin{itemize}
        \item If $|x'|=0$, then $\frac{b-a}{\max\{|x'|,1\}}\cdot \exp(-\beta \abs{|x|-|x'|}) \le \frac{b-a}{\max\{|x|,1\}}$.
        \item If $1 \le |x'|\leq|x|$, then $\abs{|x|-|x'|} = |x|-|x'|$. Define $g:(0,\infty) \to \R$ as $g(y) = \tfrac{1}{y}\cdot \exp(-\beta(n-y))$. Then we have $g'(y) = \tfrac{\exp(-\beta(n-y))}{y^2}\cdot\left(\beta y - 1\right)$, which implies that $g$ has a critical point at $y^* \coloneqq \tfrac{1}{\beta}$. Since $\lim_{y \to 0} g(y) = \lim_{y \to \infty} g(y) = \infty$, $y^*$ is a unique global minimizer. Hence, within the interval $[1,n]$, the maximum value of $g$ is realized at either $1$ or at $n$. Now $g(1) = \exp(-\beta (n-1))$ and $g(n) = \frac{1}{n}$. Thus, if $1 \le |x'|\le|x|=n$, then $\frac{b-a}{\max\{|x'|,1\}} \cdot \exp(-\beta ||x|-|x'||) \le (b-a) \cdot \max\{ \exp(-\beta (n-1)) , \frac{1}{n}\}$.
        \item If $|x'| > |x|$, then $\abs{|x|-|x'|} = |x'|-|x|$. Define $h:(0,\infty)\to \R$ as $h(y) = \tfrac{1}{y}\cdot \exp(-\beta(y-n))$. Then we have $h'(y) = \tfrac{-\exp(-\beta(y-n))}{y^2}\cdot\left(\beta y + 1\right) \le 0$ for all $y>0$. Thus, within the interval $[n+1,\infty)$, the maximuum value of $h$ is $h(n+1) = \tfrac{\exp(-\beta)}{n+1}$. So, if $|x'|>|x|=n$, then $\frac{b-a}{\max\{|x'|,1\}}\cdot \exp(-\beta \abs{|x|-|x'|}) \le \frac{b-a}{n+1}$.
    \end{itemize}
    Combining the above case analysis with Inequality~\ref{eq:ss}, we get
    \[
        \smooths_{f}^{\beta}(x) \leq (b-a) \cdot \max\left\{\exp(-\beta (|x|-1)), \frac{1}{\max\{|x|,1\}}\right\}.
    \]

    Finally, we define our mechanism $M:\cX^* \to \R$ as follows.
    \[
        M(x) \coloneqq f(x) + {\tau}\cdot\smooths_{f}^{\beta}(x)\cdot Z,
    \]
    where $Z$ is sampled from Student's $t$-distribution with $d=3$ degrees of freedom.

    From Lemma~\ref{lem:dp_smooth}, $M$ satisfies $\hat\eps$-DP with $\hat\eps = 4\beta + \frac{4}{2\sqrt{3}\tau}$. To obtain the desired $\eps$-DP guarantee, we set $\beta = \eps/12$ and $\tau = \sqrt{3}/\eps$.
    
    Conditioned on $|x|>0$, $M$ is unbiased because $\ex{}{Z} = 0$. The bound on the MSE of $M$ follows from Lemma~\ref{lem:dp_smooth} and the assumption that $P$ has variance at most 1: For $n \ge 1$,
    \begin{align*}
        \ex{X \gets P^n,M}{(M(X)-\mu(P))^2} &= \ex{X \gets P^n,M}{(f(X)-\mu(P))^2} + \tau^2 \cdot \ex{X \gets P^n,M}{\smooths_f^\beta(X)^2} \cdot \ex{}{Z^2}\\
        &\le \frac1n + \frac{3}{\eps^2} \cdot (b-a)^2 \max\left\{ \exp(-\beta(n-1)),\frac{1}{\max\{n,1\}}\right\}^2 \cdot 3 \\
        &= \frac1n + \frac{9}{\eps^2} (b-a)^2 \max\left\{ \exp(-\eps(n-1)/6), \frac{1}{n^2} \right\}.
    \end{align*}

    Our proof is now complete.
\end{proof}

\end{document}